\documentclass{article}
\usepackage[english]{babel}
\usepackage{graphicx}
\usepackage{slashed}
\usepackage{stmaryrd}
\usepackage{amsmath,amsfonts,amssymb}

\usepackage{ntheorem}
\usepackage{dsfont}
\usepackage{mathtools}
\usepackage{verbatim}
\usepackage{float}
\usepackage{accents}
\usepackage{mathrsfs}
\usepackage{amsmath}
\usepackage{amssymb}
\usepackage{enumerate}

\usepackage{hyperref}
\usepackage{tikz}
\usepackage{pgfplots}
\pgfplotsset{compat=1.14}
\usetikzlibrary{patterns}
\usetikzlibrary{positioning,arrows,arrows.meta}
\usepackage[left=2cm, right=2cm, bottom=2cm, top=2cm]{geometry}
\usepackage{titlesec}

\titleformat{\subsection}[runin]
  {\normalfont\bfseries}
  {\thesubsection}
  {1em}
  {}
  [.]

\titleformat{\subsubsection}[runin]
  {\normalfont\itshape}
  {\thesubsubsection}
  {1em}
  {}
  [.]

\usepackage{enumitem}
\theoremseparator{.} 
\newtheorem{Th}{Theorem}[section]
\newtheorem{Def}[Th]{Definition}
\newtheorem{Rq}[Th]{Remark}
\newtheorem{Pro}[Th]{Proposition}
\newtheorem{Cor}[Th]{Corollary}
\newtheorem{Lem}[Th]{Lemma}

\newcommand{\R}{\mathbb{R}}
\newcommand{\E}{\mathbb{E}}
\newcommand{\T}{\mathbf{T}}

\newcommand{\dr}{\mathrm{d}}

\newenvironment{proof}{\noindent\textit{Proof.~}}{\hfill$\square$\bigbreak} 

\title{Homeomorphic modified wave operators for the Vlasov–Poisson system}

\author{L\'eo Bigorgne\footnote{Institut de Recherche Math\'ematique de Rennes (IRMAR) - UMR 6625, CNRS, Universit\'e de Rennes, F-35000 Rennes, France.
{\em E-mail address:} {\tt leo.bigorgne@univ-rennes.fr}} }

\date{}
\begin{document}

\maketitle
    
\begin{abstract}
We prove modified scattering for small data solutions to the Vlasov–Poisson system in a functional framework where the initial data, scattering states, and asymptotic convergence are measured in the same topology. In addition, we show that the corresponding wave operators define homeomorphisms between the spaces of initial and scattering data, while enjoying a local Lipschitz continuity property in weaker norms. As a consequence, in the repulsive case, large spherically symmetric solutions are asymptotically stable. The proof relies in particular on the introduction of a suitable system of dynamic coordinates adapted to the asymptotic nonlinear flow.
\end{abstract}

  \tableofcontents

\section{Introduction}\label{SecIntro}
In this paper, we investigate the dynamics of solutions to the Vlasov--Poisson system, as well as the continuity properties of the associated scattering operators. The equations read
\begin{equation}\label{VP}
\begin{cases}
\partial_t f+v\cdot\nabla_xf -\lambda \nabla_x\phi [f] \cdot \nabla_vf=0,\\ 
\Delta_x \phi [f]=\rho [f],
\end{cases} \tag{VP}
\end{equation}
where
\begin{itemize}
\item the system describes the evolution, in phase space $\R^3_x \times \R^3_v$, of a large number of particles of mass $m=1$. The interaction between the particles is attractive when $\lambda=1$ and repulsive when $\lambda =-1$.
\item The distribution function $f \colon \R_t \times \R^3_x \times \R^3_v \to \R$ gives the density of particles in phase space at time $t$, position $x$ and velocity $v$. Although physically nonnegative, $f$ is allowed to take negative values in this paper. The spatial density $\rho[f]$ is given by
\[ \rho[f](t,x):=\int_{\R_v^3}f(t,x,v) \mathrm{d} v. \]
\item The function $\phi [f] \colon \R_t \times\R^3_x \to \R$ is the potential of the force field $-\nabla_x \phi [f]$, which is self-consistently generated by the particles through the Poisson equation.
\end{itemize}

The solutions to \eqref{VP} are global in time for a broad class of initial data \cite{Pfa,LionsPerthame}. Outside a few specific cases, such as the small data regime \cite{Bardos,HRV,Poisson,rVPWang,Duan,smallSchaeffer}, the spherically symmetric setting \cite{Horst,PankavichSpheri}, and perturbations thereof \cite{SchaefferVP}, little is known about the asymptotic behavior of solutions.

\subsection{Modified scattering dynamics for small data} It was shown by Ionescu-Pausader-Wang-Widmayer \cite{scattPoiss} that small data solutions to the Vlasov-Poisson system satisfy modified scattering. There exists asymptotic states $f_{\pm \infty} \colon \R^3_x \times \R^3_v \to \R$ such that
\begin{equation}\label{eq:introscat}
 \lim_{t \to \pm \infty} f \big(t,x+tv \pm \lambda \nabla_v \phi_{ \infty}[f_{\pm \infty}](v) \log \langle t \rangle , v \big) = f_{\pm \infty} (x,v), \qquad \qquad \Delta _v \phi_\infty \big[f_{\pm \infty} \big] = \int_{\R^3_x} f_{\pm \infty} (x, \cdot) \dr x. 
 \end{equation}
Contrary to the high-dimensional case \cite{Pankavichhigh}, because of long-range effects of the force field, the nonlinear characteristics cannot be asymptotically approximated by a linear one. Indeed, restricting to positive times, one can show, as predicted by the linearised system, that the force field is asymptotically self-similar
\[ \lim_{t \to + \infty} t^2 \nabla_x \phi [f](t,tv)= \nabla_v \phi_\infty[f_{+\infty}](v). \]
At this stage, however, one cannot yet assert that the limiting profile is $\nabla_v \phi_\infty[f_{+\infty}]$. One therefore expects the characteristics $t \mapsto (X,V)(t)$ of \eqref{VP} to satisfy $V(t) \to v_{\infty} \in \R^3_v$, as $t \to +\infty$, and
\[ V(t) \approx v_\infty + \frac{\lambda}{t} \nabla_v \phi_\infty \big[ f_{+\infty} \big](v_\infty), \qquad \qquad  X(t) \approx x_\infty + tv_\infty+\lambda \nabla_v \phi_\infty[f_{+\infty} ](v_\infty) \log (t), \]
which suggests that \eqref{eq:introscat} holds. We also refer to \cite{Choi,Panka} for related results, and to \cite{latetime,VolkerMartin} where the asymptotic expansions of $\nabla_x \phi [f]$ and $\rho [f]$ are shown to be polyhomogeneous.

In \cite{scattmap}, the converse problem to \eqref{eq:introscat} has been addressed. Any sufficiently regular scattering state $f_{+\infty}$ is attained by a unique solution $f$ to \eqref{VP} as in \eqref{eq:introscat}. If $f_{+\infty}$ is small enough, $f$ is in addition defined for all $t \geq 0$. By exploiting the time invariance of the system, this allows to define the small data forward and backward (modified) wave operators $\mathscr{W}_+$ and $\mathscr{W}_-$, as well as the (modified) scattering map $\mathscr{S}$,
\[ \mathscr{W}_- \colon f_{-\infty} \mapsto f_0, \qquad \qquad \qquad \mathscr{W}_+ \colon f_{+\infty} \mapsto f_0, \qquad \qquad \qquad \mathscr{S}= \mathscr{W}_+^{-1} \circ \mathscr{W}_- \colon f_{-\infty} \mapsto f_{+\infty}. \] 
The main two goals of this paper are the following.
\begin{enumerate}
\item Find topologies which make the scattering operators $\mathscr{W}_+$, $\mathscr{W}_-$, and $\mathscr{S}$, endomorphisms. This was listed as an open problem in \cite{scattmap}. Typycally, one assumes that the initial data $f_0$ belongs to a weighted Sobolev space, proves \eqref{eq:introscat} in a weaker topology, and obtains a scattering state $f_{+\infty}$ lying in a weighted Sobolev space with reduced regularity (both in terms of derivatives and weights). A similar loss occurs in the construction of $\mathscr{W}_+$. 

In \cite{scattmap}, the construction of $\mathscr{S}$ requires weighted $W^{2,\infty} (\R^3_x \times \R^3_v)$ regularity for $f_{-\infty}$ whereas the resulting state $f_{+\infty}$ is only controlled in the unweighted space $L^\infty(\R^3_x \times \R^3_v)$.
\item We investigate the regularity properties of the scattering operators. In particular, we prove that they are homeomorphisms in the natural norm of the problem, and locally Lipschitz in a weaker norm. This requires, in particular, to show the asymptotic stability of solutions satisfying modified scattering.

In the repulsive case $\lambda =-1$, this improves upon the result of Schaeffer \cite{SchaefferVP}, who showed that perturbations of spherically symmetric solutions exist globally and decay, but without establishing asymptotic stability.
\end{enumerate}

Let us also mention that modified scattering occurs in the following settings. 
\begin{itemize}
\item For perturbations of a point charge for the repulsive Vlasov-Poisson system \cite{PWY}.
\item For small data solutions to the two-dimensional Vlasov-Poisson system with the external unstable trapping potential $-|x|^2/2$ (see \cite{VPTPscat,HuangKwon}).
\item For the small data solutions to the Vlasov-Riesz system, where the Newtonian kernel $|x|^{-1}$ appearing in \eqref{VP} is replaced by $|x|^{-\alpha}$, with $0 < \alpha <1$. The case $1-\delta < \alpha <1$, for sufficiently small $\delta >0$, was addressed in \cite{HuangKwonRiesz}, and \cite{HongPankavich} extended the result to the range $1/2 < \alpha <1$.
\end{itemize} 
Finally, for the Vlasov-Maxwell system in the small data regime, the distribution function exhibits a modified scattering dynamic \cite{scat,BAP,Emile}, whereas the Maxwell field undergoes linear scattering, with a nontrivial electromagnetic memory effect \cite{scatmap}.

\subsection{Statement of the main results}

We start by introducing the energy norms considered throughout this paper. We fix a Lebesgue exponent $3/2<p<\infty$, and we define $p_\infty \coloneqq \lfloor 3/p \rfloor+1$, so that the Sobolev embedding $W^{p_\infty,p}(\R^3) \hookrightarrow L^\infty (\R^3)$ holds.

\begin{Def}\label{Defnorm}
Let $N \geq 0$, $n=(n_x,n_v) \in \mathbb{Z}^2$ and $g \colon \R^3_x \times \R^3_v \to \R$. Then, we define
\[ \E_N^n[g] \coloneqq \sum_{|\alpha_x|+|\alpha_v| \leq N} \bigg| \int_{\R^3_x} \int_{\R^3_v} \Big| \langle x \rangle^{n_x+|\alpha_x|} \langle v \rangle^{n_v+N-|\alpha_x|} \partial_x^{\alpha_x} \partial_v^{\alpha_v} g(x,v) \Big|^p \dr v \dr x \bigg|^{\frac{1}{p}}. \]
Moreover, 
\begin{itemize}
\item If $g$ does depend on the time variable, we will often write $\E_N^n[g](t)$ for $\E_N^n[ g(t,\cdot , \cdot) ]$,
\item We denote by $\mathcal{B}_N^n \subset W^{N,p}_{\mathrm{loc}} \big( \R^3_x \times \R^3_v \big)$ the Banach space composed by the functions $g$ such that $\E_N^n[g]<+\infty$.
\item If $n=(7,7)$, we will simply write $\E_N[g]$ for $\E_N^n[g]$ and $\mathcal{B}_N$ for $\mathcal{B}_N^n$.
\end{itemize}
\end{Def}
\begin{Rq}
We work in the range $1<p<\infty$ to have access to elliptic regularity (see also Remark \ref{Rqindexp} below). The condition $p>3/2$ implies that the first order derivatives of the Newtonian kernel $|x|^{-1}$ are $L^p$-integrable at infinity. The last assumption could be relaxed, at the expense of additional technicalities in the proof.
\end{Rq}

We now state our main result concerning the small data solutions to the Vlasov-Poisson system.

\begin{Th}\label{Th1}
Let $N \geq 2p_\infty$ and $f_{\mathrm{data}} \in \mathcal{B}_N$. There exists $\varepsilon_0 [N,p]>0$ such that, if $\E_N[f_{\mathrm{data}}] \leq \varepsilon_0$, then
\begin{enumerate}[ label = (\Alph*)]
\item \label{itman} \textnormal{Global existence.} The unique classical solution $f$ to \eqref{VP}, such that $f(0,\cdot , \cdot)=f_{\mathrm{data}}$, is global in time.
\item \label{itunan} \textnormal{Boundedness.} There exists an asymptotic state $f_\infty \in \mathcal{B}_N$ such that the function 
\begin{equation}\label{eq:defgTh}
 g(t,x,v) \coloneqq f \big( t, x+tv+\lambda \nabla_v \phi_\infty [f_\infty](v) \log \langle t \rangle , v \big) 
\end{equation}
 is uniformly bounded,
\[ \forall \, t \geq 0, \qquad \quad \E_N[g](t) \leq 2 \E_N[f_{\mathrm{data}}]. \]
\item \label{itdaou} \textnormal{Modified scattering in the strong norm.} We have $\E_N \big[g(t,\cdot , \cdot ) - f_\infty \big] \to 0$ as $t \to +\infty$, and
\[ \forall \, t \geq 0, \qquad \E_{N-1} \big[g(t,\cdot , \cdot ) - f_\infty \big] \lesssim  \frac{1}{\langle t \rangle^{\frac{1}{3}}}\E_N[f_{\mathrm{data}}]. \]
\end{enumerate}
Conversely, given $f_\infty  \in \mathcal{B}_N$, there exists $\varepsilon_1 [N,p]>0$ such that, if $\E_N[f_\infty] \leq \varepsilon_1$, then 
\begin{itemize}
\item There exists a unique global classical solution $f$ to \eqref{VP} such that $g(t,\cdot , \cdot) \to f_\infty$ as $t \to +\infty$. 
\item The properties \ref{itunan}--\ref{itdaou} hold with $f_{\mathrm{data}}$ replaced by the asymptotic data $f_\infty$.
\end{itemize}
\end{Th}
\begin{Rq}\label{RqTh1}
If, in addition, $\E_N^{(n_x,n_v)}[f_{\mathrm{data}}]<+\infty$ for some $n_x, \, n_v \geq 7$, then there exists $\mathbf{D}[N,n,p]>0$ such that
\[ \forall t\geq 0, \qquad \E_N^{(n_x,n_v)}[g](t) \leq \mathbf{D} \, \E_N^{(n_x,n_v)}[f_{\mathrm{data}}]. \]
Moreover, the convergences in \ref{itdaou} also hold in the stronger norms $\E_{N-1}^n[\cdot]$ and $\E_N^n[ \cdot ]$. We note that the convergence rate could be improved at the expense of requiring one additional $\langle v \rangle$-weight in the norm on the right-hand side.
\end{Rq}
We are then able to define the (modified) wave operators.

\begin{Def}\label{DefWplus}
Let $N \geq 2p_\infty$, and $n \in \mathbb{Z}^2$. Let $\mathcal{O}_\pm^{N,n} \subset \mathcal{B}_N^n$ be the largest set such that
\[
\begin{aligned}
\mathscr{W}_{\pm} : \mathcal{O}_\pm^{N,n} &\to \mathcal{B}_N^n \\
f_{\pm \infty} &\mapsto f(0, \cdot , \cdot),
\end{aligned}
\]
where $f$ is the unique solution to \eqref{VP} which satisfies modified scattering toward $f_{\pm \infty}$, is well-defined. The maps $\mathscr{W}_+$ and $\mathscr{W}_-$ are the modified wave operators associated to $+\infty$ and $-\infty$, respectively, for the Vlasov-Poisson system.
\end{Def}

We now state our main results concerning the regularity of the scattering operators.

\begin{Th}\label{Th2}
Let $N \geq 2p_\infty$ and $n=(n_x,n_v) \in \mathbb{Z}^2$ such that $n_x , \, n_v \geq 7$. Then,
\begin{itemize}
\item  the set $\mathcal{O}_\pm^{N,n}$ is non-empty and open,
\item $\mathscr{W}_\pm$ is a homeomorphism from $\mathcal{O}_\pm^{N,n}$ onto its image, with respect to the norm $ \E_N^n [\cdot ]$,
\item  $\mathscr{W}_\pm$ and $\mathscr{W}_\pm^{-1}$ are locally Lipschitz with respect to the norm $\E_{N-1}^n [ \cdot ] $,
\item in the case $\lambda=-1$, $\mathcal{O}_\pm^{N,n}$ contains a neighborhood of any spherically symmetric $f_\infty \in C^{N}_c (\R^3_x \times \R^3_v , \R_+)$.
\end{itemize}
\end{Th}

It allows to deduce the next result for the scattering map $\mathscr{S} \colon f_{-\infty} \mapsto f_{+\infty}$.

\begin{Cor}
Let $N \geq 2p_\infty$ and $n=(n_x,n_v) \in \mathbb{Z}^2$ such that $n_x , \, n_v\geq 7$. The scattering map $\mathscr{S} \coloneqq  \mathscr{W}_+^{-1} \circ \mathscr{W}_-$ is a homeomorphism from $\mathcal{O}_-^{N,n}$ onto its image, with respect to $ \E_N^n [\cdot ]$, and is locally Lipschitz with respect to $\E_{N-1}^n [ \cdot ] $.
\end{Cor}

\subsection{Strategy of the proof}\label{Subsecstrat}

\subsubsection{Small data solutions}

We first discuss the difficulties arising in the proof of the forward evolution part of Theorem \ref{Th1}. One way for proving \eqref{eq:introscat} consists in writing, using \eqref{VP}, that the function $g$ in \eqref{eq:defgTh} satisfies
\[ \partial_t g(t,x,v)= \frac{\lambda t}{\langle t \rangle^2} \Big(  \langle t \rangle^2\nabla_x \phi  \big(t, x+tv+\lambda \nabla_v \phi_\infty[f_\infty](v) \log \langle t \rangle  \big) - \nabla_v \phi_\infty[f_\infty](v) \Big) \cdot \nabla_x g + \frac{b(t,x,v)}{\langle t \rangle^2} \cdot \nabla_{x,v} g(t,x,v), \]
where $|b(t,x,v)| \lesssim \log \langle t +1 \rangle$, and $|\nabla_{x,v} g|$ is uniformly bounded in time. The limiting force field is usually identified before $f_\infty$, and satisfies $  t^2\nabla_x \phi (t,tv) \to  \nabla_v \phi_\infty [f_\infty](v)$ as $t \to +\infty$. The mean value theorem then yields
\begin{equation}\label{eq:scatintrointro}
 |\partial_t g (t,x,v)| \lesssim \frac{\log \langle t+1 \rangle}{\langle t \rangle^2}\Big( \langle x \rangle |\nabla_x g |(t,x,v)+|\nabla_v g |(t,x,v) \Big) . 
 \end{equation}
A similar estimate holds for the derivatives of $g$. For our purposes, namely proving modified scattering in a functional framework where the initial data, the scattering states, and the asymptotic convergence are measured in the same norm, we identify two difficulties.
\begin{enumerate}
\item  There is a $\langle x \rangle$ loss in \eqref{eq:scatintrointro}. Consequently, proving the convergence of $g$ requires assuming that $\nabla_x g$ is bounded in a stronger weighted norm.
\item There is a loss of regularity in \eqref{eq:scatintrointro}. Proving the convergence of $g$ requires sufficiently good control of $\nabla_x g$ and $\nabla_v g$. Consequently, this approach does not allow one to determine whether top order derivatives of $g$ converge. In fact, even proving boundedness for the top order $v$-derivatives of $g$ is not straightforward. These derivatives are usually estimated in terms of top order derivatives of $f(t,x+tv,v)$, which grow logarithmically.
\end{enumerate}

We deal with the first problem by propagating hierarchised weighted norms of $g$. For simplicity, we discuss this issue using the profile $f_\circ (t,x,v)=f(t,x+tv,v)$ instead of $g$ since the analysis is similar but technically simpler. At first glance, one may try to consider the $L^p$-norms of $\langle x \rangle^{n_x+|\alpha_x|}\langle v \rangle^{n_v} \partial_x^{\alpha_x} \partial_v^{\alpha_v} f_\circ$. However, denoting by $\widetilde{\T}_{\phi [f]}$ the Vlasov operator associated with \eqref{VP} in the coordinate system $(t,x+tv,v)$, we have
\begin{equation}\label{eq:fcircintro}
\widetilde{\T}_{\phi [f]} \big ( \partial_{x^i}f_\circ \big) =  - \lambda t \nabla_x \partial_{x^i} \phi [f] (t,x+tv)\cdot \nabla_x f_\circ + \lambda \nabla_x  \partial_{x^i} \phi [f] (t,x+tv) \cdot \nabla_v f_\circ .
\end{equation}
In order to propagate such norms, the second term on the right hand side forces us to bound
\begin{equation}\label{eq:spadecayintro}
\langle x \rangle \big| \nabla_x  \partial_{x^i} \phi  (t,x+tv)  \big| \lesssim  \frac{\langle x \rangle}{\langle t+|x+tv| \rangle^3} 
\end{equation}
by a quantity that decays sufficiently fast in time, uniformly in $(x,v)$. When $|x| \gg t$ and $|x+tv| \ll |x|$, this is unfortunately not possible. One way to circumvent this difficulty is to use the inequality $ |x| \leq |x+tv|+t|v|$, and to absorb the $|v|$-weight by working instead with $\E_N^n[f_\circ]$, namely the $L^p$-norms of $\langle x \rangle^{n_x+|\alpha_x|} \langle v \rangle^{n_v+N-|\alpha_x|} \partial_x^{\alpha_x} \partial_v^{\alpha_v} f_\circ$.

The second difficulty can be decomposed into two different problems. Assuming $\E_N^n[f](0) \leq \varepsilon$, 
\begin{itemize}
\item Can we prove that $\E_N^n[f_\infty] <+\infty$?
\item Can we show that $\E_N^n[g(t,\cdot , \cdot )-f_\infty ] \to 0$, as $t \to +\infty$?
\end{itemize}
Since \eqref{eq:scatintrointro} allows to prove modified scattering in a weaker norm, we would adress the first problem by Banach-Alaoglu theorem if we could prove that $t \mapsto \E_N^n[g](t)$ is uniformly bounded. In order to avoid the loss of regularity arising from \eqref{eq:scatintrointro}, one could be tempted to consider $\overline{\T}_{\phi [f]}$, the Vlasov operator in the coordinate system $(t,x+tv+\lambda \nabla_v \phi_\infty[f_\infty](v)\log \langle t \rangle,v)$, and to commute it $N$ times. However, one of the error term would carry the factor 
\begin{equation}\label{eq:forintrointrointro}
 \nabla_v^2 \partial_v^\gamma \phi_\infty [f_\infty](v), \qquad \qquad |\gamma| =N, 
 \end{equation}
that we do not control since we already face a loss of regularity in the convergence $t^2 \nabla_x \phi [f] (t,tv) \to \nabla_v \phi_\infty [f_\infty](v)$. Moreover, even if the elliptic regularity of the Poisson equation provides an estimate of
\begin{equation}\label{eq:forintroooo}
t^{3+|\gamma|}\nabla_x^2 \partial_x^\gamma \phi [f] (t,tv), \qquad \qquad |\gamma| =N, 
\end{equation}
the usual approaches allow for a small growth in time of the top order derivatives, preventing us to control \eqref{eq:forintrointrointro} by an application of the Banach-Alaoglu theorem. We deal with this issue by working in a \textit{dynamical system of coordinates} in which the distribution function admits a limit. In contrast, the previous system of coordinates, associated to $g$, is teleological as it depends on the asymptotic dynamics of the solution. We consider
\begin{equation}\label{eq:defdyna}
 h (t,x,v) \coloneqq f \big( t,x+tv+ \lambda \Phi_f (t,v) \log \langle t \rangle , v \big), \qquad \qquad \Phi_f (t,v) \coloneqq t^2 \nabla_x \phi [f](t,tv), 
 \end{equation}
so that we also have $h(t,x,v) \to f_\infty (x,v)$ as $t \to +\infty$. Moreover, although estimating top order derivatives of $h$ also requires controlling \eqref{eq:forintroooo}, the use of this dynamical coordinate system allows us to prove that both $\E_N^n[h]$ and
$\|t^{3+|\gamma|}\nabla_x^2 \partial_x^\gamma \phi [f] (t,\cdot)\|_{L^p_x}$, for $|\gamma| \leq N$, are uniformly bounded in time. This yields $\E_N^n[f_\infty] < +\infty$.
\begin{Rq}\label{Rqindexp}
Note that controlling the top order derivatives of $f_\circ$ through \eqref{eq:fcircintro} only requires controlling $\nabla_x \partial_x^\gamma \phi [f]$, for $|\gamma| \leq N$. Obtaining improved asymptotics by using a better coordinate system comes at the cost of controlling one additional derivative of the force field. The crucial use of elliptic estimates at top order then requires us to work with $L^p$-norms, with $1 < p < \infty$. We in fact work in the range $3/2 < p < \infty$ to avoid technical difficulties. As suggested by \eqref{eq:spadecayintro}, we will have to exploit the full spatial decay of the force field. However, for $p \leq 3/2$, $\nabla_x \phi [f](t,\cdot ) \notin L^p(\R^3_x)$.
\end{Rq}
\begin{Rq}
The $v$-derivatives of $h$, when expressed in terms of derivatives of $f$, are reminiscent of the modified vector fields introduced by Smulevici \cite{Poisson} (see also \cite{FJS3,dim3,Renato} for applications to other Vlasov systems). The dynamical coordinate system used in this article seems to capture the advantages of both the modified vector fields introduced in \cite{Poisson} and the coordinate system $(t,x+tv+\lambda \nabla_v \phi_\infty [f_\infty](v) \log \langle t \rangle,v)$, without suffering from their drawbacks.
\end{Rq}

The second problem, namely the convergence of $g$ with respect to the norm $\E_N^n[\cdot]$, requires a method avoiding the loss of regularity in \eqref{eq:scatintrointro}. To this end, we exploit that the transport dynamics singles out a preferred time direction along which no derivative loss occurs. More precisely, we use Duhamel's principle to obtain
\[ \partial_{x,v}^\alpha g (t,x,v) = \partial_{x,v}^\alpha g\big(0, \varphi_{0,t}(x,v) \big) +\int_{s=0}^t\Big[\overline{\T}_{\phi [f]}\big(\partial_{x,v}^\alpha g\big)\Big]\big(s,\varphi_{s,t}(x,v)\big) \dr s, \]
where $\varphi_{s,t}(x,v)$ denotes the flow associated with $\overline{\T}_{\phi[f]}$. Using the information obtained previously, one can then show that the second term on the right hand side converges in a weighted $L^p_{x,v}$-norm as $t \to +\infty$, at a quantitative rate. The convergence of the first term, without assuming additional regularity on the initial data, is proved by an approximation argument together with the convergence of the transport flow to a Lipschitz volume-preserving homeomorphism.

The backward problem can be handled similarly and is simpler in one respect. Indeed, knowing the scattering data allows one to control \eqref{eq:forintrointrointro} directly, and therefore to work immediately with the function $g$.

\subsubsection{Regularity of the scattering operators}

By time-reversibility of the Vlasov--Poisson system, it suffices to focus on $\mathscr{W}_+$. The first step of the proof consists in showing that the modified wave operator defines a map between two open sets. In other words, we need to prove that any solution $\mathbf{f}$ to \eqref{VP} satisfying modified scattering is asymptotically stable under perturbations of both the initial and scattering data. Let us now describe the additional difficulties arising in the analysis of the perturbed solution $f$, compared to the small data regime.
\begin{itemize}
\item The error terms in the commuted equations for the derivatives of $f_\circ - \mathbf{f}_\circ$ are merely linear in the small quantities. Since these error terms exhibit borderline decay, a Grönwall argument would yield a growth of the form $t^A$, with $A$ large, for the derivatives of $f_\circ -\mathbf{f}_\circ$, preventing us from closing the bootstrap argument. As a consequence, we are forced to use \textit{dynamical coordinate systems} in order to prove global existence, whereas for small data solutions they are only truly needed for the scattering analysis. We then study $h-\mathbf{h}$, where $h$ is given by \eqref{eq:defdyna} and
\[ \mathbf{h}(t,x,v) \coloneqq \mathbf{f} \big( t,x+tv+\lambda \Phi_{\mathbf{f}} (t,v) \log \langle t \rangle , v \big), \qquad \Phi_{\mathbf{f}}(t,v) \coloneqq t^2 \nabla_x \phi [\mathbf{f} ] (t,tv). \]

\item Derivatives of $h - \mathbf{h}$ are coupled to higher-order derivatives of $\mathbf{h}$. As a consequence, the top order derivatives of $h$ must be controlled without directly exploiting a small parameter. Moreover, deriving optimal decay estimates for $\nabla_x \phi [f]$ and its derivatives in terms of $h$ leads to a strongly nonlinear dependence on $\E_N[h]$.
\item Denoting by $\mathcal{E}_{N-1}[\Phi_f-\Phi_{\mathbf{f}}]$ a suitable $N-1$-th order norm of $\Phi_f-\Phi_{\mathbf{f}}$, we will show that
\[ \mathcal{E}_{N-1}[\Phi_f-\Phi_{\mathbf{f}}] (t) \lesssim \E_{N-1}[h-\mathbf{h}](t)+\langle t \rangle^{-1/2} \mathcal{E}_{N-1}[\Phi_f-\Phi_{\mathbf{f}}](t).\]
As a consequence, the potentially large implicit constant in $\lesssim$ makes the boundedness of $\Phi_f-\Phi_{\mathbf f}$ nontrivial, and the estimates must therefore be arranged so as to avoid any circular dependence.
\end{itemize}

To deal with these difficulties, we proceed as follows.

\begin{itemize}
\item We first appeal to Cauchy stability to deal with a compact interval of time $[0,T_0]$. If $\E_{N-1}[f_\circ - \mathbf{f}_\circ](0)$ is small enough, then $f$ exists on $[0,T_0]$, and $\E_{N-1}[f_\circ - \mathbf{f}_\circ](T_0) \leq C_{T_0} \E_{N-1}[f_\circ - \mathbf{f}_\circ](0)$.
\item It allows to show that $\E_{N-1}[h-\mathbf{h}](T_0)$ is small. Using $T_0^{-1/2}$ as a small parameter, we are then able to control $\mathcal{E}_{N-1}[\Phi_f-\Phi_{\mathbf{f}}] (t)$ for $t \geq T_0$. It allows us to prove that $f$ is a global solution, and that $\E_{N-1}[h-\mathbf{h}](t)$ remains small for all times.
\item Finally, our uniform bound on $\E_{N-1}[h]$ allows to control $\nabla_x \partial_x^\gamma \phi [f]$, for any $|\gamma| \leq N$, by elliptic estimates. Thus, in view of Remark \ref{Rqindexp}, we are able to show that the top order order derivatives of $f_\circ$ grow logarithmically. From here, one can estimate $\mathcal{E}_N[\Phi_f]$, and then derive boundedness for $\E_N[h]$. 
\end{itemize}

Since the above argument shows that both $\mathscr{W}_+$ and $\mathscr{W}_+^{-1}$ are locally Lipschitz with respect to the norm $\E_{N-1}[\cdot]$, it remains to prove that $\mathscr{W}_+$ is a homeomorphism with respect to the stronger norm $\E_N[\cdot ]$. To this end, we combine the previous ideas with standard arguments for proving local-in-time Cauchy stability for quasilinear PDEs (see, for instance, \cite[Section~4.5]{Bahouri}).

\begin{Rq}
In the perspective of proving analogues of Theorems \ref{Th1} and \ref{Th2} for the Vlasov--Maxwell system, one expects to face the following difficulties.
\begin{itemize}
\item The resonant phenomenon occurring between the electromagnetic field and particles with very high momentum. One way to deal with this issue is to lose powers of $(v,x-t\widehat{v})$, where $\widehat{v}\coloneqq v/\langle v\rangle$, in the estimates, which would be problematic in the present setting. We expect that this difficulty could be avoided by taking advantage of the null structure of the Lorentz force, described for instance in \cite[Lemma~4.1]{massless}.

\item The lack of elliptic regularity. Because of this, the approach developed in the present paper cannot be adapted by setting $
\Phi_f(t,v) \propto t^2[E+\widehat{v}\times B](t,x+t\widehat{v})$, since one would not be able to close the top order estimates. One expects to overcome this difficulty by exploiting the so-called Glassey--Strauss decomposition of the fields. However, the exact expression of $\Phi_f$ remains to be determined.

\item Finding a suitable functional framework for the electromagnetic field. In contrast with \eqref{VP}, the field satisfies an evolutionary equation and, in particular, exhibits linear scattering. One therefore has to identify a Banach space on which the scattering map acts as an endomorphism.
\end{itemize}
\end{Rq}

\subsection{Outline of the paper} The remainder of the article is organised as follows.
\begin{itemize}
\item \textbf{Section \ref{Secprep}}. We introduce the notation used throughout this paper and prove several decay estimates. In particular, we establish decay estimates for velocity averages of solutions to Vlasov equations, as well as for the associated force fields, adapted to our class of coordinate systems.
\item \textbf{Section \ref{SecTphi}}. We study the properties of a transport operator in a dynamical or teleological coordinate system, where the force field is assumed to be sufficiently regular. We prove boundedness and convergence results which, in the analysis of solutions to the Vlasov-Poisson system, correspond to modified scattering. We also show that the associated asymptotic Cauchy problem is well-posed. As an application, we study the long time dynamics of small data solutions to \eqref{VP} and prove Theorem \ref{Th1}.
\item \textbf{Section \ref{Sec3}}. We show that the modified wave operator $\mathscr{W}_+$ is a bijection between two open sets. More precisely, we prove that global solutions to \eqref{VP} that satisfy modified scattering are asymptotically stable under perturbations of both the initial data and the scattering data.
\item \textbf{Section \ref{Sec5}}. We conclude the proof of Theorem \ref{Th2} by proving that $\mathscr{W}_+$ is a homeomorphism between the spaces of scattering and initial data.
\item \textbf{Appendix A}. We prove an energy estimate used to control top order derivatives for large data solutions. We also show higher-order estimates for spherically symmetric solutions to the repulsive Vlasov-Poisson system.
\end{itemize}

\section{Preliminaries}\label{Secprep}

\subsection{Notations}

Let, throughout this paper,  $p > \frac{3}{2}$ and $p_\infty \coloneqq \lfloor 3/p \rfloor +1$, so that the Sobolev embedding $W^{p_\infty ,p} (\R^3) \hookrightarrow L^\infty (\R^3)$ holds. 

For a multi-index $\alpha \in \mathbb{N}^3$, we set $\partial_x^\alpha = \partial_{x^1}^{\alpha_1} \partial_{x^2}^{\alpha_2} \partial_{x^3}^{\alpha_3}$ and $\partial_v^\alpha = \partial_{v^1}^{\alpha_1} \partial_{v^2}^{\alpha_2} \partial_{v^3}^{\alpha_3}$. We further denote the length of $\alpha$ by $|\alpha| \coloneqq \alpha_1+\alpha_2 + \alpha_3$. For derivatives with respect to $(x,v) \in \mathbb{R}^3 \times \mathbb{R}^3$, we use the convention that a multi-index $\alpha \in \mathbb{N}^6$ is written as $\alpha = (\alpha_x,\alpha_v) \in \mathbb{N}^3 \times \mathbb{N}^3$, and we set
\[
\partial_{x,v}^\alpha \coloneqq \partial_x^{\alpha_x}\partial_v^{\alpha_v}, \qquad \qquad |\alpha| \coloneqq |\alpha_x|+|\alpha_v|.
\]
The Hessian of a function $\phi \colon \R^3 \to \R^3$ is denoted by $\nabla^2 \phi$. For $x \in \R$ or $x \in \R^3$, we denote the Japanese bracket of $x$ by $\langle x \rangle \coloneqq \sqrt{1+|x|^2}$.

When introducing a constant $C$, we will write $ C[a_1,\dots,a_i] > 0$ to indicate that it depends only on the parameters $a_1,\dots,a_i$. Finally, we will use the notation $A \lesssim B$ to mean that there exists a constant $C>0$ such that $A \leq C B$. In practice, the constant $C$ will only depend on $p$, $N$, the regularity order of the solutions, $n=(n_x,n_v)$, which measures the decay of the solutions at infinity, and $ \mathbf{\Lambda} $, where $\mathbf{\Lambda}$ quantifies the size of a reference solution. The dependence in $ \mathbf{\Lambda} $ can be removed if $\mathbf{\Lambda}$ is small.

\subsection{Elliptic estimates}

We recall here classical results. We start by a weighted Sobolev inequality.

\begin{Pro}\label{ProSob}
Let $a \geq 0$. For any sufficiently regular function $\psi \colon \R^3 \to \R$, we have, for all $(t,x) \in \R_+ \times \R^3$,
\[ \langle t+|x| \rangle^a |\psi|(x) \lesssim \frac{1}{\langle t+|x| \rangle^{\frac{3}{p}}} \sum_{|\kappa| \leq p_\infty} \Big\| \langle t+|y| \rangle^{a+|\kappa|} \partial_y^\kappa \psi (y)  \Big\|_{L^p(\R^3_y)} . \]
\end{Pro}
\begin{proof}
If $t+|x| \leq 1$, this follows from the standard Sobolev embedding $W^{p_\infty ,p} (\R^3) \hookrightarrow L^\infty (\R^3)$. Otherwise, we apply a local Sobolev inequality, in the ball $\{|y| \leq 1/2 \}$, to the function
\[ y \mapsto \big\langle t+|x+y(t+|x|)| \big\rangle^a \psi \big( x+y(t+|x|) \big). \]
Then, we observe that, on the domain of integration, $t+|x| \lesssim t+|x+y(t+|x|)|$, and we perform the change of variables $z(y)=x+y(t+|x|)$.
\end{proof}

Next, we recall elliptic estimates. 

\begin{Pro}\label{Proellip}
Assume that $\rho \colon \R^3 \to \R$ is sufficiently regular and consider the unique solution to $\Delta \phi = \rho$ such that $\phi (x) \to 0$ as $|x| \to + \infty$. Then, the following estimates holds,
\begin{align*}
\big\| \langle x \rangle^a \nabla_x^2 \phi (x) \big\|_{L^p (\R^3_x)} & \lesssim \big\| \langle x \rangle^a \rho (x) \big\|_{L^p (\R^3_x)}, \qquad \qquad \qquad -\frac{3}{p} < a <3\frac{p-1}{p}, \\
\big\| \langle \tau +|x| \rangle^2 \nabla_x \phi \big\|_{L^\infty (\R^3_x)} & \lesssim \big\| \langle \tau +|x| \rangle^3 \rho \big\|_{L^\infty (\R^3_x)},
\end{align*}
where the last one holds uniformly in $\tau \geq 0$.
\end{Pro}
\begin{proof}
For the first estimate, we use $1<p<\infty$ and that $x \mapsto \langle x \rangle^{ap}$ is in Muckenhoupt's class $A_p$ if $-3<ap<3(p-1)$. For the second one, we use that $\phi = |x|^{-1} \ast \rho$, and
\[ \int_{\R^3_y} \frac{\dr y}{| y - x |^2 \langle \tau+|y| \rangle^3} \lesssim \frac{1}{\langle x \rangle^2}, \qquad \qquad \int_{\R^3_y} \frac{\dr y}{| y - x |^2 \langle \tau+|y| \rangle^3} \lesssim \frac{1}{\tau^2}\int_{\R^3_z} \frac{\dr z}{|z|^2(\tau^{-1}+1+|z-x/\tau|)^3} \lesssim \frac{1}{\tau^2}, \]
where the last inequality holds for $|x| \leq \tau$.
\end{proof}

We will in fact use consequences of this result. We recall that $p>3/2$.

\begin{Cor}\label{Corellip}
Let $\rho \colon \R^3 \to \R$ be a sufficiently regular function and $\phi$ be the unique solution to $\Delta \phi = \rho$. Then,
for any multi-index $\gamma$ and all $\tau \geq 0$, we have
\begin{align*}
\big\| \langle \tau+ |x| \rangle^{1+|\gamma|} \nabla_x^2 \partial_x^\gamma \phi (x) \big\|_{L^p (\R^3_x)} & \lesssim \sum_{|\alpha| \leq |\gamma|} \big\| \langle \tau +|x| \rangle^{1+|\alpha|} \partial_x^\alpha \rho (x) \big\|_{L^p (\R^3_x)},  \\
\big\| \langle \tau +|x| \rangle^{2+|\gamma|} \nabla_x \partial_x^\gamma \phi (x) \big\|_{L^\infty (\R^3_x)} & \lesssim \sum_{|\alpha| \leq |\gamma|+p_\infty} \big\| \langle \tau +|x| \rangle^{3\frac{p-1}{p}+|\alpha|} \partial_x^\alpha \rho (x) \big\|_{L^p(\R^3_x)}.
\end{align*}
\end{Cor}
\begin{proof}
We will use the relations $\Delta \partial_x^{\gamma} \phi = \partial_x^\gamma \rho $,
\[ [\Delta, \Omega_{ij} ]=0, \quad [\Delta, S]=2\Delta, \qquad \quad \; \Omega_{ij} \coloneqq x^i \partial_{x^j}-x^j\partial_{x^i}, \quad S \coloneqq \sum_{1 \leq k \leq 3}   x^k \partial_{x^k}, \qquad \quad \; |x|^2 \partial_{x^k} = x^k S+\sum_{1 \leq i \leq 3} x^i \Omega_{ik}. \]
The last identity in particular implies that, for any smooth function $\psi$,
\begin{equation*}
\langle x \rangle^{|\gamma|} \big| \partial_x^\gamma \psi \big|(x) \lesssim  \sum_{|\beta| \leq |\gamma|} \big| Z^\beta \psi \big|(x) \lesssim  \sum_{|\kappa| \leq |\gamma|} \langle x \rangle^{|\kappa|} \big| \partial_x^\kappa \psi \big|(x) , 
\end{equation*}
where the sum is taken over all differential operators $Z^\beta$ obtained as compositions of at most $|\gamma|$ vector fields in $\{ \partial_{x^1},\partial_{x^2},\partial_{x^3}, \Omega_{12}, \Omega_{13}, \Omega_{23}, S \}$. As $[\partial_{x^k},S]=\partial_{x^k}$ and $[\partial_{x^k},\Omega_{ij}]=\delta_{k,i}\partial_{x^j}-\delta_{k,j} \partial_{x^i}$, Proposition \ref{Proellip}, applied with $a=1$, yields
\[ \big\| \langle x \rangle^{1+|\gamma|} \nabla^2_x \partial_x^{\gamma} \phi \big\|_{L^p(\R^3_x)} \lesssim \sum_{|\beta| \leq |\gamma|} \big\| \langle x \rangle \nabla^2_x Z^\beta \phi \big\|_{L^p(\R^3_x)} \lesssim   \sum_{|\beta| \leq |\gamma|} \big\| \langle x \rangle  Z^\beta \rho \big\|_{L^p(\R^3_x)} \lesssim  \sum_{|\kappa| \leq |\gamma|} \big\| \langle x \rangle^{1+|\kappa|}  \partial_x^\kappa \rho \big\|_{L^p(\R^3_x)}. \]
We conclude the $L^p$ estimate by using $\langle \tau +|x| \rangle \lesssim \tau + \langle x \rangle \lesssim \langle \tau +|x| \rangle$ and Proposition \ref{Proellip}, applied with $a=0$. Using in addition Proposition \ref{ProSob}, we similarly obtain the $L^\infty$ estimate.
\end{proof}

\subsection{Estimates for velocity averages}

We fix, for the remainder of this section, $N \geq 2p_\infty-1$ and a constant $\mathbf{\Lambda}  \geq 0$. We consider further
\begin{itemize}
\item A constant $a \geq 0$, and a function $(t,v) \mapsto \Phi (t,v) \in \R^3$, such that, for all $t \in \R_+$,
 \begin{equation}\label{EQ:AssumpPhi}
 \mathcal{E}_{N}[\Phi](t) \coloneqq  \big\| \langle v \rangle^2 \Phi (t,v) \big\|_{L^\infty (\R^3_v)} + \sum_{ |\gamma| \leq N}  \big\|\langle v \rangle^{1+|\gamma|} \nabla_v \partial_v^\gamma \Phi (t,v) \big\|_{ L^p ( \R^3_v)} \leq \mathbf{\Lambda} \log^a \langle t + 1 \rangle .  \tag{H$\Phi$}
  \end{equation}
\item The functions, sometimes vue as variables, 
\[ z(t,x,v) \coloneqq x+tv+\lambda \Phi (t,v) \log \langle t \rangle, \qquad \qquad \mathcal{Z}(t,y,v) \coloneqq y-tv-\lambda \Phi (t,v) \log \langle t \rangle. \]
 In what follows, we will often write $z$ for $z(t,x,v)$, and $\mathcal{Z}$ for $\mathcal{Z}(t,x,v)$.
\item A sufficiently regular function $f: \R_+ \times \R^3_x \times \R^3_v \to \R$, and $g(t,x,v) \coloneqq f \big(t,z(t,x,v),v \big)$.
\end{itemize} 
The estimates proved in this section will be applied to various choices of the function $\Phi$, including $\Phi_f$ and $\nabla_v \phi_\infty[f_\infty]$, introduced in Section \ref{Subsecstrat}. Although in most applications we will have $a=0$, the resulting logarithmic growth is harmless. Accordingly, we do not attempt to optimize the logarithmic losses in the estimates below. We start by relating the derivatives of $f$ to the ones of $g$,
\begin{equation}
 \begin{aligned}
\partial_t g (t,x,v) & = \big[ \partial_t f+v \cdot \nabla_x f +\lambda t \langle t \rangle^{-2} \Phi \cdot \nabla_x f+ \lambda \log \langle t \rangle \partial_t \Phi \cdot \nabla_x f \big] \big( t,z(t,x,v),v \big), \\
\partial_{x^i} g (t,x,v) & = \big[ \partial_{x^i} f \big] \big( t,z(t,x,v),v \big), \\
\partial_{v^i} g (t,z,v) & = \big[ t \partial_{x^i}f+\partial_{v^i} f + \lambda \log \langle t \rangle \partial_{v^i} \Phi  \cdot \nabla_x f \big]\big( t,z (t,x,v),v \big). 
\end{aligned}
\label{eq:derivftogv}
\end{equation}
 We will also make use of the function
\begin{equation}\label{eq:deffcirc}
 f_\circ (t,x,v) \coloneqq f(t,x+tv,v), \qquad \qquad \partial_x^{\alpha_x} \partial_v^{\alpha_v} f_\circ (t,x,v) = \big[ \partial_x^{\alpha_x} \big( t \partial_x+\partial_v \big)^{\alpha_v} f \big](t,x+tv,v), 
 \end{equation}
which corresponds to $f$ composed by the linear flow of $\partial_t+v \cdot \nabla_x$, and coincides with $g$ if $\Phi \equiv 0$. These relations allow to express a transport operator, defined in the coordinate system $(t,x,v)$, using the coordinates $(t,z,v)$.
\begin{Lem}\label{LemTbar}
Let $\phi : \R_+\times \R^3_x \to \R$ be a potential, and 
\[ \T_\phi \coloneqq \partial_t + v \cdot \nabla_x - \lambda \nabla_x \phi \cdot \nabla_v . \]
 Then, $\T_\phi (f)=F$ if and only if $\overline{\T}_\phi (g)=F(t,z(t,x,v),v)$, where
\[ \overline{\T}_{\phi } = \partial_t + \frac{\lambda t}{\langle t \rangle^2} \Big(  \langle t \rangle^2\nabla_x \phi  (t,z) - \Phi (t,v) \Big) \cdot \nabla_x  -\lambda \log \langle t \rangle \partial_t \Phi \cdot \nabla_x +   \log \langle t \rangle \nabla_v \Phi \cdot \nabla_x \phi (t,z ) \cdot \nabla_x -\lambda \nabla_x \phi (t,z) \cdot \nabla_v   . \]
\end{Lem}
\begin{Rq}
For all choices that will be made in this paper, $\nabla_v \Phi$ is a symmetric matrix-valued function. Consequently, there is no ambiguity in the notation $\nabla_v \Phi \cdot \nabla_x \phi (t,z ) \cdot \nabla_x$, which stands for $\sum_{i,j} \partial_{v^i} \Phi^j \partial_{x^i} \phi (t,z ) \partial_{x^j}$.
\end{Rq}
We now control the derivatives of $f_\circ$ by the ones of $g$. Recall from Definition \ref{Defnorm} the expression of $\E_N^n [\cdot]$.

\begin{Lem}\label{Corderivgtof}
Let $n =(n_x,n_v) \in \mathbb{Z}^2$. We have, for all $t \in \R_+$,
\begin{align*}
\E_N^n [ f_\circ](t) \lesssim \E_N^n[g](t)+ \log^{(a+1)(2N+n_x)-a} \hspace{-0.06em} \langle t+1 \rangle \, \mathcal{E}_{N}[\Phi](t)  \cdot \E_N^n[g](t). 
 \end{align*} 
\end{Lem}
\begin{Rq}\label{RqCorderivgtof}
Since $f_\circ$ and $g$ have symmetric roles, up to changing $\Phi$ in $-\Phi$, we also have 
\[\E_N^n [ g](t) \lesssim \log^{(a+1)(2N+n_x)} \hspace{-0.06em} \langle t+1 \rangle  \cdot \E_N^n[f_\circ](t). \]
\end{Rq}
\begin{proof}
Let $|\alpha| \leq N$. Iterating \eqref{eq:derivftogv}, we can write $\partial_x^{\alpha_x} \partial_v^{\alpha_v}g(t,x,v)-\partial_x^{\alpha_x}\partial_v^{\alpha_v} f_\circ(t,x+\lambda \Phi (t,v) \log \langle t \rangle,v)$ as a linear combination of the terms
\[ \log^{|\beta_x|}\langle t \rangle \,  \overline{\Phi} (t,v) \, \partial_x^{\alpha_x} \partial_{x}^{\beta_x} \partial_v^{\beta_v} g  (t,x,v), \qquad \overline{\Phi} (t,v) = \prod_{1 \leq i \leq |\beta_x|} \partial_{v^{j_i}}\partial_v^{\xi_i} \Phi^{k_i},  \qquad \sum_{1 \leq i\leq |\beta_x|} |\xi_i|+ |\beta| = |\alpha_v| , \quad |\beta_x| \geq 1 .\]
Then, if $|\beta_x| \geq 2$, we can bound all the factors in $\overline{\Phi}$ but one since $\|\langle v \rangle \nabla_v \partial_v^\gamma \Phi\|_{L^\infty_v} \lesssim \mathbf{\Lambda}\log^a \langle t+1 \rangle$ for any $|\gamma| \leq N-p_\infty$. This follows from \eqref{EQ:AssumpPhi}, the Sobolev embedding $W^{p_\infty,p}(\R^3_v) \hookrightarrow L^\infty (\R^3_v)$, and $N \geq 2p_\infty -1$. We obtain
\[ \big|\partial_x^{\alpha_x} \partial_v^{\alpha_v}g(t,x,v)-\partial_x^{\alpha_x}\partial_v^{\alpha_v} f_\circ \big( t,x+\lambda \Phi (t,v) \log \langle t \rangle,v \big) \big| \lesssim \log^b \langle t+1\rangle \sum_{|\xi|+|\beta| \leq |\alpha_v|} \frac{ \big|\langle v \rangle \nabla_v \partial_v^\xi \Phi (t,v) \big| }{\langle v \rangle^{|\beta_x|}}  \big| \partial_x^{\alpha_x} \partial_{x,v}^{\beta}  g \big|(t,x,v) , \]
where $|\beta| \geq 1$ in the sum, and $b \leq a(N-1)+N$. Next, we obtain, with $b' \coloneqq (a+1)(2N+n_x)-a$,
\begin{equation}\label{eq:auxiliaire}
\bigg\| \frac{\langle x \rangle^{n_x+|\alpha_x|} \langle v \rangle^{n_v+N}}{\langle v \rangle^{|\alpha_x|}} \Big( \partial_x^{\alpha_x} \partial_v^{\alpha_v}f_\circ (t,x,v)- \big[\partial_x^{\alpha_x}\partial_v^{\alpha_v} g \big] \big( t,x-\lambda \Phi (t,v) \log \langle t \rangle,v \big)  \Big) \bigg\|_{L^p_{x,v}} \lesssim \log^{b'} \! \langle t+1 \rangle \, \mathcal{E}_N [\Phi](t)  \E_N^n[g](t)
\end{equation}
from the following properties.
\begin{enumerate}[label = (\textbf{\alph*})]
\item \label{itemtraou111} We have $\big| \big\langle x - \lambda \Phi (t,v) \log \langle t \rangle \big\rangle -\langle x \rangle \big| \leq \| \Phi (t,v) \|_{L^\infty_v} \log \langle t \rangle \lesssim \log^{a+1} \langle t +1 \rangle$.
\item If $|\xi| \leq N-p_\infty$, we bound $\langle v \rangle | \nabla_v \partial_x^\xi \Phi| (t,v)$ by the Sobolev embedding $W^{p_\infty,p}(\R^3_v) \hookrightarrow L^\infty (\R^3_v)$.
\item If $|\alpha_x|+|\beta_x|+|\beta_v| \leq p_\infty -1$, we have, for almost all $v \in \R^3_v$,
\[ \int_{\R^3_x} \! \frac{\langle x \rangle^{n_x+|\alpha_x|} \langle v \rangle^{N+n_v}}{\langle v \rangle^{|\alpha_x|+|\beta_x|}}  \big| \partial_x^{\alpha_x} \partial_{x,v}^{\beta}  g (t,x,v)\big|^p \dr x \lesssim \bigg\| \frac{\langle x \rangle^{n_x+|\alpha_x|} \langle v \rangle^{N+n_v}}{\langle v \rangle^{|\alpha_x|+|\beta_x|}}  \partial_x^{\alpha_x} \partial_{x,v}^{\beta}  g (t,x,v) \bigg\|_{W^{p_\infty,p}_v L^p_x}^p \! \leq \big|\E_N^n[g](t)\big|^p , \]
by using the Sobolev embedding for Banach-valued functions $W^{p_\infty,p}(\R^3_v,L^p(\R^3_x)) \hookrightarrow L^\infty (\R^3_v,L^p(\R^3_x))$.
\end{enumerate}
We conclude the proof using \eqref{eq:auxiliaire}, \ref{itemtraou111} and by performing a change of variables.
\end{proof}

We now relate the velocity averages of $f$ to the ones of $g$.

\begin{Lem}\label{Lemrelintvgtof}
Let $|\alpha| \leq N$. Then, denoting $\mathcal{Z}(t,x,v)$ by $\mathcal{Z}$, we have, for all $(t,x) \in \R_+ \times \R^3$,
\begin{equation*}
\langle t+|x| \rangle^{|\alpha|} \big| \partial_{x}^{\alpha}  \rho [f] \big| (t,x)    \lesssim  \sum_{|\xi|+|\beta| \leq |\alpha|}   \int_{\R^3_v} \bigg( 1+\frac{\langle v \rangle^{1+|\xi|}}{\langle t \rangle^{\frac{3}{4}}}\big|\nabla_v \partial_v^{\xi} \Phi \big|(t,v)\bigg) \cdot \big\langle \mathcal{Z}\big\rangle^{|\beta_x|}\langle v \rangle^{|\beta_v|} \big| \partial_{x,v}^\beta g \big| \big(t, \mathcal{Z} ,v \big) \dr v .
 \end{equation*}
\end{Lem}
\begin{proof}
Note that the case $t+|x| \leq 1$ follows from $\partial_x^\alpha f(t,x,v)=\partial_x^\alpha g (t,\mathcal{Z},v)$, so that we assume $t+|x| \geq 1$. To obtain the spatial decay, we found convenient to use, in this proof, a weighted-derivatives approach. We define by $\mathbb{G}$ the set of the vector fields
\begin{equation}\label{eq:gainx}
  G_k \coloneqq t\partial_{x^k}, \qquad  \Omega_{ij} \coloneqq x^i \partial_{x^j}-x^j\partial_{x^i}, \qquad S \coloneqq \sum_{1 \leq \ell \leq 3}   x^\ell \partial_{x^\ell}, \qquad \qquad |x|^2 \partial_{x^k} = x^k S+\sum_{1 \leq i \leq 3} x^i \Omega_{ik} , 
  \end{equation}
where $1 \leq k \leq 3$ and $1 \leq i < j \leq 3$. We introduce further the set $\widehat{\mathbb{G}}$ of the $(x,v)$-weighted derivatives
\[   \widehat{G}_k \coloneqq t\partial_{x^k}+\partial_{v^k}, \qquad \widehat{\Omega}_{ij} \coloneqq  x^i \partial_{x^j}-x^j\partial_{x^i}+v^i\partial_{v^j}-v^j \partial_{v^i}, \qquad \widehat{S} \coloneqq \sum_{1 \leq \ell \leq 3}   x^\ell \partial_{x^\ell}+v^\ell \partial_{v^\ell}. \]
We consider an ordering on $\mathbb{G}$ and $\widehat{\mathbb{G}}$. Moreover, given a multi-index $\gamma \in \llbracket 1 , 7 \rrbracket^q$ of length $|\gamma|=q$, we define $Z^\gamma$ and $\widehat{Z}^\gamma$ as $Z^{\gamma_1} \dots Z^{\gamma_q}$ and $\widehat{Z}^\gamma = \widehat{Z}^{\gamma_1} \dots \widehat{Z}^{\gamma_q}$ respectively. Then, using first \eqref{eq:gainx}, and performing then integration by parts, we get
\begin{equation}\label{eq:decayrhox0}
(t+|x|)^{|\alpha|} \big| \partial_x^\alpha \rho [f] \big|(t,x)  \lesssim \sum_{|\gamma| \leq |\alpha|} \big| Z^\gamma \rho [f] \big|(t,x) \lesssim \sum_{|\kappa| \leq |\alpha|} \big|  \rho \big[ \widehat{Z}^\kappa f \big] \big|(t,x) . 
\end{equation}
It will be convenient to use that we can, without loss of generality, assume that $\widehat{Z}^\kappa$ is of the form
\[ \widehat{Z}^\kappa =  \widehat{G}^{\alpha_v} \widehat{\Omega}^{\alpha_\Omega} \widehat{S}^q, \qquad |\alpha_v|+|\alpha_\Omega|+q = |\kappa|,\]
where $G^{\alpha_v}=G_1^{\alpha_v^1}G_2^{\alpha_v^2}G_3^{\alpha_v^3}$, and $\Omega^{\alpha_\Omega}$ is a differential operator of order $|\alpha_\Omega|$ composed by the rotational vector fields $\Omega_{ij}$, $1 \leq i < j \leq 3$. We can assume that since the only non trivial commutation relations between the elements of $\widehat{\mathbb{G}}$ are
\begin{align*}
  [\widehat{G}_k, \widehat{\Omega}_{ij} ]=\delta_{i,k} \widehat{G}_j-\delta_{j,k} \widehat{G}_i, \quad \; \; [\widehat{\Omega}_{ij}, \widehat{\Omega}_{k \ell} ] =  \delta_{j,k} \widehat{\Omega}_{i \ell}+\delta_{j,\ell} \widehat{\Omega}_{ki}+\delta_{i,\ell} \widehat{\Omega}_{jk}+\delta_{i,k} \widehat{\Omega}_{\ell j}, \quad \; \;  [\widehat{G}_k , \widehat{S}]=\widehat{G}_k, \quad \; \; [\widehat{\Omega}_ {ij}, \widehat{S} ]= \widehat{\Omega}_{ij} .
\end{align*}
We fix such a differential operator $\widehat{Z}^\kappa=  \widehat{G}^{\alpha_v} \widehat{\Omega}^{\alpha_\Omega} \widehat{S}^q$, with $|\kappa| \leq |\alpha|$, and let us write $\widehat{Z}^\kappa f$ in terms of derivatives of $g$. We observe first 
\[  \widehat{G}_k f (t,x,v)=\partial_{v^k} f_\circ (t,x-tv,v), \qquad \widehat{\Omega}_{ij} f (t,x,v)= \big[\widehat{ \Omega}_{ij} f_\circ \big](t,x-tv,v)  , \qquad \widehat{S} f (t,x,v) = \big[\widehat{ S} f_\circ \big](t,x-tv,v), \]
so that
\begin{equation} 
\widehat{Z}^\kappa f(t,x,v) =  \big[ \partial_v^{\alpha_v} \widehat{\Omega}^{\alpha_\Omega} \widehat{S}^q  f_\circ \big] (t,x-tv,v)  . \label{EQ:forapres}
\end{equation}
Using that $f_\circ (t,x,v)=g(t,x-\lambda \Phi (t,v) \log \langle t \rangle,v)$, we can write $\widehat{Z}^\kappa f(t,x,v)- \partial_v^{\alpha_v} \widehat{\Omega}^{\alpha_\Omega} \widehat{S}^q g (t,\mathcal{Z},v)$ as a linear combination of the terms
\[\log^{|\beta_x|}\langle t \rangle \cdot \overline{\Phi} (t,v) \cdot \big[  \partial_x^{\beta_x} \partial_v^{\beta_v} \widehat{\Omega}^{\beta_\Omega} \widehat{S}^r g \big]\big( t, \mathcal{Z}(t,x,v) , v \big), \qquad \overline{\Phi} \coloneqq \prod_{1 \leq i \leq |\beta_x|}  \partial_v^{\gamma^i} \Omega^{\gamma_\Omega^i} S^{m^i} \Phi^{j_i} , \]
where we have $1 \leq j_i \leq 3$, $|\beta_x| \geq 1$, and
\[ |\beta_v|+|\beta_\Omega|+r+\sum_{1 \leq i \leq |\beta_x| } |\gamma^i|+|\gamma_\Omega^i|+m^i \leq |\alpha_v|+|\alpha_\Omega|+q , \qquad \qquad |\gamma^i|+|\gamma_\Omega^i|+m^i \geq 1 . \]
The key idea is that a new factor of $\overline{\Phi}$ is generated whenever $\widehat{G}_k$, $\widehat{\Omega}_{ij}$, or $\widehat{S}$ differentiate the $\Phi (t,v)$-dependent part of $\mathcal{Z}$. This creates a new $x$-derivative applied to $g$. We now have to deal with the logarithmical growth of such terms. Let us show
\begin{align}
\hspace{-3.2mm} \bigg| \rho \big[ \widehat{Z}^\kappa f \big] (t,x) - \! \int_{\R^3_v} \! \big[ \partial_v^{\alpha_v} \widehat{\Omega}^{\alpha_\Omega} \widehat{S}^q g \big] \big(t,\mathcal{Z},v \big)  \dr v \bigg| \! \lesssim \! \sum_{|\xi|+|\kappa'| \leq |\kappa|} \int_{\R^3_v}\frac{\langle v \rangle^{1+|\xi|}}{\langle t \rangle^{\frac{3}{4}}}\big|\nabla_v \partial_v^{\xi} \Phi  \big| \cdot \langle \mathcal{Z}\rangle^{|\kappa'_x|}\langle v \rangle^{|\kappa'_v|} \big| \partial_{x,v}^{\kappa'} g \big| \big(t, \mathcal{Z} ,v \big) \dr v ,  \label{eq:decayrhox}
 \end{align}
which, in view of the inequality \eqref{eq:23456789} stated below, will imply the result. For this, observe that, for a function $h$ and since $\mathcal{Z}=x-tv-\lambda \Phi (t,v) \log \langle t \rangle$,
\begin{align*}
  \partial_{x^i} h \big( t,\mathcal{Z},v \big)  = - \frac{1}{t} \Big( \partial_{v^i} \big[ h (t,\mathcal{Z},v ) \big]+\lambda \log \langle t \rangle \partial_{v^i} \Phi \cdot \nabla_x h \big(t, \mathcal{Z},v \big)-\partial_{v^i} h \big(t, \mathcal{Z},v \big) \Big) .
  \end{align*}
 We then exploit that $|\beta_x| \geq 1$, and we perform integration by parts. We then deal with the $\Phi$ coefficients by noting that, for $1 \leq |\gamma|+|\gamma_\Omega|+m \leq N+1$, we have
\[ \big|  \partial_v^{\gamma} \Omega^{\gamma_\Omega} S^{m} \Phi^{j} \big| (t,v) \lesssim \sum_{|\xi| \leq |\gamma|+|\gamma_\Omega|+m-1} \langle v \rangle^{1+|\xi|} \big| \nabla_v \partial_x^\xi \Phi \big| (t,v). \]
Moreover, by \eqref{EQ:AssumpPhi}, we have $\langle v \rangle^{1+|\xi|}|\nabla_v \partial_v^{\xi} \Phi| \lesssim \mathbf{\Lambda}\log^a \langle t+1 \rangle$ for all $|\xi| \leq N-p_\infty$. In particular, as $N \geq 2 p_\infty-1$, we can estimate pointwise all but one $\Phi$ factors. We finally obtain \eqref{eq:decayrhox} by using $\log^{N+1+Na} \langle t+1 \rangle \lesssim \langle t \rangle^{1/4}$, and by noticing
\begin{equation}\label{eq:23456789} 
\Big|\partial_x^{\beta_x} \partial_v^{\beta_v} \widehat{\Omega}^{\beta_\Omega} \widehat{S}^r g \Big| (t,x,v) \lesssim \sum_{|\kappa| \leq |\kappa_x|+|\beta_v|+|\beta_\Omega|+r} \langle x \rangle^{|\kappa_x|} \langle v \rangle^{|\kappa_v|} \big| \partial_x^{\kappa_x} \partial_v^{\kappa_v} g \big| (t,x,v) . 
\end{equation}
\end{proof}

For proving decay estimates, we will make use, throughout this paper, of the following change of variables.

\begin{Lem}\label{Lemcdv}
Let $x \in \R^3_x$. Then, there exists $C[p]>0$ such that, if $t > C \mathbf{\Lambda} \log^{a+1} \langle \mathbf{\Lambda} \rangle$, the map
\[v \mapsto x+tv+ \lambda \Phi(t,v) \log \langle t \rangle \] 
is a $C^1$ diffeomorphism of $\R^3$ onto $\R^3$, whose jacobian determinant is bounded between $t^3/2$ and $2t^3$.
\end{Lem} 
\begin{proof}
This follows from \eqref{EQ:AssumpPhi}, which implies that for $t$ large enough, $v \mapsto v+ \lambda \Phi(t,v) \frac{\log \langle t \rangle}{t}$ is a $C^1$ diffeomorphism of $\R^3$ onto $\R^3$, whose jacobian determinant is bounded between $1/2$ and and $2$.
\end{proof}

We prove a last result, which, together with Lemma \ref{Lemrelintvgtof}, will directly implies decay for $\rho [f]$.

\begin{Lem}\label{Lem236458}
Let $|\xi|+|\beta| \leq N$. Then, for all $t \geq 0$, we have
\begin{align*}
\bigg\| \langle t+|x| \rangle^{3 \frac{p-1}{p}} \int_{\R^3_v} \big\langle \mathcal{Z}\big\rangle^{|\beta_x|}\langle v \rangle^{|\beta_v|} \big| \partial_{x,v}^\beta g \big| \big(t, \mathcal{Z} ,v \big) \dr v \bigg\|_{L^p(\R^3_x)} & \lesssim  \E_N^{(6,6)}[g](t), \\
\bigg\| \langle t+|x| \rangle^{3 \frac{p-1}{p}} \int_{\R^3_v} \frac{\langle v \rangle^{1+|\xi|}}{\langle t \rangle^{\frac{3}{4}}}\big|\nabla_v \partial_v^{\xi} \Phi (t,v) \big| \cdot \big\langle \mathcal{Z}\big\rangle^{|\beta_x|}\langle v \rangle^{|\beta_v|} \big| \partial_{x,v}^\beta g \big| \big(t, \mathcal{Z} ,v \big) \dr v \bigg\|_{L^p(\R^3_x)} & \lesssim  \frac{\mathcal{E}_N [\Phi ](t)}{\langle t \rangle^{\frac{3}{4}}} \cdot \E_N^{(6,6)}[g](t).
\end{align*}
\end{Lem}
\begin{proof}
Note first that if $|x| \geq 2  \mathbf{\Lambda}  \log^{a+1} \langle t+1 \rangle$, we have $|x| \leq 2 \langle t \rangle \langle \mathcal{Z} \rangle \langle v \rangle$ since
\[ | x | \leq \big|\mathcal{Z}(t,x,v) \big|+t|v|+\|\Phi \|_{L^\infty_{t,v}} \log \langle t \rangle \leq \big|\mathcal{Z}(t,x,v) \big|+t|v|+|x|/2 . \]
We focus now on the velocity average of $h(t,\mathcal{Z}(t,x,v),v)$, for a sufficiently regular distribution function $h$. We have, using the previous estimate, together with the Hölder inequality,
\begin{equation*}
\bigg\| \langle t+|x| \rangle^{3 \frac{p-1}{p}} \! \int_{\R^3_v} | h | \big(t, \mathcal{Z}  ,v\big)  \dr v  \bigg\|_{L^p_x }  \leq \langle t \rangle^{3 \frac{p-1}{p}} \bigg\| \int_{\R^3_v}  \frac{ \dr v }{\langle \mathcal{Z} (t,x,v) \rangle^{\frac{3p}{p-1}}\langle v \rangle^{\frac{3p}{p-1}} } \bigg\|_{L^\infty_x }^{\frac{p-1}{p}} \Big\|  \langle \mathcal{Z} \rangle^{6- \frac{3}{p}} \langle v \rangle^{6- \frac{3}{p}}  h \big(t, \mathcal{Z} ,v \big) \Big\|_{L^p_{x,v}} . 
\end{equation*}
According to Lemma \ref{Lemcdv}, if $t$ is large enough, we can perform the change of variables $y(v)= \mathcal{Z}(t,x,v)$ and obtain
\[ \int_{\R^3_v}  \frac{ \dr v }{\langle \mathcal{Z} (t,x,v) \rangle^{3 \frac{p-1}{p}} } \lesssim \frac{2}{t^3} \int_{\R^3_y} \frac{\dr y}{\langle y \rangle^{3 \frac{p-1}{p}}} \lesssim \frac{1}{t^3}, \qquad \qquad \qquad \int_{\R^3_v}  \frac{ \dr v }{\langle v \rangle^{3 \frac{p-1}{p}} } < +\infty  . \]
We then deduce that
\begin{equation}\label{eq:111111} 
\bigg\| \langle t+|x| \rangle^{3 \frac{p-1}{p}} \! \int_{\R^3_v} | h | \big(t, \mathcal{Z} (t,x,v),v \big)  \dr v  \bigg\|_{L^p (\R^3_x \times \R^3_v) } \leq \Big\| \langle x \rangle^{6- \frac{3}{p}} \langle v \rangle^{6- \frac{3}{p}} h ( t,x,v)  \Big\|_{L^p (\R^3_x \times \R^3_v) } ,
\end{equation}
where we performed the change of variables $x'(x)=\mathcal{Z}(t,x,v)$. Then, we apply \eqref{eq:111111} to
\[ h(t,x,v) = \langle x \rangle^{ |\beta_x|} \langle v \rangle^{|\beta_v|} \partial_{x,v}^\beta g, \qquad  h(t,x,v) = \langle v \rangle^{1+|\xi|} \big|\nabla_v \partial^\xi_v \Phi (t,v)\big| \langle x \rangle^{ |\beta_x|} \langle v \rangle^{|\beta_v|} \partial_{x,v}^\beta g ,\]
 for $|\xi|+|\beta|\leq |\alpha|$, and we recall from Definition \ref{Defnorm} the expression of $\E_N^{(6,6)}[\cdot ]$. The latter quantities require an additional step to be bounded. If $|\xi| \leq N-p_\infty$, we use a Sobolev embedding to get $\langle v \rangle^{1+|\xi|}|\nabla_v \partial_v^\xi \Phi| \lesssim \mathcal{E}_N[\Phi]$. Otherwise $|\beta| \leq p_\infty-1$ and by the Sobolev embedding $W^{p_\infty,p}(\R^3_v,L^p(\R^3_x)) \hookrightarrow L^\infty ( \R^3_v,L^p(\R^3_x))$, we have
\begin{align*}
 \big\| \langle x \rangle^{6- \frac{3}{p}} \langle v \rangle^{6- \frac{3}{p}}  h (t,x,v) \big\|_{L^p_{x,v} } & \lesssim \big\|\langle x \rangle^{6- \frac{1}{p} + |\beta_x|} \langle v \rangle^{6- \frac{1}{p}+|\beta_v|} \partial_{x,v}^\beta g \big\|_{ W^{p_\infty,p}_v L^p_x} \cdot \big\| \langle v \rangle^{1+|\xi|} \nabla_v \partial_v^\xi \Phi \big\|_{L^p_x} . 
 \end{align*}
\end{proof}

Combining Lemmata \ref{Lemrelintvgtof} and \ref{Lem236458} with \eqref{EQ:AssumpPhi} yields the following result.

\begin{Pro}\label{Prodecayintvbis}
For any $|\alpha| \leq N$ and all $t \geq 0$, we have
\[  \Big\| \langle t+|x| \rangle^{3\frac{p-1}{p}+|\alpha|} \partial_x^\alpha \rho [f] (t,x) \Big\|_{L^p(\R^3_x)} \lesssim  \E_N^{(6,6)}[g](t) . \]
\end{Pro}

\subsection{Estimates for the force field}

We start by studying the asymptotic force field, defined as the solution to $\Delta \phi_\infty [f_\infty] = \int_x f_\infty (x, \cdot ) \dr x$, for a sufficiently regular function $f_\infty \colon \R^3_x \times \R^3_v \to \R$.

\begin{Pro}\label{Prophiinfty}
We have, for any $|\alpha| \leq N$,
\[ \big\| \langle v \rangle^{1+|\alpha|} \nabla^2_v \partial_v^{\alpha} \phi_\infty [f_\infty] \big\|_{L^p (\R^3_v)}  \lesssim \sup_{|\kappa| \leq |\alpha|} \big\| \langle x \rangle^3 \langle v \rangle^{1+|\kappa|} \partial_v^\kappa f_\infty \big\|_{L^p (\R^3_x \times \R^3_v)} . \]
For any $|\gamma| \leq N-p_\infty$, we have 
\[\langle v \rangle^{2+|\gamma|}\big|\nabla_v \partial_v^\gamma \phi_\infty [ f_\infty]\big|(v) \lesssim  \sup_{|\kappa| \leq |\gamma|+p_\infty} \big\| \langle x \rangle^3 \langle v \rangle^{3+|\kappa|} \partial_v^\kappa f_\infty \big\|_{L^p (\R^3_x \times \R^3_v)}.\]
\end{Pro}
\begin{proof}
This follows from the elliptic estimates of Corollary \ref{Corellip}, applied with $\tau=0$, and the Hölder inequality.
\end{proof}

We now control the force field $\nabla_x \phi [f]$, satisfying $\Delta \phi [f]=\rho [f]$. We note $\partial_x^\alpha \phi [f]=\phi [\partial_x^\alpha f]$.

\begin{Pro}\label{Proforcefield}
Let $|\alpha| \leq N$. For all $t \geq 0$, we have
\[ \langle t \rangle^{\frac{2p-3}{p}} \Big\| \langle t+|x| \rangle^{1+|\alpha|} \nabla_x^2  \phi \big[ \partial_x^\alpha f \big] \Big\|_{L^p(\R^3_x)} \lesssim \E_N[g](t)  . \]
For any $|\gamma| \leq N-p_\infty$, we have, for all $(t,x) \in \R_+ \times \R^3_x$,
\[ \langle t+|x| \rangle^{2+|\gamma|}  \big| \nabla_x  \phi \big[ \partial_x^{\gamma} f \big] \big|(t,x) \lesssim  \E_N[g](t). \]
\end{Pro}
\begin{proof}
For both inequalities, we apply the elliptic estimates of Corollary \ref{Corellip} together with Proposition \ref{Prodecayintvbis}.
\end{proof}

\subsubsection{Improved estimates for the force field}

Let $\phi \colon \R^3 \to \R$ be a sufficiently regular potential and recall from Lemma \ref{LemTbar} the definition of $\T_\phi$. We assume now that $f$ is a solution to $\T_\phi (f)=0$ and let us show that improved estimates for $\phi [f]$ can be obtained. Compared with Proposition \ref{Proforcefield}, we will control one derivative less.

\begin{Lem}\label{LemComfcirc}
Recall that $f_\circ (t,x,v) \coloneqq f(t,x+tv,v)$. Then, almost everywhere and for any $|\alpha| \leq N-1$, 
\[ \int_{\R^3_y} \big[\partial_t \partial_v^\alpha f_\circ \big] \Big( t,y,v-\frac{y}{t} \Big) \dr y = 0 . \]
\end{Lem}
\begin{proof}
We use that $\T_\phi (f)=0$, so that $\partial_t f+ v\cdot \nabla_x f= \lambda \nabla_x \phi \cdot \nabla_v f$, and
\[ \partial_t f_\circ = \big[ \partial_t+v\cdot \nabla_x f\big](t,x+tv,v), \qquad \partial_{x^i} f_\circ = \partial_{x^i} f (t,x+tv,v), \qquad \partial_{v^i} f_\circ = \big[t\partial_{x^i}f+\partial_{v^i}f \big](t,x+tv,v). \]
By Leibniz formula, we then deduce
\[ \partial_t \partial_v^\alpha f_\circ = \lambda \sum_{\gamma + \beta = \alpha}{\alpha \choose \beta}  t^{|\gamma|}\nabla_x \partial_x^\gamma \phi (t,x+tv) \cdot \nabla_v \partial_v^\beta f_\circ - {\alpha \choose \beta} t^{|\gamma|+1}\nabla_x \partial_x^\gamma \phi (t,x+tv) \cdot \nabla_x  \partial_v^\beta f_\circ  . \]
It remains to evaluate this relation at $(t,y,v-y/t)$ and to note that 
\[ t\nabla_y \Big[ \big[\partial_v^\beta f_\circ \big] \Big( t,y,v - \frac{y}{t} \Big) \Big]=t\big[\nabla_x\partial_v^\beta f_\circ \big] \Big( t,y,v - \frac{y}{t} \Big)-\big[\nabla_v\partial_v^\beta f_\circ \big] \Big( t,y,v - \frac{y}{t} \Big).\]
\end{proof}

In the next two results, we will make use of the relation
\begin{equation}\label{eq:cdv}
 t^3 \rho [h] (t,x) =t^3 \int_{\R^3_v} h (t,x,v)  \dr v = \int_{\R^3_y} h_\circ \Big( t,y,\frac{x-y}{t} \Big) \dr y, 
\end{equation}
which is obtained by performing the change of variables $y=x-tv$.

\begin{Lem}\label{LemdtPhi}
Let $|\gamma| \leq N-1$. Then, for all $t \geq 1$, we have
\[ \Big\| \langle v \rangle \partial_t \Big[  t^{3+|\gamma|} \nabla^2_x \phi  \big[ \partial_x^\gamma f \big](t,tv) \Big] \Big\|_{L^p(\R^3_v)} \lesssim  \frac{\log^{b}\langle t \rangle}{ t^2}\E_{N}[g](t) , \]
where $b \coloneqq (a+1)(2N+7)$. For any $|\kappa| \leq N-p_\infty$ and all $v \in \R^3_v$, we have
\[ \langle v \rangle \Big|   \partial_t \Big[  t^{2+|\kappa|} \nabla_x \phi  \big[ \partial_x^\kappa f \big](t,tv) \Big] \Big| \lesssim  \frac{\log^{b}\langle t \rangle}{t ^2}\E_{N}[g](t)  . \]
\end{Lem}
\begin{proof}
The proof relies on elliptic estimates, the relation $\partial_{v^i} [\psi (tv)]=t \partial_{x^i} \psi (tv)$, and
\[ \Delta_v \big[ \partial_t \big( t^{1+|\gamma|} \phi  \big[ \partial_x^\gamma f \big](t,tv) \big) \big] = \partial_t \big( t^{3+|\gamma|} \rho \big[ \partial_x^{\gamma}  f \big](t,tv) \big)= \partial_t \big( t^{3}  \rho \big[ (t \partial_x+\partial_v)^\gamma f \big](t,tv) \big). \]
Recall that $\partial_v^\gamma f_\circ (t,x,v)=[(t \partial_x+\partial_v)^\gamma f](t,x+tv,v)$, so that \eqref{eq:cdv}, applied to $h=(t \partial_x+\partial_v)^\gamma f$, yields
\begin{align*}
 \partial_t \big( t^{3}  \rho \big[ (t \partial_x+\partial_v)^\gamma f \big](t,tv) \big) &  =\int_{\R^3_y} \big[ \partial_t \partial_v^\gamma f_\circ \big] \Big( t,y,v-\frac{y}{t} \Big) +\frac{y}{t^2} \cdot \big[ \nabla_v \partial_v^\gamma f_\circ \big] \Big( t,y,v-\frac{y}{t} \Big)   \dr y . 
 \end{align*}
According to the previous Lemma \ref{LemComfcirc}, the first term on the right hand side vanishes. For the second term, we apply Hölder inequality in $y$ to get
\[  \bigg\|\langle v \rangle \int_{\R^3_y} \frac{y}{t^2} \cdot \big[ \nabla_v \partial_v^\gamma f_\circ \big] \Big( t,y,v-\frac{y}{t} \Big)  \dr y \bigg\|_{L^p_v} \leq \frac{1}{t^2} \Big\| \langle v \rangle \langle y \rangle^4 \big[ \nabla_v \partial_v^\gamma f_\circ \big] \Big( t,y,v-\frac{y}{t} \Big)  \Big\|_{L^p_{y,v}} = \frac{1}{t^2} \Big\| \langle v \rangle \langle y \rangle^5  \nabla_v \partial_v^\gamma f_\circ  ( t,y,v )  \Big\|_{L^p_{y,v}}  , \]
where, in the last step, we performed a change of variables and use the inequality $\langle z+w \rangle \leq \langle z \rangle\langle w \rangle$.
The last term can then be expressed in terms of $g$ through Lemma \ref{Corderivgtof}. It allows to derive the $L^p$ bound by the elliptic estimates of Proposition \ref{Proellip}, applied with $a=1$. 

By the Sobolev embedding $W^{p_\infty,p}(\R^3_v) \hookrightarrow L^\infty (\R^3_v)$, we get the pointwise bounds for $|\kappa| \geq 1$. For the last case, we will use the pointwise bound in Proposition \ref{Proellip} for $\tau=0$. For this, observe that by using first a Sobolev embedding in $v$, and then the Hölder inequality in $y$, we have
\begin{align}
 \bigg\|\langle v \rangle^3 \int_{\R^3_y} \frac{y}{t^2} \cdot \big[ \nabla_v  f_\circ \big] & \Big( t,y,v-\frac{y}{t} \Big)  \dr y \bigg\|_{L^\infty_v}  \leq \sum_{|\beta| \leq p_\infty}\frac{1}{t^2}\bigg\| \langle v \rangle^3 \int_{\R^3_y} y \cdot \big[ \nabla_v \partial_v^\beta f_\circ \big] \Big( t,y,v-\frac{y}{t} \Big)  \dr y \bigg\|_{L^p_v} \label{EQUATI:1} \\
& \lesssim  \sum_{|\beta| \leq p_\infty} \frac{1}{t^2}\bigg\| \langle v \rangle^3 \langle y \rangle^4 \big[\nabla_v  \partial_v^\beta f_\circ \big] \Big( t,y,v-\frac{y}{t} \Big)   \bigg\|_{L^p_{y,v}} \lesssim \sum_{|\beta| \leq p_\infty} \frac{1}{t^2}\bigg\| \langle v \rangle^3 \langle y \rangle^7 \big[\nabla_v  \partial_v^\beta f_\circ \big] ( t,y,v )   \bigg\|_{L^p_{y,v}}. \nonumber
 \end{align}
 We conclude by using again Lemma \ref{Corderivgtof}.
\end{proof}

We complement here the last result by dealing with the small times. It will in particular allow to control the error terms of \eqref{type3} in the commutation formula determined below. We use again the constant $b \coloneqq (a+1)(2N+7)$.

\begin{Lem}\label{Lemdpphismall}
Assume $N \geq 2p_\infty$. We have, for all $(t,v) \in \R_+ \times \R^3_v$,
\begin{alignat*}{2}
 \Big\| \langle v \rangle \partial_t \Big[  t^{3+|\gamma|} \nabla^2_x  \phi  \big[ \partial_x^\gamma f \big](t,tv) \Big] \Big\|_{L^p(\R^3_v)} & \lesssim \frac{\log^{b}\langle t+1 \rangle}{ \langle t \rangle^2} \E_{N}[g](t) , \qquad \qquad && p_\infty -1 \leq |\gamma| \leq N-1 , \\
 \langle v \rangle \Big|   \partial_t \Big[  t^{2+|\kappa|} \nabla_x  \phi  \big[ \partial_x^\kappa f \big](t,tv) \Big] \Big|& \lesssim \frac{\log^{b}\langle t+1 \rangle}{ \langle t \rangle^2} \E_{N}[g](t) , \qquad \qquad && |\kappa| \leq N-p_\infty.
 \end{alignat*}
\end{Lem}
\begin{proof}
In view of the previous Lemma \ref{LemdtPhi}, it suffices to deal with small times $0 \leq t \leq 1$. We will use the relation
\begin{align}
\langle v \rangle \big|\partial_t \big( t^{2+|\gamma|} \nabla_x \phi  \big[ \partial_x^\gamma f \big](t,tv) \big) \big| & \lesssim t \langle v \rangle \big| t^{|\gamma|} \nabla_x \phi  \big[ \partial_x^\gamma f \big](t,tv) \big|+t^2 \langle v \rangle^2 \big|t^{| \gamma|} \nabla^2_x \phi [ \partial_x^\gamma f] (t,tv) \big| \nonumber \\
& \quad \, +t \langle v \rangle \big| t^{1+|\gamma|} \partial_t \nabla_x \phi [ \partial_x^\gamma f] (t,tv) \big|. \label{EQU:truc}
\end{align}
We then need to control $\partial_t \nabla_x \phi [f]$. We exploit for this the conservation of the number of particles. More precisely, as $ \partial_t f+ v \cdot \nabla_x f+\nabla_v \cdot \big(f \nabla_x \phi \big)=0$, we have
\[ \Delta \partial_t \phi [f] = \partial_t \rho [f]= -\nabla_x \cdot \rho [ v f ]=-\rho \big[ v \cdot \nabla_x f \big]. \]
Note further that $v \cdot \nabla_x f(t,z(t,x,v),v)= v \cdot \nabla_x g (t,x,v)$. Then, we can apply Proposition \ref{Prodecayintvbis} to estimate $\rho \big[ v f \big]$ and its derivatives up to order $N$. Then, proceeding as in Proposition \ref{Proforcefield}, one obtains through elliptic estimates that, for all $(t,x) \in \R_+ \times \R^3_x$,
\begin{alignat}{2}
 \langle t \rangle^{\frac{2p-3}{p}} \Big\| \langle t+|x| \rangle^{2+|\alpha|} \nabla_x^2 \partial_t  \phi \big[ \partial_x^\alpha f \big] \Big\|_{L^p(\R^3_x)}& \lesssim \E_N[g](t) , \qquad && |\alpha| \leq N-1, \label{EQU:dtsmall} \\
 \langle t+|x| \rangle^{3+|\kappa|}  \big| \nabla_x \partial_t  \phi \big[ \partial_x^{\kappa} f \big] \big|(t,x) & \lesssim  \E_N[g](t), \qquad \qquad && |\kappa| \leq N-1-p_\infty , \label{EQU:dtsmall2}
\end{alignat}
where we used that $\E_N[g]=\E_N^{(7,7)}[g]$.

Using $t \langle x/t \rangle \leq t+|x|$, Proposition \ref{Proforcefield}, and \eqref{EQU:dtsmall2}, we obtain the pointwise decay estimates for $|\gamma| \leq N-1-p_\infty$. For the $L^p$ estimate, note that for a function $\psi (x)$,
\[ \big\| t^s \langle v \rangle^s t^q \psi (tv) \big\|_{L^p(\R^3_v)} \leq t^{q-\frac{3}{p}} \big\| \langle t+|x| \rangle^s \psi (x) \big\|_{L^p(\R^3_x)}, \qquad q, \, s \geq 0  . \]
This is obtained from the change of variables $x=tv$, and $t \langle x/t \rangle = (t^2+|x|^2)^{1/2}$. Using Proposition \ref{Proforcefield} and \eqref{EQU:dtsmall}, we can then bound in $L^p(\R^3_v)$, uniformly in $0 \leq t \leq 1$, the three terms in the right hand side of \eqref{EQU:truc}, provided that $|\gamma| \geq 3/p$. We finally derive the pointwise estimates for $|\kappa|=N-p_\infty$ through a Sobolev embedding.
\end{proof}

We can now prove that, for suitable functions $f$, $t^2 \nabla_x \phi[f] (t,tv)$ converges as $t \to + \infty$. Recall for this that $\phi_\infty [f_\infty]$ is the unique potential, going to $0$ as $|x| \to +\infty$, satisfying $\Delta_v \phi_\infty [f_\infty]= \int_x f_\infty (x,\cdot) \dr x$.

\begin{Cor}\label{Corphiphiinfty}
Assume that $N \geq 2 p_\infty$, that $\E_N[g]$ is uniformly bounded, and that there exists $f_\infty \colon \R^3_x \times \R^3_v \to \R$ such that
\[   \lim_{t \to + \infty} \big\| \langle v \rangle^3 \langle x \rangle^4 \big( g(t,x,v)-f_\infty (x,v) \big) \big\|_{L^\infty ( \R^3_x \times \R^3_v)}=0 .\]
 Then,  for all $t \geq 0$ and any $p_\infty -1 \leq  |\gamma| \leq N-1$,
 \[ \Big\| \langle v \rangle \Big(  t^{3+|\gamma|} \nabla^2_x \phi  \big[ \partial_x^\gamma f \big](t,tv) - \nabla^2_v \phi_\infty \big[ \partial_x^\gamma f \big](v)  \Big) \Big\|_{L^p(\R^3_v)} \lesssim  \frac{\log^{b}\langle t+1 \rangle}{ \langle t \rangle } \sup_{\tau \geq 0} \E_{N}[g](\tau), \]
where $b \coloneqq (a+1)(2N+7)$. Moreover, for all $(t,v) \in [1,+\infty[ \times \R^3_v$ and any $|\kappa| \leq N-p_\infty$,
\[ \langle v \rangle \Big|  t^{2+|\kappa|} \nabla_x \phi \big[ \partial_x^\kappa f \big] (t,tv) - \nabla_v  \phi_\infty \big[ \partial_x^\kappa f_\infty \big] (v) \Big| \lesssim  \frac{\log^{b}\langle t +1 \rangle}{\langle t \rangle} \sup_{\tau \geq 0}\E_{N}[g](\tau).\]
\end{Cor}
\begin{proof}
The assumption on $\E_N[g]$, together with the previous Lemma \ref{LemdtPhi} and the Minkowski integral inequality, imply that $ t^2 \nabla_x \phi [f](t,tv)$ converges with the stated rate of convergence.

To identify the limit, we note that \eqref{eq:cdv} and the fundamental theorem of calculus imply
\begin{align*} 
\bigg\| \langle v \rangle^3  t^3 \rho [f](t,tv)- \langle v \rangle^3 \int_{\R^3_x} f_\circ (t,x,v) \dr x \bigg\|_{ L^\infty (\R^3_v)} & \leq \frac{1}{t} \int_{\theta=0}^1\bigg\| \langle v \rangle^3\int_{\R^3_y} y \cdot \nabla_v   f_\circ \Big(t,y, v - \theta\frac{y}{t} \Big) \dr y \bigg\|_{L^\infty  ( \R^3_v)} \dr \theta \\
& \lesssim \frac{1}{t} \log^{b} \langle t +1 \rangle \E_N[g](t) ,
\end{align*}
where, in the last step, we proceed as in \eqref{EQUATI:1}. Using that the spatial average of $g$ and $f_\circ$ correspond and
\[ \Big\| \langle v \rangle^2 \nabla_v \Big( t \phi [f] (t,tv)-\phi_\infty [f_\infty] (v) \Big) \Big\|_{L^\infty(\R^3_v)} \lesssim \bigg\| \langle v \rangle^3 \bigg(  t^3 \rho [f] (t,tv)- \int_{\R^3_x} f_\infty (x,v) \dr x \bigg)  \bigg\|_{ L^\infty (\R^3_v)} ,\]
which follows from Proposition \ref{Proellip}, we finally derive that $t^2 \nabla_x \phi [f](t,tv)$ converges uniformly to $\nabla_v \phi_\infty [f_\infty]$.
\end{proof}

\subsection{Commutation formula}

Let $ \phi \colon \R_+ \times \R^3 \to \R$ be a potential. The goal is to compute the commutators of $\overline{\T}_\phi$, the transport operator introduced in Lemma \ref{LemTbar}. To simplify the notation, we introduce the following convention.
\begin{Def}\label{DefpolynomPhi}
For $r \in \mathbb{N}$ and $q \leq N$, we will denote by $P_{r,q}(\Phi)$ a quantity of the form
\[ \prod_{1 \leq k \leq r} \partial_{v^{j_k}}  \partial_{v}^{\alpha_k} \Phi^{i_k}(t,v), \qquad \qquad \sum_{1 \leq k \leq r} |\alpha_k| =q, \qquad 1 \leq i_k, \, j_k \leq 3, \]
where $P_{0,0}(\Phi)=1$ by convention.
\end{Def}

We obtain the first order commutation formula by using the expression of $\overline{\T}_{\phi}$ given in Lemma \ref{LemTbar}, that $\partial_{x^i}z^j= \delta_{i,j}$, and $\partial_{v^i} z^j = t \delta_{i,j}+\lambda \partial_{v^i} \Phi^j (t,v) \log \langle t \rangle$.
 
 \begin{Lem}\label{LemCom}
 We have
 \begin{align*}
 \big[ \T_{\phi } , \partial_{x^i} \big] & = - \lambda t \nabla_x \partial_{x^i} \phi  (t,z)\cdot \nabla_x  - \log \langle t \rangle  \nabla_v \Phi \cdot \nabla_x \partial_{x^i}\phi(t,z) \cdot \nabla_x + \lambda \nabla_x  \partial_{x^i} \phi (t,z) \cdot \nabla_v
 \end{align*}
 and 
 \begin{align*}
 \big[ \T_{\phi } , \partial_{v^i} \big]  = & - \frac{\lambda t}{ \langle t \rangle^2}\Big( \langle t \rangle^2 t\nabla_x \partial_{x^i} \phi  (t,z) - \partial_{v^i}\Phi \Big) \cdot \nabla_x - t \log \langle t \rangle  \partial_{v^i} \Phi  \cdot \nabla_x^2 \phi (t,z) \cdot \nabla_x  + \lambda \log \langle t \rangle \partial_t \big( \partial_{v^i} \Phi\big) \cdot \nabla_x \\
 & -  \log \langle t \rangle \nabla_v \partial_{v^i} \Phi \cdot \nabla_x \phi (t,z) \cdot \nabla_x   -  t \log \langle t \rangle \nabla_v \Phi \cdot \nabla_x  \partial_{x^i} \phi(t,z) \cdot \nabla_x \\
 & - \sum_{1 \leq k \leq 3} \lambda \log^2\langle t \rangle \partial_{v^i} \Phi \cdot \nabla_v \Phi^k \cdot \nabla_x^2 \phi (t,z)  \cdot \partial_{x^k}  +\lambda t \nabla_x  \partial_{x^i} \phi (t,z) \cdot \nabla_v + \log \langle t \rangle \partial_{v^i} \Phi \cdot \nabla_x^2 \phi (t,z)\cdot  \nabla_v .
 \end{align*}
 \end{Lem}

By iterating the above, one can derive the following high order commutation relation.

\begin{Cor}\label{CorComm}
Let $|\alpha_x|+|\alpha_v| \leq N$. Then, $\big[ \, \overline{\T}_{\phi }, \partial_x^{\alpha_x} \partial_v^{\alpha_v} \big]$ is a linear combination of the following terms.
\begin{itemize}
\item The ones of \eqref{type1}
\begin{align}
& \frac{t}{\langle t \rangle^2}\Big( \langle t \rangle^2 t^{|\gamma|}\nabla_x \partial_x^\gamma \phi  (t,z) -  \partial_{v}^{\gamma}\Phi (t,v) \Big) \cdot \nabla_x \partial_x^{\alpha_x} \partial_v^{\beta_v}, \qquad \qquad |\gamma| \geq 1, \quad \gamma+\beta_v = \alpha_v. \label{type1} \tag{\text{type}--$1$}
\end{align}
\item The ones of \eqref{type2}, with the convention that $\partial_{x^k}^\delta \partial_{v^k}^{1-\delta}=\partial_{x^k}$ if $\delta=1$ and $\partial_{x^k}^\delta \partial_{v^k}^{1-\delta}=\partial_{v^k}$ if $\delta=0$,
\begin{align*}
& t^{\delta+|\alpha_v|-|\beta_v|-r-q} \log^r \langle t \rangle P_{r,q}(\Phi) \partial_{x^i}  \partial_x^{\gamma}  \phi (t,z) \cdot \partial_{x^k}^\delta \partial_{v^k}^{1-\delta}\partial_x^{\beta_x} \partial_v^{\beta_v},  \label{type2} \tag{\text{type}--$2$}
\end{align*}
where $r \leq |\alpha_v|+1$ and $q+|\gamma|+|\beta| = |\alpha|$. Moreover, we have $|\beta_x| \leq |\alpha_x|$, $|\beta_v| \leq |\alpha_v|$, $|\beta|<|\alpha|$, and the additional conditions
\begin{enumerate}
\item $\delta+|\alpha_v|-|\beta_v|-r-q \geq 0$, so that there are no negative powers of $t$,
\item if $|\beta_x|=|\alpha_x|$, then $r \geq 1$.
\end{enumerate}
\item The ones of \eqref{type3}, which vanish if $\Phi$ does not depend on $t$,
\begin{equation}\label{type3} 
\log \langle t \rangle \partial_t \partial_{v}^\gamma (\Phi) \cdot \nabla_x \partial_x^{\alpha_x} \partial_{v}^{\beta_v}  , \qquad \qquad |\gamma| \geq 1, \quad \gamma+\beta_v = \alpha_v \tag{\text{type}--$3$}
\end{equation}
\end{itemize}
\end{Cor}

\begin{proof}
We argue by induction on $ |\alpha|=|\alpha_x|+|\alpha_v| $ and we note that the first order case $ |\alpha|=1 $ is treated by Lemma \ref{LemCom}. In particular, for the second additional condition for the terms of \eqref{type2}, the case $|\beta_x| = |\alpha_x|$ can only occur when commuting with $\partial_{v^i}$. Assume now that the result holds for all multi-indices of length $m$, and let $ |\alpha_x|+|\alpha_v|=m+1$. Write
\[ \partial_x^{\alpha_x}\partial_v^{\alpha_v} = \partial_{w}\partial_x^{\kappa_x}\partial_v^{\kappa_v}, \qquad |\kappa_x|+|\kappa_v|=m, \]
where $\partial_w$ denotes either one space derivative $\partial_{x^i}$ or one velocity derivative $\partial_{v^i}$. Then
\[ \big[ \,  \overline{\T}_{\phi[f]},\partial_x^{\alpha_x}\partial_v^{\alpha_v} \big] = \big[  \, \overline{\T}_{\phi[f]},\partial_w \big]\partial_x^{\kappa_x}\partial_v^{\kappa_v} + \partial_w\Big( \big[ \, \overline{\T}_{\phi[f]},\partial_x^{\kappa_x}\partial_v^{\kappa_v} \big]\Big). \]
The first term is of the required form by the first order commutation formula of Lemma \ref{LemCom}. It then remains to check that the three classes of terms are stable under one additional derivative.

If $\partial_w=\partial_{x^i}$, then no new factor of $t$, $\log \langle t \rangle$, or $\Phi$ is created, since $\partial_{x^j} z^k=\delta_{j,k}$. As a consequence, differentiating any term of \eqref{type1}, \eqref{type2}, or \eqref{type3} simply increases the number of $x$-derivatives falling either on $\partial_x^\gamma \phi $ or on the final differential operator. The structure is thus preserved.

Assume now that $\partial_w=\partial_{v^i}$. Then, there are four cases to consider.

\begin{itemize}
\item The derivative $\partial_{v^i}$ hits the differential operator $\nabla_x \partial_x^{\alpha_x}\partial_v^{\beta_v}$ or $\partial_{x^k}^\delta \partial_{v^k}^{1-\delta}\partial_x^{\beta_x} \partial_v^{\beta_v}$. Then $|\beta_v|$ increases by one, which preserves the structure of the terms \eqref{type1}--\eqref{type3}.
\item The derivative $\partial_{v^i}$ hits the factor $P_{r,q}(\Phi)$. In that case, this produces $r$ quantities of the form $P_{r,q+1}(\Phi)$. Thus the structure of the terms of \eqref{type2} is again preserved. 
\item The derivative $\partial_{v^i}$ hits the force field $\nabla_x \partial_x^\gamma \phi(t,z)$. Using $\partial_{v^i} z^j = t\delta_{ij}+\lambda \log \langle t \rangle\,\partial_{v^i}\Phi^j$,
we obtain
\[ \partial_{v^i}\big[\partial_{x^k}\partial_x^\gamma \phi (t,z)\big] = t \partial_{x^i}\partial_{x^k}\partial_x^\gamma \phi(t,z) + \lambda \log \langle t \rangle \sum_{1 \leq j \leq 3} \partial_{v^i}\Phi^j \partial_{x^j}\partial_{x^k} \partial_x^\gamma \phi (t,z). \]
Hence, when the velocity derivative $\partial_{v^i}$ falls on $\nabla_x \partial_x^\gamma \phi(t,z)$, it can only
\begin{enumerate}[ label = (\roman*)]
\item \label{itemitem1}  produce one additional power of $t$, together with one more $x$-derivative on $\phi$,
\item \label{itemitem2} or produce one additional factor $\log \langle t \rangle \partial_{v^i} \Phi$, and again one more $x$-derivative on $\phi$.
\end{enumerate}
For the terms of \eqref{type2}, this is exactly the origin of the exponent for the factor $t^{\delta+|\alpha_v|-|\beta_v|-r-q}$, and this proves that the constraint $q+|\gamma|+|\beta|=|\alpha|$ is preserved.

A term of \eqref{type1} gives rise to quantity with the same structure when \ref{itemitem1} occurs. Indeed, in that case $\partial_{v^i}$ also hits $\partial_v^\gamma \Phi$. If \ref{itemitem2} occurs, it gives rise to a \eqref{type2} term.
\item The derivative $\partial_{v^i}$ hits $\partial_t \partial_v^\gamma \Phi$. In that case, the structure of the terms of \eqref{type3} is preserved.
\end{itemize}

Therefore each additional derivative sends any admissible term of order $m$ into a linear combination of terms of \eqref{type1}--\eqref{type3} of order $m+1$. This concludes the induction.
\end{proof}

In practice, and in view of the hirarchised norms that we will propagate, we will appeal to the next result.

\begin{Cor}\label{Corenergy}
We can bound $\langle x \rangle^{|\alpha_x|} \langle v \rangle^{-|\alpha_x|}\big| \big[ \T_{\phi } , \partial_x^{\alpha_x} \partial_v^{\alpha_v} \big](g) \big|$ by a linear combination of the terms
\begin{alignat*}{2}
& \frac{\langle v \rangle}{\langle t \rangle  \langle x \rangle}  \Big| \langle t \rangle^2 t^{|\gamma|} \nabla_x \partial_{x}^{\gamma}  \phi  (t,z) - \partial_{v}^{\gamma} \Phi \Big| \cdot \frac{\langle x \rangle^{|\beta_x|}}{\langle v \rangle^{|\beta_x|}} \big| \partial_x^{\beta_x} \partial_v^{\beta_v} g \big|, \qquad && |\gamma| \geq 1, \quad \gamma+\beta_v = \alpha_v, \quad |\beta_x|=|\alpha_x|+1, \\ 
&\log^Q \! \langle t+1 \rangle  \big( 1+| \langle v \rangle \nabla_v \partial_v^{\xi} \Phi |\big) \Big| \langle t+|z| \rangle^{|\gamma|} \nabla_x \partial_x^\gamma \phi (t,z) \Big| \frac{\langle x \rangle^{|\beta_x|}}{\langle v \rangle^{|\beta_x|}} \big| \partial_x^{\beta_x} \partial_v^{\beta_v} g \big|, \quad \;  && |\gamma|+|\xi| +|\beta| \leq |\alpha|+1, \; \; |\xi|, \,|\gamma|, \, |\beta| \leq |\alpha|, \\
& \log \langle t \rangle  \big| \langle  v \rangle \partial_t \big( \partial_v^{\gamma} \Phi \big) \big| \cdot \frac{\langle x \rangle^{|\beta_x|}}{\langle v \rangle^{|\beta_x|}} \big| \partial_x^{\beta_x} \partial_v^{\beta_v} g \big|, \qquad &&|\gamma| \geq 1, \quad \big( |\gamma|-1 \big) +|\beta| \leq |\alpha| ,
\end{alignat*}
where $Q= (a+1)N+1$ and $|\gamma| \leq |\alpha|$.
\end{Cor}
\begin{proof}
We will use the following inequality, which follows from $z=x+tv+\Phi \log \langle t \rangle$ and $|\Phi| \leq \mathbf{\Lambda} \log^a \langle t+1 \rangle$,
\begin{align}
\langle x \rangle & \lesssim \langle z \rangle+ t \, \langle v \rangle +  \log^{a+1} \langle t+1 \rangle \lesssim   \langle t+ |z| \rangle \, \langle v \rangle . \label{eq:x}
\end{align}
We consider first the terms of \eqref{type2}. Since $\| \nabla_v \Phi\|_{W^{N-p_\infty,\infty}_v} \lesssim \mathbf{\Lambda} \log^a \langle t+1 \rangle$ by \eqref{EQ:AssumpPhi}, we have for $1 \leq r \leq N+1$, 
\begin{equation}\label{eq:Ppq}
 \big| P_{r,q} (\Phi) \big| \lesssim \log^{(r-1)a} \langle t+1 \rangle \big|  \nabla_v \partial_v^\xi  \Phi \big| \lesssim \frac{\log^{Na} \langle t+1 \rangle}{\langle v \rangle}\big| \langle v \rangle \nabla_v \partial_v^\xi  \Phi \big|  , \qquad \qquad |\xi| \leq q.
\end{equation}
We then need to bound
\[ \mathcal{T} \coloneqq  \langle t \rangle^{\delta +|\alpha_v|-|\beta_v|-r-q} \big| \partial_{x^i}  \partial_x^\gamma \phi (t,z) \big| \cdot \langle x \rangle^{|\alpha_x|} \langle v \rangle^{-|\alpha_x|} \big|\partial_{x^i}^\delta \partial_{v^i}^{1-\delta} \partial_x^{\beta_x} \partial_v^{\beta_v} g \big|, \qquad \quad \delta \in \{ 0,1 \}, \quad  |\beta_x| \leq |\alpha_x|, \quad |\beta_v| \leq |\alpha_v|, \]
$q+|\gamma|+|\beta| = |\alpha|$, and $|\beta| < |\alpha|$. We note that we used the first additional condition, namely that no term of \eqref{type2} involves negative powers of $t$. Let $\kappa$ such that $\partial_x^{\kappa_x} \partial_v^{\kappa_v} = \partial_{x^i}^\delta \partial_{v^i}^{1-\delta} \partial_x^{\beta_x} \partial_v^{\beta_v}$.
\begin{itemize}
\item Assume first that $\delta=1$ and $|\beta_x|=|\alpha_x|$, so that $|\kappa_x|=|\alpha_x|+1$, and $r \geq 1$ by the second additional condition. We have
\[
\mathcal{T} = \langle v \rangle   \langle t \rangle^{1+|\alpha_v|-|\beta_v|-r-q}  \big| \partial_{x^i}  \partial_x^\gamma \phi (t,z) \big| \cdot \langle x \rangle^{|\kappa_x|-1} \langle v \rangle^{-|\kappa_x|} \big|  \partial_x^{\kappa_x} \partial_v^{\kappa_v} g \big|.
\]
Then, note that in that case, $|\gamma| = |\alpha_v|-|\beta_v|-q$, and $1-r \leq 0$. Moreover, the extra factor $\langle v \rangle$ is absorbed by \eqref{eq:Ppq} since $r \geq 1$.
\item Otherwise $|\beta_x|+\delta=|\kappa_x | \leq |\alpha_x|$ and we use \eqref{eq:x} to get
\[
\mathcal{T} \lesssim  \langle t \rangle^{\delta+|\alpha_v|-|\beta_v|-r-q} \langle t+| z| \rangle^{|\alpha_x|-|\kappa_x|}  \big| \partial_{x^i}  \partial_x^\gamma \phi (t,z) \big| \cdot \langle x \rangle^{|\kappa_x|} \langle v \rangle^{-|\kappa_x|} \big|  \partial_x^{\kappa_x} \partial_v^{\kappa_v} g \big|.
\]
As $|\gamma| = |\alpha|-|\beta|-q$, we have $\langle t \rangle^{\delta+|\alpha_v|-|\beta_v|-r-q} \langle t + |z| \rangle^{|\alpha_x|-|\kappa_x|} \leq \langle t+|z| \rangle^{|\gamma|-r}$, which allows to conclude the analysis of this error term.
\end{itemize}
For the terms of \eqref{type1} and \eqref{type3}, we simply observe that the differential operator applied to $g$ is $\partial_x^{\beta_x} \partial_{v}^{\beta_v}=\nabla_x \partial_x^{\alpha_x} \partial_{v}^{\beta_v} $, where $|\beta_x|=|\alpha_x|+1$. In the latter case, we additionnaly use $ \langle x \rangle^{-1} \leq  1$.
\end{proof}

\subsection{Energy inequality} The next energy inequality will be used for both the forward and backward problems.

\begin{Pro}\label{Proenergy}
Let $0 \leq t_1 \leq t_2 \leq +\infty$. Then, for any $p \geq 1$,
\[ \Big| \big\| g(t_1,\cdot , \cdot ) \big\|_{L^p(\R^3_x \times \R^3_v)} - \big\| g(t_2,\cdot , \cdot ) \big\|_{L^p(\R^3_x \times \R^3_v)} \Big| \leq p \int_{t=t_1}^{t_2} \big\| \overline{\T}_\phi (g) (t,\cdot,\cdot) \big\|_{L^p(\R^3_x \times \R^3_v)} \dr t. \]
\end{Pro}
\begin{proof}
We start by working with the function $f$. As $\T_\phi = \partial_t+v\cdot \nabla_x-\lambda \nabla_x \phi \cdot \nabla_v$, integration by parts provide
\begin{align*}
\Big| \big\| f(t_1,\cdot , \cdot) \big\|_{L^p_{x,v}}^p-\big\| f(t_2,\cdot , \cdot) \big\|_{L^p_{x,v}}^p \Big| & = \bigg| \int_{t=t_1}^{t_2} \int_{\R^3_x} \int_{\R^3_v} \T_\phi \big(  |f|^p \big) \dr v \dr x \dr t \bigg| \\
& \leq p \sup_{t_1 \leq t \leq t_2} \big\| f(t_1,\cdot , \cdot) \big\|_{L^p_{x,v}}^{p-1} \int_{t=t_1}^{t_2} \big\| \T_\phi (f) (t,\cdot,\cdot) \big\|_{L^p(\R^3_x \times \R^3_v)} \dr t ,
\end{align*}
where, in the last step, we used $\T_\phi (|f|^p)=p \T_\phi (f) f|f|^{p-2}$ and the Hölder inequality. It allows to derive the result, since the $L^p$ norms of $f$ (respectively $\T_\phi (f)$) and $g$ (respectively $\overline{\T}_\phi (g)$) coincide, by the change of variables $x \mapsto x+tv+\lambda \Phi (t,v) \log \langle t \rangle$.
\end{proof}

\section{Analysis of the transport operator $\overline{\T}_\phi$}\label{SecTphi}

The purpose of this section is to study, given a sufficiently regular potential $\phi$ and function $\Phi$, the properties of the solutions to $\overline{\T}_\phi (g)=0$ arising from sufficiently regular initial or scattering data. In particular, we will show boundedness for $g$, convergence as $t \to +\infty$, and that the asymptotic Cauchy problem is well-posed. 

\subsection{Assumptions on the force field} Let, for all this section, $N \geq 2p_\infty$. We begin by specifying the class of potentials $\phi$ and coordinates $(t,z(t,x,v),v)$ under consideration. Our motivation for considering such a class of potential is given in Remark \ref{Rqcomment} below.

\begin{Def}\label{Defforcefield}
Let $I$ be an interval and $t \mapsto \Lambda_t$ be a nonnegative function. We will say that the couple $(\phi,\Phi)$, composed by a potential $\phi \colon I \times \R^3 \to \R$ and a function $\Phi \colon I \times \R^3_v \to \R^3$, belongs to the classe $\mathcal{W}^N(I,\Lambda)$ if it satisfies the following decay estimates. 
\begin{itemize}[leftmargin=1.5em]
\item For any $|\kappa| \leq N-p_\infty$, $|\gamma| \leq N$, and all $(t,x) \in I \times \R^3_x$,
\begin{align}
 \langle t+|x| \rangle^{2+|\kappa|} \big| \nabla_x \partial_x^\kappa \phi \big|(t,x) & \leq \Lambda_t , \label{eq:phiassump1} \tag{$\phi$, $L^\infty$, weak} \\
\langle t\rangle^{2- \frac{3}{p}}\big\| \langle t+|x| \rangle^{1+|\gamma|} \nabla_x^2 \partial_x^\gamma \phi (t,x) \big\|_{L^p(\R^3_x)} & \leq \Lambda_t . \label{eq:phiassump2} \tag{$\phi$, $L^p$, weak} 
\end{align}
\item For all $t \in I$, 
\begin{equation}\label{EQU:HPhi}
\mathcal{E}_N[\Phi](t) = \Big\| \langle v \rangle^2 \Phi (t,v) \Big\|_{L^\infty (\R^3_v)}+ \sum_{|\gamma| \leq N} \Big\| \langle v \rangle^{1+|\gamma|} \nabla_v \partial_v^\gamma \Phi (t,v) \Big\|_{L^p(\R^3_v)}  \leq \Lambda_t .  \tag{$\Phi$, bound}
\end{equation}
\item For any $|\kappa| \leq N-p_\infty$, $p_\infty \leq |\alpha| \leq N$, and all $(t,v) \in I \times \R^3_v$,
\begin{align}
\frac{ \langle v \rangle t}{\langle t \rangle^2}\Big| \langle t \rangle^2 t^{|\kappa|} \nabla_x \partial_x^{\kappa} \phi ( t,tv ) -  \partial_v^\kappa \Phi (t,v) \Big| + \langle v \rangle\Big| \partial_t  \partial_v^\kappa \Phi \Big|(t,v) &  \leq  \frac{ \Lambda_t}{\langle t \rangle^{\frac{9}{5}}}, \label{eq:phiassump3} \tag{$L^\infty \hspace{-0.3mm}$, sharp} \\
\hspace{-2em}  \frac{t}{\langle t \rangle^2}\Big\| \langle v \rangle \Big( \langle t \rangle^2 t^{1+|\alpha|} \nabla_x \partial_x^{\alpha} \phi (t,tv)-  \partial_v^\alpha \Phi (t,v) \Big) \Big\|_{L^p(\R^3_v)} \! + \Big\| \langle v \rangle \partial_t  \partial_v^\alpha \Phi (t,v) \Big\|_{L^p(\R^3_v)} &   \leq  \frac{\Lambda_t}{ \langle t \rangle^{\frac{9}{5}}}. \label{eq:phiassump4} \tag{$L^p \hspace{-0.3mm}$, sharp}
\end{align}
\end{itemize}
Associated to the function $\Phi$, we define the coordinate system $(t,z(t,x,v),v)$, where
\[ z(t,x,v) \coloneqq x+tv+\lambda \Phi (t,v) \log \langle t \rangle. \]
We will often write $z$ instead of $z(t,x,v)$.
\end{Def}
\begin{Rq}\label{Rq:phiassump2}
By \eqref{eq:phiassump2} and the weighted Sobolev embedding of Proposition \ref{ProSob}, we have
\[\langle t \rangle^2\langle t+|x| \rangle^{|\kappa|} \big|\nabla_x \partial_x^\kappa \phi \big| (t,x)\lesssim \Lambda_t , \qquad \qquad |\kappa|=N+1-p_\infty.\]
\end{Rq}

Let us show two results, which motivate Definition \ref{Defforcefield}. The next one will be useful in Sections \ref{Secmodwave} and \ref{SubsecForward} when we will perform a bootstrap argument during the analysis of a solution $f$ to the Vlasov-Poisson system. 

\begin{Cor}\label{Rqclass}
Let $f  \colon I \times \R^3_x \times \R^3_v \to \R$, and $\Phi \colon I \times \R^3_v \to \R^3_v$ be two functions, and assume that for all $t \in I$,
\[ \mathcal{E}_N[\Phi](t) \leq \mathbf{\Lambda}\log^a \langle t+1 \rangle, \qquad \qquad \qquad  \E_N [g](t) <+ \infty, \]
for some constants $a, \, \mathbf{\Lambda} \geq 0$, where the function $g$ is defined by
\[ g(t,x,v) \coloneqq f \big( t, x+tv+\lambda \Phi(t,v) \log \langle t \rangle , v \big)  . \]
Let further $\Phi_f (t,v) \coloneqq t^2 \nabla_x \phi [f] (t,tv)$. Then, there exists $C [N,p,\mathbf{\Lambda},a]>0$ such that $(\phi [f],\Phi_f) \in \mathcal{W}^N \big( I, C\E_N[g] \big)$.
\end{Cor}
\begin{Rq}
In practice, we will apply Corollary \ref{Rqclass} either with $\Phi=0$, or with $\Phi=\Phi_f$. 
\end{Rq}
\begin{proof}
The assumptions allow us to work within the framework of Section \ref{Secprep}. 
\begin{itemize}
\item We obtain \eqref{eq:phiassump1}--\eqref{eq:phiassump2} from Proposition \ref{Proforcefield}. 
\item Then, by the change of variables $x=tv$ and $t \langle x/t\rangle \leq t+|x|$, there holds
\begin{equation}\label{toremember}
 \Big\| \langle v \rangle^{1+|\gamma|} t^{3+|\gamma|} \nabla^2_x \partial_x^\gamma \phi [f] (t,tv) \Big\|_{L^p(\R^3_v)}  \leq t^{2-\frac{3}{p}} \Big\| (t+|x|)^{1+|\gamma|} \nabla^2_x \partial_x^\gamma \phi [f]  (t,x) \Big\|_{L^p(\R^3_x)} . 
 \end{equation}
Note then that $t^{2-\frac{3}{p}} \leq \langle t \rangle^{2-\frac{3}{p}}$ since $p >3/2$. Also, we have $\langle v \rangle^2 |\Phi_f (t,v) | \leq ( t+|tv| )^2 |\nabla_x \phi [f]|(t,tv)$ by using $t\langle v \rangle \leq t+t|v|$. As \eqref{eq:phiassump1}--\eqref{eq:phiassump2} holds, this gives \eqref{EQU:HPhi}.
\item The estimates for $\partial_t \Phi$ are given by Lemma \ref{Lemdpphismall}. Finally, the last estimates to establish consists in controlling 
\[   \frac{ \langle v \rangle t}{\langle t \rangle^2}\Big| \langle t \rangle^2 t^{|\kappa|} \nabla_x \partial_x^{\kappa} \phi [f] ( t,tv ) -  \partial_v^\kappa \Phi (t,v) \Big| =\frac{ \langle v \rangle t}{\langle t \rangle^2}\Big| t^{|\kappa|} \nabla_x \partial_x^{\kappa} \phi [f] ( t,tv ) \Big| . \] For the $L^\infty$ case, we use again $t \langle v \rangle \leq t+t|v|$ and \eqref{eq:phiassump1}. For the $L^p$ estimate, we perform the change of variables $x=tv$ to get , for $p_\infty - 1 \leq |\alpha| \leq N-1$,
\begin{align}
\frac{1}{\langle t \rangle^2} \Big\| \langle v \rangle t^{2+|\alpha|} \nabla_x^2  \partial_x^\alpha \phi [  f ](t,tv) \Big\|_{L^p(\R^3_v)} & \leq  \frac{1}{\langle t \rangle^2}\Big\| (t+|x|) t^{1+|\alpha|-\frac{3}{p}} \nabla_x^2  \partial_x^\alpha \phi [  f ](t,x) \Big\|_{L^p(\R^3_x)} \nonumber \\
& \leq \frac{t^{1+|\alpha|-\frac{3}{p}}}{\langle t \rangle^{4+|\alpha|-\frac{3}{p}}} \cdot \langle t \rangle^{2-\frac{3}{p}} \Big\| \langle t+|x| \rangle^{1+|\alpha|} \nabla^2_x \partial_x^\alpha \phi [f]  (t,x) \Big\|_{L^p(\R^3_x)} .  \label{eq:521479}
\end{align}
Then, to bound the right hand side for $0 \leq t \leq 1$, we require $|\alpha| \geq p_\infty-1 = \lfloor 3/p \rfloor$, and we use \eqref{eq:phiassump2}.
\end{itemize}
\end{proof}

We now prove a similar result, which will be useful for the construction of solutions arising from scattering data. In that setting, the coordinate system used to control the solutions is teleological rather than dynamical.

\begin{Cor}\label{Rqclassteleo}
Let $f_\infty \colon  \R^3_x \times \R^3_v \to \R$ and $f  \colon I \times \R^3_x \times \R^3_v \to \R$ be two functions, and assume that
\[ \E_N[f_\infty] = \mathbf{\Lambda}, \qquad \qquad \qquad   \forall \, t \in I, \quad  \E_N [g](t)<+\infty, \]
for some constant $ \mathbf{\Lambda} \geq 0$, where the function $g$ is defined by
\[ g(t,x,v) \coloneqq f \big( t, x+tv+\lambda \nabla_v \phi_\infty [f_\infty](v) \log \langle t \rangle , v \big)  . \]
Then, there exists $C [N,p,\mathbf{\Lambda}]>0$ such that $\big(\phi [f], \nabla_v \phi_\infty [f_\infty] \big) \in \mathcal{W}^N \big(  I, C (\mathbf{\Lambda}+\E_N[g]) \big)$.
\end{Cor}
\begin{proof}
The strategy is slightly different than the one of the previous proof.
\begin{itemize}
\item According to Proposition \ref{Prophiinfty}, there exists $C_1[N,p]>0$ such that $\mathcal{E}_N\big[\nabla_v \phi_\infty [f_\infty] \big] \leq C_1 \mathbf{\Lambda}$. In particular, \eqref{EQU:HPhi} holds, which allows us to apply the results of Section \ref{Secprep} with $\Phi = \nabla_v \phi_\infty [f_\infty]$.
\item We obtain that $\phi [f]$ satisfies \eqref{eq:phiassump1}--\eqref{eq:phiassump2} from Proposition \ref{Proforcefield}.
\item For \eqref{eq:phiassump3}--\eqref{eq:phiassump4}, note first that $\partial_t \phi_\infty[f_\infty]=0$. Then, write
\begin{align*}
 \frac{ \langle v \rangle t}{\langle t \rangle^2}\Big| \langle t \rangle^2 t^{|\kappa|} \nabla_x \partial_x^{\kappa} \phi [f] ( t,tv ) -  \nabla_v \partial_v^\kappa  \phi_\infty [f_\infty](v) \Big| & \leq \frac{ \langle v \rangle }{\langle t \rangle}\Big| t^{2+|\kappa|} \nabla_x \partial_x^{\kappa} \phi [f] ( t,tv ) -  \nabla_v \partial_v^\kappa  \phi_\infty [f_\infty](v) \Big|  \\
 & \quad \, + \frac{ \langle v \rangle t}{\langle t \rangle^2}\Big|  t^{|\kappa|} \nabla_x \partial_x^{\kappa} \phi [f] ( t,tv )\big| . 
 \end{align*}
 The first term on the right hand side can be estimated using Corollary \ref{Corphiphiinfty}. For the second term, we apply the final step in the proof of the previous Corollary \ref{Rqclass}. See in particular \eqref{eq:521479}.
\end{itemize}
\end{proof}

\begin{Rq}\label{Rqcomment}
In this paper, in addition to the cases previously discussed, we will consider two other pairs $(\phi, \Phi)$.
\begin{itemize}
\item A potential of the form $\phi [f]$, with $f$ a solution to \eqref{VP}, and $\Phi \equiv 0$. In this case, $\Lambda_t$ will be proportional to $\E_N[f_\circ]$, and allowed to grow logarithmically.
\item A potential of the form $\phi [f-\mathbf{f}]$, with $f$ and $\mathbf{f}$ two solutions to \eqref{VP}, and a function $\Phi = \Phi_{f-\mathbf{f}}$. In this case, $\Lambda_t$ will be proportional to $\E_N[g-\mathbf{g}](t)$.
\end{itemize}
\end{Rq}

We fix, for all this section, an interval $I \subset \R_+$, a nonnegative function $t \mapsto \Lambda_t$, and $(\phi, \Phi) \in \mathcal{W}^N \big(I,\Lambda \big)$. We assume further that there exists $\mathbf{\Lambda} , \, a \geq 0$ such that $\Lambda_t \leq \mathbf{\Lambda} \log^a \langle t+1 \rangle$, so that \eqref{EQU:HPhi} allows us to work within the framework of Section \ref{Secprep}.

\subsection{Energy estimates}

Motivated by the definition of the norms $\E_N^n [\cdot]$, we introduce the following shorthand notation.

\begin{Def}\label{Defgalpha}
Let $|\alpha| \leq N$, $(n_x ,  n_v ) \in \mathbb{Z}^2$, and $\mathfrak{g} \colon \R_+ \times \R^3_x \times \R^3_v \to \R$. We define the function $\mathfrak{g}_\alpha$ as
\[  \mathfrak{g}_\alpha (t,x,v) \coloneqq \langle x \rangle^{n_x+|\alpha_x|} \langle v \rangle^{n_v+N-|\alpha_x|} \partial_x^{\alpha_x} \partial_v^{\alpha_v} \mathfrak{g} (t,x,v) . \]
\end{Def}

We now prove a result which will in particular allow to prove boundedness for the solutions to $\overline{\T}_\phi (g)=0$.

\begin{Pro}\label{ProboundednessGeneral}
Recall that $(\phi, \Phi) \in \mathcal{W}^N(I,\Lambda)$. Let $g \colon I \times \R^3_x \times \R^3_v \to \R$ be a sufficiently regular function, and $n=(n_x,n_v) \in \mathbb{Z}^2$. Then, there exists a constant $C_0[N,n,p,\mathbf{\Lambda}]>0$ such that,
\[ \forall \, t \in I, \qquad \quad \Big\| \Big[ \, \overline{\T}_\phi , \langle x \rangle^{n_x+|\alpha_x|}\langle v \rangle^{n_v+N-|\alpha_x|} \partial_x^{\alpha_x} \partial_v^{\alpha_v} \Big](g) (t,\cdot , \cdot) \Big\|_{L^p(\R^3_x \times \R^3_v)} \leq C_0  \frac{\Lambda_t}{\langle t\rangle ^{\frac{4}{3}}} \E_N^n[g](t). \]
\end{Pro}
\begin{Rq}\label{Rqpoly}
As can be checked from the proof of Corollary \ref{Corenergy}, the dependence of $C_0$ on $\mathbf{\Lambda}$ is polynomial. More precisely, there exists $C_1[N,n,p]>0$ and $q[N] \in \mathbb{N}$ such that $C_0=C_1 \langle \mathbf{\Lambda} \rangle^q$.
\end{Rq}
\begin{proof}
Note first that, for all $t \in I$ and all $(x,v) \in \R^3_x \times \R^3_v$,
\begin{equation}\label{eq:Txv}
 \big| \overline{\T}_{\phi} \big( \langle v \rangle \big) \big| \lesssim  \frac{\Lambda_t}{\langle t \rangle^2} , \qquad \qquad  \big| \overline{\T}_{\phi} \big( \langle x \rangle \big) \big| \lesssim \Lambda_t \frac{\log \langle t+1 \rangle}{\langle t \rangle^{\frac{9}{5}}} \langle x \rangle.
 \end{equation}
Indeed, using the expression for $\overline{\T}_\phi$ provided by Lemma \ref{LemTbar}, we have
\begin{align*}
\overline{\T}_{\phi} \big( \langle v \rangle \big)& = -\lambda \nabla_x \phi (t,z) \cdot \frac{v}{\langle v \rangle}, \\
\overline{\T}_{\phi} \big( \langle x \rangle \big)& = \frac{\lambda t}{\langle t \rangle^2} \Big( \langle t \rangle^2\nabla_x \phi  (t,z) - \Phi (t,v) \Big) \cdot \frac{x}{\langle x \rangle} -\lambda \log \langle t \rangle \partial_t \Phi \cdot \frac{x}{\langle x \rangle}+   \log \langle t \rangle \nabla_v \Phi (t,v) \cdot \nabla_x \phi (t,z) \cdot \frac{x}{\langle x \rangle}.
\end{align*}
It remains to use \eqref{eq:phiassump1} and \eqref{eq:phiassump3}, which in particular imply, together with the mean value theorem,
\begin{equation}\label{eq:29}
\frac{t}{\langle t \rangle^2} \Big| \langle t \rangle^2\nabla_x \phi [f] (t,z) - \Phi (t,v) \Big| \lesssim t \Big| \nabla_x \phi [f] (t,z) -  \nabla_x \phi (t,tv) \Big|+ \Lambda_t\frac{\log\langle t+1 \rangle}{\langle t \rangle^{\frac{9}{5}}} \lesssim  \Lambda_t \frac{\log \langle t+1 \rangle }{\langle t \rangle^{\frac{9}{5}}} \langle x \rangle. 
 \end{equation}
As a consequence, we have
\[ \Big\| \big[ \, \overline{\T}_\phi , g_\alpha \big] (t,\cdot , \cdot) \Big\|_{L^p_{x,v}} \lesssim  \Lambda_t \frac{\log \langle t +1 \rangle}{\langle t \rangle^{\frac{9}{5}}}\big\| g_\alpha (t,\cdot , \cdot) \big\|_{L^p_{x,v}}+ \Big\|  \langle x \rangle^{n_x+|\alpha_x|} \langle v \rangle^{n_v+N-|\alpha_x|}  \big[ \, \overline{\T}_\phi , \partial_x^{\alpha_x} \partial_v^{\alpha_v} g \big] (t,\cdot , \cdot) \Big\|_{L^p_{x,v}}  . \]
According to Corollary \ref{Corenergy}, we are then lead to bound the quantities
\begin{align*}
 \mathfrak{I}_1 & \coloneqq   \bigg\| \frac{t \langle v \rangle}{\langle t \rangle^2 \langle x \rangle}  \Big| \langle t \rangle^2 t^{|\gamma|} \nabla_x \partial_{x}^{\gamma} \phi (t,z) -  \partial_{v}^{\gamma} \Phi (t,v) \Big| \cdot   g_\beta \bigg\|_{L^p(\R^3_x \times \R^3_v)} , \\
 \mathfrak{I}_3 & \coloneqq   \Big\| \langle v \rangle  \big| \partial_t \partial_v^\gamma \Phi (t,v) \big| \cdot   g_\beta \Big\|_{L^p(\R^3_x \times \R^3_v)}  ,  
 \end{align*}
 where $1 \leq |\gamma| \leq |\alpha|$ and $\big( |\gamma|-1 \big) +|\beta| \leq |\alpha|$, and
 \begin{align*}
  \mathfrak{I}_2 & \coloneqq     \Big\| \Big(1+ \big| \langle v \rangle \nabla_v \partial_v^{\xi} \Phi (t,v) \big| \Big) \Big| \langle t+|z| \rangle^{|\gamma|} \nabla_x \partial_x^{\gamma} \phi (t,z) \Big|  \cdot   g_\beta \Big\|_{L^p(\R^3_x \times \R^3_v)}  ,
 \end{align*}
 where $|\xi|, \, |\gamma|, \, |\beta| \leq |\alpha|$ and $|\xi|+|\gamma|+|\beta| \leq |\alpha|+1$. We start by dealing with $\mathfrak{I}_2$. 
 \begin{enumerate}[label = (\alph*)]
 \item Assume first that $|\xi|, \, |\gamma| \leq N-p_\infty$. Then, by \eqref{EQU:HPhi}, together with a Sobolev embedding, and \eqref{eq:phiassump1},
 \[ \mathfrak{I}_2 \lesssim  (1+\Lambda_t)\frac{ \Lambda_t}{\langle t \rangle^2}  \big\|  g_\beta (t,\cdot, \cdot) \big\|_{L^p(\R^3_x \times \R^3_v)}  \leq  (1+\Lambda_t ) \frac{ \Lambda_t}{\langle t \rangle^2} \E_N^n[g](t) . \]
 \item Assume now that $|\xi| \geq N+1-p_\infty$. Then, $|\gamma|+|\beta| \leq p_\infty$ so we can still estimate the force field pointwise. We use in addition the Sobolev embedding for Banach-valued functions $W^{p_\infty,p}(  \R^3_v,L^p(\R^3_x)) \hookrightarrow L^\infty ( \R^3_v,L^p(\R^3_x))$ to get $\int_x |g_\beta|^p \dr x \lesssim |\E_N^n[g]|^p$, so that
\[ \mathfrak{I}_2  \lesssim \frac{ \Lambda_t}{\langle t \rangle^2}  \big\|  g_\beta (t,\cdot, \cdot) \big\|_{L^p(\R^3_x \times \R^3_v)}+ \frac{ \Lambda_t}{\langle t \rangle^2} \big\| \langle v \rangle \nabla_v \partial_v^{\xi} \Phi (t,v)  \big\|_{L^p_v}  \bigg\|\int_{\R^3_x} \big| g_\beta  \big|^p (t,x,v) \dr x \bigg\|^{\frac{1}{p}}_{L^\infty(\R^3_v)}    \lesssim  (1+\Lambda_t)\frac{ \Lambda_t}{\langle t \rangle^2} \E_N^n[g](t) . \]
 \item \label{it3} Otherwise, we have $|\gamma| \geq N+1-p_\infty$ so that $|\xi|+|\beta| \leq p_\infty$ and $\langle v \rangle| \nabla_v \partial_v^{\xi} \Phi |(t,v) \lesssim \Lambda_t$. We start by performing the affine change of variables $y(x)= z(t,x,v)$. It yields
  \begin{align*}
  \mathfrak{I}_2 & \lesssim (1+\Lambda_t ) \Big\| \langle t+|y| \rangle^{|\gamma|} \nabla_x \partial_x^\gamma \phi (t,y) \Big\|_{L^p(\R^3_y)} \bigg\|  \int_{\R^3_v}  \big|  g_\beta \big|^p \big(t,y -tv-\lambda \Phi(t,v) \log \langle t \rangle ,v \big) \dr v \bigg\|_{L^\infty (\R^3_y)}^{\frac{1}{p}}.
  \end{align*}
We estimate the force field using \eqref{eq:phiassump2}. For the distribution function, according to Lemma \ref{Lemcdv}, we can perform the change of variables $x(v)= y -tv-\lambda \Phi(t,v) \log \langle t \rangle$ for $t \geq t_0[p,\mathbf{\Lambda}]$, where $t_0 \geq 1$. We obtain 
 \begin{align*}
  \mathfrak{I}_2 & \lesssim \frac{(1+\Lambda_t) \Lambda_t}{\langle t \rangle^{2-\frac{3}{p}}} \cdot \frac{1}{t^{\frac{3}{p}}} \bigg\|  \int_{\R^3_x}  \big|  g_\beta \big|^p(t,x,w) \dr x \bigg\|_{L^\infty (\R^3_w)}^{\frac{1}{p}} \lesssim \frac{(1+\Lambda_t) \Lambda_t}{t^2}  \E_N^n[g](t)   .
  \end{align*}
  Otherwise, we have $0 \leq t \leq t_0$. Then, by a Sobolev embedding, and using that $t$ is bounded above, we have
  \begin{align*}
  \mathfrak{I}_2 & \lesssim (1+ \Lambda_t ) \Lambda_t \langle t \rangle^{-2+\frac{3}{p}} \big\| g_\beta \big\|_{W^{p_\infty,p}_x L^p_v} \lesssim (1+\Lambda_t)  \Lambda_t \langle t \rangle^{-2} \, \E_N^n[g](t)  .
  \end{align*}
\end{enumerate}  
We conclude the analysis of $\mathfrak{I}_2$ using $1+\Lambda_t \leq 1+\mathbf{\Lambda} \log^a \langle t+1 \rangle$. We now focus on $\mathfrak{I}_1$, and we note that $\mathfrak{I}_1 \leq \mathfrak{I}_1^1+\mathfrak{I}_1^2$, where
\begin{align}
\mathfrak{I}_1^1 & \coloneqq   \bigg\| \frac{\langle v \rangle t}{ \langle x \rangle }  \Big| t^{|\gamma|} \nabla_x \partial_{x}^{\gamma} \phi (t,z) - t^{|\gamma|} \nabla_x \partial_{x}^{\gamma} \phi (t,tv) \Big| \cdot   g_\beta \bigg\|_{L^p(\R^3_x \times \R^3_v)}   ,  \label{eq:frakI11} \tag{$\mathfrak{I}_1^1$}\\
\mathfrak{I}_1^2 & \coloneqq \bigg\| \frac{\langle v \rangle t}{ \langle t \rangle^2 }  \Big| \langle t \rangle^2 t^{|\gamma|} \nabla_x \partial_{x}^{\gamma} \phi (t,tv) - \partial_{v}^{\gamma}\Phi (t,v) \Big| \cdot   g_\beta \bigg\|_{L^p(\R^3_x \times \R^3_v)} . \nonumber
\end{align}
If $|\gamma| \leq N-p_\infty$, we estimate pointwise the force field using \eqref{eq:phiassump3} to get
\[ \mathfrak{I}_1^2 \lesssim \frac{ \Lambda_t }{\langle t \rangle^{\frac{9}{5}}} \big\| g_\beta (t,\cdot , \cdot) \big\|_{L^p (\R^3_x \times \R^3_v)}  \lesssim   \frac{ \Lambda_t}{\langle t \rangle^{\frac{9}{5}}} \E_N^n[g](t) .  \]
Otherwise $|\gamma| \geq N+1-p_\infty$, so that $|\beta| \leq p_\infty$ and $\int_x |g_\beta|^p \dr x \lesssim |\E_N^n[g]|^p$. Using \eqref{eq:phiassump4}, we obtain
\[ \mathfrak{I}_1^2 \lesssim  \frac{t}{\langle t \rangle^2} \Big\| \langle v \rangle \Big( \langle t \rangle^2 t^{|\gamma|} \nabla_x \partial_{x}^{\gamma} \phi (t,tv) -  \partial_{v}^{\gamma} \Phi (t,v) \Big) \Big\|_{L^p(\R^3_v)}  \bigg\| \int_{\R^3_x} | g_\beta|^p  \dr x \bigg\|_{L^\infty(\R^3_v)}^{\frac{1}{p}}  \lesssim  \frac{ \Lambda_t}{\langle t \rangle^{\frac{9}{5}}} \E_N^n[g](t). \]

Next, we deal with $\mathfrak{I}_1^1$. To this end, we split the phase space $\R^3_x \times \R^3_v$ into $\mathbf{D}_t \sqcup \mathbf{D}_t^{\mathrm{c}}$, where $\mathbf{D}_t \coloneqq \{t_1+ | v| \leq \sqrt{t} \}$, and $t_1[p,\mathbf{\Lambda}] \geq 1$ is a constant to be fixed later. We then introduce
\begin{align*}
\mathfrak{Q}_1 & \coloneqq   \bigg\| \frac{\langle v \rangle t}{ \langle x \rangle}  \big|t^{|\gamma|} \nabla_x \partial_{x}^{\gamma} \phi (t,z) - t^{|\gamma|} \nabla_x \partial_{x}^{\gamma} \phi (t,tv) \big| \cdot   g_\beta \bigg\|_{L^p(\mathbf{D}_t)}   , \\
\mathfrak{Q}_2 & \coloneqq \bigg\| \frac{\langle v \rangle t}{ \langle x \rangle}  \big|t^{|\gamma|} \nabla_x \partial_{x}^{\gamma} \phi (t,tv) \big| \cdot   g_\beta \bigg\|_{L^p(\mathbf{D}_t^{\mathrm{c}})} , \\
\mathfrak{Q}_3 & \coloneqq \bigg\| \frac{\langle v \rangle t}{ \langle x \rangle}  \big|t^{|\gamma|} \nabla_x \partial_{x}^{\gamma} \phi (t,z) \big| \cdot   g_\beta \bigg\|_{L^p(\mathbf{D}_t^{\mathrm{c}})}  .
\end{align*}
\begin{itemize}
\item  \textit{The domain $\{ t_1+|v| \leq \sqrt{t} \}$.} The fundamental theorem of calculus yields
\[\mathfrak{Q}_1  =    \bigg\|\int_{\theta =0}^1 \frac{t \langle v \rangle|x+\lambda \Phi (t,v) \log \langle t \rangle|}{ \langle x \rangle} \Big| t^{|\gamma|} \nabla_x^2 \partial_{x}^{\gamma} \phi \Big(t,tv+\theta \big(x+\lambda \Phi (t,v) \log \langle t \rangle \big) \Big)\Big| \dr \theta  \cdot   g_\beta \bigg\|_{L^p(\mathbf{D}_t)} . \]
Then, we note that on $\mathbf{D}_t$, $t \langle v \rangle|x+\lambda \Phi (t,v) \log \langle t \rangle| \lesssim t^{3/2} \langle x \rangle \log^{a+1} \langle t+1 \rangle$. If $|\gamma| \leq N-p_\infty$, we estimate pointwise the force field using \eqref{eq:phiassump3} and Remark \ref{Rq:phiassump2} to get
\[ \mathfrak{Q}_1 \lesssim  \Lambda_t\frac{\log^{a+1} \langle t+1 \rangle}{\langle t \rangle^{\frac{3}{2}}} \big\| g_\beta (t,\cdot , \cdot) \big\|_{L^p (\R^3_x \times \R^3_v)}  \leq   \Lambda_t\frac{\log^{a+1} \langle t+1 \rangle}{\langle t \rangle^{\frac{3}{2}}} \big\| g_\beta (t,\cdot , \cdot) \big\|_{L^p (\R^3_x \times \R^3_v)} .  \]
We now consider the case $|\gamma | \geq N+1- p_\infty$, so that $|\beta| \leq p_\infty$. The analysis shares some similarities with the one carried out in \ref{it3}. If $t_1$ is chosen sufficiently large, one can apply Lemma \ref{Lemcdv} to perform the change of variables $y(v)=tv+\theta (x+\lambda \Phi (t,v) \log \langle t \rangle )$. It yields, using in addition \eqref{eq:phiassump2},
\begin{align*}
 \mathfrak{I}_1^1 & \lesssim t^{\frac{3}{2}} \log^{a+1} \langle t+1 \rangle \cdot \frac{1}{t^{\frac{3}{p}}} \Big\|t^{|\gamma|} \nabla_x^2 \partial_x^\gamma \phi (t,y) \Big\|_{L^p(\R^3_y)} \bigg\|  \int_{\R^3_x}  \big|  g_\beta \big|^p(t,x,w) \dr v \bigg\|_{L^\infty (\R^3_w)}^{\frac{1}{p}}  \lesssim \Lambda_t \frac{\log^{a+1} \langle t+1 \rangle}{\langle t \rangle^{\frac{3}{2}}} \E_N^n[g](t)  .
 \end{align*}
 
\item  \textit{The integral $\mathfrak{Q}_2$, over the domain $\{ t_1+|v| > \sqrt{t} \}$.}  Assume first that $|\gamma| \leq N+1-p_\infty$, so that \eqref{eq:phiassump1} and Remark \ref{Rq:phiassump2} give
\[ \mathfrak{Q}_2 \lesssim \Lambda_t \bigg\| \frac{\langle v \rangle t}{ \langle t+t|v| \rangle^2}   g_\beta \bigg\|_{L^p(\mathbf{D}_t^{\mathrm{c}})} \leq \Lambda_t \big\|  \langle t+t|v| \rangle^{-1}  \big\|_{L^\infty(\mathbf{D}_t^{\mathrm{c}})} \big\| g_\beta \big\|_{L^p(\R^3_x \times \R^3_v)} \lesssim \frac{\Lambda_t}{\langle t \rangle^{\frac{3}{2}}} \E_N^n[g](t) , \]
where we use $t \langle v \rangle \leq t+t|v|$ and that $\langle tv \rangle \gtrsim \langle t \rangle^{3/2}$ on $\mathbf{D}_t^{\mathrm{c}}$. If $|\gamma| \geq N+2-p_\infty$, we have $|\beta| \leq p_\infty-1$ so that we can use again $\int_x |g_\beta|^p \dr x \lesssim |\E_N^n[g](t)|^p$. In view of the domain of integration, it yields
\[ \mathfrak{Q}_2 \lesssim \frac{t^2}{\langle t \rangle^{\frac{|\gamma|}{2}-1}} \Big\|t^{|\gamma|-1} \langle v \rangle^{|\gamma|-1} \nabla_x \partial_x^\gamma \phi (t,tv)  \Big\|_{L^p (\R^3_v)} \bigg\| \int_{\R^3_x} |g_\beta |^p (t,x,v) \dr x \bigg\|_{L^\infty (\R^3_v)}^{\frac{1}{p}}   . \]
Note now that $|\gamma| \geq 3$. Therefore, performing the change of variables $x=tv$ and using then \eqref{eq:phiassump2}, we obtain
\[ \mathfrak{Q}_2 \lesssim \frac{t^2}{\langle t \rangle^{\frac{|\gamma|}{2}-1} t^{\frac{3}{p}}}  \Big\|(t+|x|)^{|\gamma|-1} \nabla_x \partial_x^\gamma \phi (t,x)  \Big\|_{L^p (\R^3_x)} \E_N^n[g](t) \lesssim \Lambda_t   \frac{t^{2-\frac{3}{p}}}{\langle t \rangle^{\frac{|\gamma|}{2}-1} \langle t \rangle^{3-\frac{3}{p}}}  \E_N^n[g](t) \leq \frac{\Lambda_t}{\langle t \rangle^{\frac{3}{2}}} \E_N^n[g](t). \]

\item  \textit{The integral $\mathfrak{Q}_3$.} Assume first $|\gamma| \leq N+1-p_\infty$, so that by \eqref{eq:phiassump1} and Remark \ref{Rq:phiassump2},
\[ \mathfrak{Q}_3 \lesssim \Lambda_t \bigg\| \frac{ \langle v \rangle t}{ \langle x \rangle \langle t+|z| \rangle^2} \bigg\|_{L^\infty(\mathbf{D}_t^{\mathrm{c}})}  \big\| g_\beta (t, \cdot , \cdot) \big\|_{L^p(\R^3_x \times \R^3_v)} \lesssim  \frac{\Lambda_t}{\langle t \rangle^{\frac{3}{2}}} \E_N^n [g](t)  , \]
provided that, for all $t \in I$, $x \in \R^3_x$ and $|v| \geq \sqrt{t}-t_1$,
\begin{equation}\label{eq:techni}
\frac{ \langle v \rangle t}{ \langle x \rangle \langle t+|z| \rangle^2} \lesssim \frac{1}{\langle t \rangle^{\frac{3}{2}}}.
\end{equation}
As $t \langle v \rangle \leq t+t|v|$, $tv=z-x-\lambda \Phi (t,v) \log \langle t \rangle$ and $|\Phi (t,v)| \lesssim \log^a \langle t+1 \rangle$, we have
\[ \frac{ \langle v \rangle t}{ \langle x \rangle \langle t+|z| \rangle^2} \lesssim \frac{\langle t \rangle+\langle z \rangle}{\langle x \rangle \langle t+|z| \rangle^2}+\frac{1}{\langle t+|z| \rangle^2} \lesssim  \frac{1}{\langle x \rangle  \langle t+|z| \rangle}+\frac{1}{\langle t \rangle^2} . \]
If $|z| \geq t^{\frac{3}{2}}/2$, the right hand side of the previous inequality is bounded by $ \langle t \rangle^{-\frac{3}{2}}$. Otherwise $|z| \leq t^{\frac{3}{2}}/2$, so
\[ |x| \geq  \big| tv - \lambda \Phi (t,v) \log \langle t \rangle \big| - |z| \geq t^{\frac{3}{2}}-t_1t- \mathbf{\Lambda}  \log^{a+1} \langle t+1 \rangle - \frac{1}{2}t^{\frac{3}{2}},\]
which allows to deduce \eqref{eq:techni}. Finally, assume that $|\gamma| \geq N+2-p_\infty$, so that $|\beta| \leq p_\infty-1$. Here, we proceed as in \ref{it3} to get
 \begin{align*}
  \mathfrak{Q}_3 & \lesssim   \bigg\| \frac{\langle v \rangle t^4}{\langle x \rangle \langle t+|z| \rangle^2} \bigg\|_{L^\infty (\mathbf{D}_t^{\mathrm{c}})} \big\| \langle t+|y| \rangle^2  t^{|\gamma|-3} \nabla_x \partial_x^\gamma \phi (t,y) \big\|_{L^p_y} \bigg\|\int_{\R^3_v}  \big|  g_\beta \big|^p \big(t,y-tv-\lambda \Phi (t,v) \log \langle t \rangle,v \big) \dr v \bigg\|_{L^\infty_y }^{\frac{1}{p}}  . 
  \end{align*}
  We bound the first factor by $t^{3/2}$ using \eqref{eq:techni}. The second one is controlled by $\Lambda_t \langle t \rangle^{-3+\frac{3}{p}}$ through \eqref{eq:phiassump2}. For the last factor, we separate between the case $t \leq t_0$ and $t \geq t_0$ to bound it by $\langle t \rangle^{-3/p} \E_N^n[g](t)$.
\end{itemize}

We finally deal with $\mathfrak{I}_3$. We then fix $1 \leq |\gamma| \leq N$ and $|\beta| \leq |\alpha|+1-|\gamma|$. Assume first that $|\gamma| \leq N-p_\infty$, so that, by \eqref{eq:phiassump3},
\[ \mathfrak{I}_3 \lesssim \frac{\Lambda_t}{\langle t \rangle^{\frac{9}{5}}} \|g_\beta\|_{L^p(\R^3_x \times \R^3_v)} \leq \frac{\Lambda_t}{\langle t \rangle^{\frac{9}{5}}} \E_N^n[g](t). \]
Otherwise, $|\gamma| \geq N+1-p_\infty$ and $|\beta| \leq p_\infty$. Using \eqref{eq:phiassump4} and $\int_x |g_\beta |^p \dr x \lesssim |\E_N^n [g]|^p$, we have
\[ \mathfrak{I}_3 \lesssim  \big\| \langle v \rangle  \partial_t  \partial_v^{\gamma} \Phi (t,v)  \big\|_{L^p( \R^3_v)} \E_N^n[g](t) \lesssim \frac{\Lambda_t}{\langle t \rangle^{\frac{9}{5}}} \E_N^n[g](t). \] 
\end{proof}

As a consequence of the previous analysis, one can estimate $\partial_t g$ and its derivatives up to order $N-1$. For global solutions, it will allow to obtain convergence to a limiting distribution function. From now on, we assume that $a=0$, that is $\|\Lambda_t \|_{L^\infty(I)}\leq \mathbf{\Lambda}$.

\begin{Cor}\label{CorscattGeneral}
Recall that $(\phi, \Phi) \in \mathcal{W}^N(I,\mathbf{\Lambda})$, let $g \colon I \times \R^3_x \times \R^3_v \to \R$ be a sufficiently regular solution to $\overline{\T}_\phi (g)=0$, and let $n=(n_x,n_v) \in \mathbb{Z}^2$. For any $|\alpha| \leq N-1$, we have
\[ \forall \, t \in I, \qquad  \big\| \partial_t g_\alpha \big\|_{L^p(\R^3_x \times \R^3_v)} \lesssim  \frac{\mathbf{\Lambda} }{\langle t \rangle^{\frac{4}{3}}} \E_N^n[g](t). \]
\end{Cor}
\begin{proof}
We use the expression of $\overline{\T}_\phi$, given by Lemma \ref{LemTbar}, to rewrite $\partial_t g_\alpha$. As $|\langle v \rangle \nabla_v \Phi |(t,v) \lesssim \mathbf{\Lambda}$, it yields
\begin{align*}
 | \partial_t g_\alpha |  \lesssim \! \big| \overline{\T}_\phi \big( g_\alpha \big) \big|+\frac{t}{\langle t \rangle^2} \Big| \langle t \rangle^2\nabla_x \phi  (t,z) - \Phi (t,v) \Big|  | \nabla_x  g_\alpha |+   | \nabla_x \phi  | (t,z ) \bigg( \frac{\log \langle t \rangle}{\langle v \rangle} | \nabla_x g_\alpha | +  | \nabla_v g_\alpha | \bigg)+\log \langle t \rangle | \partial_t \Phi || \nabla_x g_\alpha | .
\end{align*}
Using \eqref{eq:phiassump3} to bound $\partial_t \Phi$, and \eqref{eq:phiassump1}, we get
\begin{align*}
 | \partial_t g_\alpha | \! \lesssim \big| \overline{\T}_\phi \big( g_\alpha \big) \big|+\! \sum_{|\beta|\leq N} \! \frac{\langle v \rangle t}{ \langle x \rangle  \langle t \rangle^2} \Big| \langle t \rangle^2\nabla_x \phi  (t,z)-\langle t \rangle^2 \nabla_x \phi (t,tv)+\langle t \rangle^2 \nabla_x \phi (t,tv) - \Phi (t,v) \Big|  |  g_\beta |+ \mathbf{\Lambda} \frac{ \log \! \langle t+1 \rangle}{\langle t \rangle^{\frac{9}{5}}} |g_\beta|.
\end{align*}
We now observe that bound derived for \eqref{eq:frakI11} applies in the case $|\gamma|=0$. Thus, using also \eqref{eq:phiassump3}, we get
\[ \big\| \partial_t g_\alpha \big\|_{L^p ( \R^3_x \times \R^3_v) } \lesssim \big\| \overline{\T}_\phi \big( g_\alpha \big) \big\|_{L^p ( \R^3_x \times \R^3_v) } + \mathbf{\Lambda} \frac{ \log \langle t +1 \rangle}{\langle t \rangle^{\frac{3}{2}}} \E_N^n[g](t)    . \]
It remains to apply Proposition \ref{ProboundednessGeneral}.
\end{proof}

We will combine the last estimate with the next result to obtain scattering.

\begin{Lem}\label{Lemforscat}
Let $T \geq 0$ and $\mathfrak{g} \colon [T,+\infty[ \times \R^3_x \times \R^3_v \to \R$ be a function satisfying $\| \partial_t \mathfrak{g} \|_{L^p_{x,v}} \in L^1_t$ and $\mathfrak{g}(T,\cdot , \cdot) \in L^p_{x,v}$. Then, there exists $\mathfrak{g}_\infty \in L^p(\R^3_x \times \R^3_v)$ such that
\[ \forall \, t \in [T,+\infty[, \qquad   \big\| \mathfrak{g}(t,\cdot , \cdot)- \mathfrak{g}_\infty  \big\|_{L^p(\R^3_x \times \R^3_v)}  \leq \int_{\tau=t}^{+\infty} \big\| \partial_t \mathfrak{g} (\tau , \cdot , \cdot) \big\|_{L^p(\R^3_x \times \R^3_v)} \dr \tau . \]
\end{Lem}
\begin{proof}
By the Minkowski integral inequality, we have, for all $t_2 \geq t_1 \geq T$,
\[   \big\| \mathfrak{g}(t_2,\cdot , \cdot)- \mathfrak{g}(t_1,\cdot , \cdot)  \big\|_{L^p(\R^3_x \times \R^3_v)}  = \bigg\| \int_{\tau=t_1}^{t_2} \partial_t \mathfrak{g} (\tau,\cdot , \cdot) \dr \tau \bigg\|_{L^p(\R^3_x \times \R^3_v)} \leq \int_{\tau=t_1}^{+\infty} \big\| \partial_t \mathfrak{g} (\tau , \cdot , \cdot) \big\|_{L^p_{x,v}} \dr \tau . \]
As the right hand side goes to $0$ as $t_1 \to +\infty$, the map $t \mapsto g(t,\cdot , \cdot)$ is Cauchy in $L^p_{x,v}$.
\end{proof}

\subsection{Convergence of characteristics and consequences}

Let $T \geq 0$ and assume, throughout this section, that $I=[T,+\infty[$. Recall that $(\phi, \Phi) \in \mathcal{W}^N(I,\mathbf{\Lambda})$, for some constant $\mathbf{\Lambda}$. We will be interested in the asymptotic properties of the characteristics associated with the transport operator 
\[\overline{\T}_\phi= \partial_t+a_x (t,x,v) \cdot \nabla_x-\lambda \nabla_x \phi \big(t,z(t,x,v) \big) \cdot \nabla_v,    \]
 where
\[  a_x(t,x,v) \coloneqq \frac{\lambda t}{\langle t \rangle^2} \Big( t^2\nabla_x \phi  (t,z) - \Phi (t,v) \Big)   + \log \langle t \rangle \nabla_v \Phi \cdot \nabla_x \phi (t,z )  -\lambda \log \langle t \rangle \partial_t \Phi  . \]

\begin{Def}
Let $t, \, t_0 \geq T$ and $(x,v) \in \R^3_x \times \R^3_v$. We denote by $\varphi$ the characteristic flow of $\overline{\T}_\phi$,
\[ \varphi_{t,t_0}(x,v) \coloneqq \big( \mathbf{X}(t,t_0,x,v),\mathbf{V}(t,t_0,x,v) \big), \]
where $(\mathbf{X},\mathbf{V})(t_0,t_0,x,v)=(x,v)$ and, using the dot to denote differentiation with respect to the first variable $t$,
\begin{align*}
\dot{\mathbf{X}}(t,t_0,x,v) &=a_x\big(t,\mathbf{X}(t,t_0,x,v),\mathbf{V}(t,t_0,x,v) \big)  , \\
\dot{\mathbf{V}}(t,t_0,x,v)&=-\lambda \nabla_x \phi \big(t,z(t,\mathbf{X}(t,t_0,x,v),\mathbf{V}(t,t_0,x,v)) \big) .
\end{align*}
\end{Def}

We start by proving that the characteristics converge.

\begin{Lem}\label{Lemcharac}
Let $t_0 \geq T$ and $(x,v) \in \R^3_x \times \R^3_v$. Then, there exists $C[\mathbf{\Lambda}]>0$ such that, for all $t \geq T$,
\[ \big| \mathbf{V}(t,t_0,x,v)-v \big| \leq C, \qquad \qquad  \big| \mathbf{X}(t,t_0,x,v) \big| \leq \langle x \rangle e^C   . \]
Moreover, there exists a continuous map $\varphi_{\infty,t_0}(x,v) =(\mathbf{X},\mathbf{V})(\infty,t_0,x,v)$ such that, for all $t \geq T$,
\[  \big|\mathbf{V}(t,t_0,x,v)- \mathbf{V}(\infty,t_0,x,v) \big| \leq C \langle t \rangle^{-1}, \qquad \qquad \big|\mathbf{X}(t,t_0,x,v)- \mathbf{X}(\infty,t_0,x,v) \big| \leq \langle x \rangle \langle t \rangle^{-\frac{1}{2}} Ce^C. \]
\end{Lem}
\begin{proof}
Recall from \eqref{eq:Txv} that
\[ \overline{\T}_\phi (|v|) \lesssim \langle t\rangle^{-2}, \qquad \qquad \overline{\T}_\phi (|x|) \lesssim \langle t \rangle^{-\frac{3}{2}} \langle x \rangle. \]
It remains to apply Duhamel's principle and, for $\mathbf{X}$, the Grönwall lemma. The convergences follow.
\end{proof}

We now show the convergence of the flow, as $t_0 \to +\infty$.

\begin{Pro}\label{Procharac}
Let $t \geq T$. There exists $\varphi_{t,\infty}(x,v) \in C^0(\R^3_x \times \R^3_v,\R^3_x \times \R^3_v)$ such that
\begin{enumerate}[ label = (\Roman*)]
\item \label{objet1} For any compact subset $K \Subset \R^3_x \times \R^3_v$, we have $\| \varphi_{t,t_0} - \varphi_{t,\infty} \|_{L^{\infty}(K)}  \lesssim C \langle t_0 \rangle^{-1/3}$, where $C[K,\mathbf{\Lambda}] >0$.
\item \label{objet1bis} For any compact subset $K \Subset \R^3_x \times \R^3_v$, there exists a compact $K'$ such that $\varphi_{t,t_0}^{-1}(K) \subset K'$ for all $t_0 \in [T,+\infty]$.
\item \label{objet2} $\varphi_{t,\infty}$ is a homeomorphism, and $\varphi_{t,\infty}^{-1}=\varphi_{\infty,t}$. Besides, $\varphi_{t,\infty} \in W^{1,\infty}_{\mathrm{loc}} (\R^3_x \times \R^3_v,\R^3_x \times \R^3_v)$, and $\det ( \dr \varphi_{t,\infty}) \equiv 1$.
\end{enumerate}
\end{Pro}
\begin{proof}
Let us fix $t$ and introduce $h(t_0,x,v) \coloneqq \varphi_{t,t_0}(x,v)$, so that $\overline{\T}_\phi(h)=0$. Note now that for $n \coloneqq (-N-3,-N-3)$, we have $\E_N^n [h](t) <+ \infty$. Recall from Definition \ref{Defgalpha} the notation $h_\alpha$. As a consequence, by applying the energy inequality of Proposition \ref{Proenergy}, to $h_\alpha$ for any $|\alpha| \leq N$, we obtain
\[ \E_N^n [h](t_0)-\E_N^n[h](t) \lesssim \bigg| \int_{\tau=t}^{t_0} \big\| \overline{\T} (h_\alpha) (\tau,x,v) \big\|_{L^p(\R^3_x \times \R^3_v)} \dr \tau \bigg| \lesssim \bigg| \int_{\tau=t}^{t_0} \frac{1}{\langle \tau \rangle^{\frac{4}{3}}} \E_N^n[h](\tau) \dr \tau \bigg|, \]
where, in the last step, we used Proposition \ref{ProboundednessGeneral}. Boundedness follows from the Grönwall inequality. Then, Corollary \ref{CorscattGeneral}, together with Lemma \ref{Lemforscat}, yields the existence of $\varphi_{t,\infty}$ such that
\[ \sup_{t_0 \geq T} \E_N^n [h](t_0) < +\infty, \qquad \qquad \lim_{t_0 \to +\infty} \E_{N-1}^n \big[ \varphi_{t,t_0}-\varphi_{t,\infty} \big] =0 .\]
Unfortunately, if $N$ or $p$ are not sufficiently large, we cannot conclude the proof without additional estimates. We claim now that, for any $|\alpha| \leq 1$,
\[\forall \, (\tau ,x,v) \in [T,+\infty[ \times \R^3_x \times \R^3_v, \quad \big| \T_\phi (h_\alpha ) \big|(\tau,x,v) \leq \mathbf{C} \langle t \rangle^{-\frac{4}{3}} \sum_{|\beta| \leq 1} \big| h_\beta \big|(\tau,x,v), \qquad \qquad \mathbf{C}[\mathbf{\Lambda}]>0 . \]
For this, we proceed as in the proof of Proposition \ref{ProboundednessGeneral}. 
\begin{itemize}
\item We bound $\T_\phi (h_\alpha)$ using first Corollary \ref{Corenergy}, applied with $|\alpha|=1$, and \eqref{eq:Txv}.
\item Then, in any term given by Corollary \ref{Corenergy}, we can estimate the force field and the coefficients $\Phi$ pointwise through \eqref{EQU:HPhi}, \eqref{eq:phiassump3} and \eqref{eq:phiassump1}, exactly as in the proof of Proposition \ref{ProboundednessGeneral}.
\end{itemize}
Hence, we obtain from Duhamel's principle and the Grönwall inequality, that
\[ \forall \, t , \, \tau \geq T, \quad \big\| h_\alpha  \big\|_{W^{1,\infty}(\R^3_x \times \R^3_v)} \leq \mathbf{C}, \qquad \qquad \mathbf{C}[\mathbf{\Lambda}] >0  . \]
Proceeding as in Corollary \ref{CorscattGeneral}, we get
\[ \forall \, t , \, \tau \geq T, \quad \big\| \partial_\tau \langle x \rangle^{-N-3} \langle v \rangle^{-N-3} \varphi_{t,\tau} (x,v)  \big\|_{L^{\infty}(\R^3_x \times \R^3_v)} \leq \mathbf{C} \langle \tau \rangle^{-\frac{4}{3}}, \qquad \qquad \mathbf{C}[\mathbf{\Lambda}] >0  ,\]
which allows us to obtain \ref{objet1} and that $(x,v) \mapsto \varphi_{t,\infty}(x,v)$ is continuous. As $\varphi_{t,t_0} \circ\varphi_{t_0,t}=\varphi_{t_0,t} \circ \varphi_{t,t_0}=\mathrm{id}$ for all $t_0 \geq T$, and that both $\varphi_{t,t_0}$ and $\varphi_{t_0,t}$ are continuous for all $T \leq t_0 \leq +\infty$, we get the first part of \ref{objet2}. The Banach-Alaoglu theorem yields $\varphi_{t,\infty} \in W^{1,\infty}_{(x,v),\mathrm{loc}}$. To conclude the proof of \ref{objet2}, we combine 
\[ \forall \, t_0 \in [T,+\infty[, \quad \det \big(\dr  \varphi_{t,t_0} \big) \equiv 1, \qquad \qquad \forall \, T \leq t_0 \leq + \infty, \quad \big\| \nabla_{x,v} \, \varphi_{t,t_0} \big\|_{L^\infty (K)} \leq \mathbf{C}[\mathbf{\Lambda},K], \]
for any compact $K \Subset \R^3_x \times \R^3_v$, with the strong convergence in $L^p_{(x,v),\mathrm{loc}}$, $\varphi_{t,t_0} \to \varphi_{t,\infty}$, as $t_0 \to +\infty$. Finally, \ref{objet1bis} follows from Lemma \ref{Lemcharac}.
\end{proof}

We are now able to show that sufficiently regular solutions to $\T_\phi (f)=0$ converge strongly in $\E_N^n[\cdot]$, along suitable corrections to the characteristics of the free transport operator.

\begin{Cor}\label{Corscatstrong}
Assume that $g$ satisfies $\overline{\T}_\phi(g)=0$, and that $\sup_{t \geq T} \E_N^n [g] (t) < + \infty$ for some $n \in \mathbb{Z}^2$. Then, there exists a function $f_\infty : \R^3_x \times \R^3_v \to \R$ such that
\[ \lim_{t \to + \infty} \E_N^n \big[g (t , \cdot , \cdot) - f_\infty \big] (t) =0. \]
Moreover, we have for all $t \geq T$,
\[ \E_{N-1}^n \big[ g (t , \cdot , \cdot) - f_\infty \big] \lesssim  \frac{\mathbf{\Lambda}}{\langle t \rangle^{\frac{1}{3}}} \sup_{ \tau \geq T} \E_N^n [g] ( \tau ). \]
\end{Cor}
\begin{proof}
Let $\mathbf{B} \coloneqq \sup_{t \geq T} \E_N^n[g](t)$. According to Duhamel's principle, we have, for $|\alpha| \leq N$,
\[ g_\alpha(t,x,v) = g_\alpha(T, \cdot , \cdot) \circ \varphi_{T,t}(x,v) +H_\alpha(t,x,v), \qquad \qquad H_\alpha(t,x,v) \coloneqq \int_{\tau = T}^t \overline{\T}_\phi (g_\alpha)(\tau, \cdot , \cdot) \circ \varphi_{\tau,t}(x,v) \dr \tau   . \]
Note now that $H_\alpha (t, \cdot , \cdot)$ converges in $L^p(\R^3_x \times \R^3_v)$ as $t \to +\infty$. Indeed, one can show that it is a Cauchy sequence from the Minkowski integral inequality, and
\[   \int_{\tau=t}^{+ \infty} \Big\| \overline{\T}_\phi (g_\alpha)(\tau, \cdot , \cdot) \circ \varphi_{\tau,t}(x,v) \Big\|_{L^p (\R^3_x \times \R^3_v)} \dr \tau =  \int_{\tau=t}^{+ \infty} \Big\| \overline{\T}_\phi (g_\alpha)(\tau, x , v) \Big\|_{L^p (\R^3_x \times \R^3_v)} \dr \tau \leq \frac{C[N,n,p,\mathbf{\Lambda}]}{\langle t \rangle^{\frac{1}{3}}} \mathbf{B} , \]
where we used first that $(x,v) \mapsto \varphi_{\tau,t}(x,v)$ is a $C^1$-diffeomorphism of Jacobian determinant equal to $1$, and then Proposition \ref{ProboundednessGeneral}. For the first term, let $\epsilon >0$ and consider $h \in C^1_c(\R^3_x \times \R^3_v)$ such that $\|h-g_\alpha (T, \cdot , \cdot ) \|_{L^p_{x,v}} \leq \epsilon$. By \ref{objet2}, $\varphi_{\tau,\infty}$ is a Lipschitz volume-preserving homeomorphism, so
\[ \big\| g_\alpha(T, \cdot , \cdot) \circ \varphi_{T,t} - g_\alpha(T, \cdot , \cdot) \circ \varphi_{T,\infty} \big\|_{L^p_{x,v}} \leq  2\big\| g_\alpha(T, \cdot , \cdot)- h \big\|_{L^p_{x,v}}+ \big\| h(T, \cdot , \cdot) \circ \varphi_{T,t} - h(T, \cdot , \cdot) \circ \varphi_{T,\infty} \big\|_{L^p_{x,v}}   . \] 
We bound the second term using \ref{objet1}--\ref{objet1bis} and the mean value theorem. As a consequence, denoting the limit of $H_\alpha (t,\cdot,\cdot)$ by $H_\alpha^\infty$, we finally get
\begin{equation}\label{eq:forconti}
 \big\| g_\alpha (t,\cdot , \cdot)- g_\alpha(T, \cdot , \cdot) \circ \varphi_{T,\infty} - H^\infty_\alpha \big\|_{L^p_{x,v}} \leq 2 \epsilon + C \langle t \rangle^{-1/3}, \qquad  C\big[N,n,p,\mathbf{\Lambda},\mathbf{B},\| \nabla h \|_{L^\infty}, \mathrm{supp} \, h \big] >0 . 
 \end{equation}
The convergence with a rate for the weaker norm follows from Corollary \ref{CorscattGeneral} and Lemma \ref{Lemforscat}.
\end{proof}

We can finally prove well-posedness for the asymptotic Cauchy problem associated with $\overline{\T}_\phi$.

\begin{Pro}\label{ProAsympCau}
Let $ f_\infty \in L^p( \R^3_x \times \R^3_v )$. Then, there exists a unique $g \in L^\infty \big( [T,+\infty[, L^p(  \R^3_x \times \R^3_v ) \big)$ such that
\begin{equation}\label{AsympCauch}
 \overline{\T}_\phi (g)=0, \qquad \qquad g(t,\cdot , \cdot) \to f_\infty \quad \text{as $t \to + \infty$}. \tag{ACpb}
 \end{equation}
Moreover, if $\E_N^n[f_\infty] <+\infty$ for some $n \in \mathbb{Z}^2$, there exists $\mathbf{C}_0[N,n,p, \mathbf{\Lambda}] >0$ such that,
\[ \forall \, t \geq T , \qquad \quad \E_N^n[g](t) \leq  \E_N^n[f_\infty] e^{\mathbf{C}_0 \mathbf{\Lambda} \langle T \rangle^{-\frac{1}{3}}}. \]
\end{Pro}
\begin{proof}
Let $g \in L^\infty \big( [T,+\infty[, L^p(  \R^3_x \times \R^3_v ) \big)$ be a weak solution to \eqref{AsympCauch}. Then, we write
\[ g(t,x,v) = g\big( t_0, \varphi_{t_0,t}(x,v) \big)-g\big( t_0, \varphi_{\infty,t}(x,v) \big)+g\big( t_0, \varphi_{\infty,t}(x,v) \big).\]
By assumption, the last term on the right hand side converges to $f_\infty \big(\varphi_{\infty,t}(x,v) \big)$ in $L^p_{x,v}$ as $t_0 \to +\infty$. In view of Lemma \ref{Lemcharac} and Proposition \ref{Procharac}, we can argue as in the proof of Corollary \ref{Corscatstrong}, when we proved that $g_\alpha(T,  \varphi_{T,t} ) \to g_\alpha(T,  \varphi_{T,\infty})$ as $t \to +\infty$ in $L^p_{x,v}$, to get
\[  \big\|g\big( t_0, \varphi_{t_0,t}(x,v) \big)-g\big( t_0, \varphi_{\infty,t}(x,v) \big) \big\|_{L^p(\R^3_x \times \R^3_v)} \to 0 \qquad \text{as $t_0 \to + \infty$.} \]
We then deduce, by passing to the limit $t_0 \to +\infty$, that $g(t, \cdot , \cdot)=f_\infty \circ \varphi_{\infty, t}$. Existence holds since Proposition \ref{Procharac} ensures that $f_\infty \circ \varphi_{\infty, t}$ is a weak solution to \eqref{AsympCauch} with the stated regularity.

Finally, if $\E_N^n[f_\infty] <+\infty$, we conclude by standard arguments as follows.
\begin{itemize}
\item Assume first that $f_\infty$ is compactly supported, so that, by Lemma \ref{Lemcharac}, the same holds for $g_k(t,\cdot,\cdot) \coloneqq f_\infty \circ \varphi_{k,t}$. In particular, $\E_N^n[g_k](t) <+ \infty$ for all $t \geq T$. We can then apply the energy estimate of Proposition \ref{Proenergy} to $[g_k]_\alpha$, for any $|\alpha| \leq N$, and Proposition \ref{ProboundednessGeneral}, to get
\[ \E_N^n[g_k](\tau) -\E_N^n[f_\infty] \leq p   \sum_{|\alpha| \leq N} \bigg|\int_{t=k}^{\tau} \Big\|  \overline{\T}_\phi \big( [g_k]_\alpha \big) \Big\|_{L^p(\R^3_x \times \R^3_v)} \dr t \bigg| \leq \bigg|\int_{t = k}^{\tau} pC_0  \frac{\mathbf{\Lambda}}{\langle \tau \rangle^{\frac{4}{3}}} \E_N^n[g_k](\tau) \dr \tau \bigg| . \] 
Applying Grönwall's inequality yields a bound that is uniform in $k$. The Banach-Alaoglu theorem finally provides the desired energy bound for $g$.
\item Otherwise, let $(\chi_k)_{k \geq 1}$ be a family of cutoff functions such that $\chi_k (x,v) =1$ for $|x|+|v| \leq k$ and $\chi (x,v) =0$ for $|x|+|v| \geq 2k$. Then, we apply the previous result to $g_k$, the unique solution to \eqref{AsympCauch} arising from the asymptotic data $f_\infty \chi_k$. It remains to pass to the limit $k \to +\infty$.
\end{itemize}
\end{proof}

\subsection{Existence of a solution to \eqref{VP} with asymptotic data}\label{Secmodwave}

We now apply the results of this section to construct a solution $f$ to the Vlasov--Poisson system \eqref{VP} corresponding to the scattering data $f_\infty \in \mathcal{B}$, thereby proving the second part of Theorem \ref{Th1}. More precisely, we work in the coordinate system $(t,z,v)$, where
\[ z(t,x,v) \coloneqq x+tv+\lambda \nabla_v \phi_\infty [f_\infty] (v) \log \langle t \rangle, \qquad \qquad \Delta_v \phi_\infty [f_\infty]= \int_{\R^3_x} f_\infty (x, \cdot ) \dr x. \]
We then show that the asymptotic Cauchy problem associated with \eqref{VP},
\begin{equation}\label{eq:VPasymp}
\overline{\T}_{\phi [f]}(g)=0, \qquad \qquad  \lim_{t \to +\infty} g(t,x , x ) \to f_\infty (x,v),
\end{equation}
where the expression of $\overline{\T}_{\phi [f]}$ is given in Lemma \ref{LemTbar}, admits a unique solution belonging to $L^\infty ([T,+\infty[,\mathcal{B})$, for some $T \geq 0$. The function $f$ given by $f(t,x,v)=g(t, x-tv-\lambda \nabla_v \phi_\infty [f_\infty](v) \log(t),v)$ will then be a solution to \eqref{VP}, that is $\T_{\phi [f]}(f)=0$.

We start by defining a sequence of approximate solutions as follows.
\begin{itemize}
\item $g_0$ is the constant in time solution $g_0 (t,x,v) \coloneqq f_\infty (x,v) $.
\item For $k \in \mathbb{N}$, we define $g_{k+1}$ as the the unique solution to the asymptotic Cauchy problem
\[ \overline{\T}_{\phi [f_k]} (g_{k+1})  =0, \qquad \qquad \lim_{t \to + \infty} g_{k+1}(t,x,v)=f_\infty (x,v), \]
where $f_k \big( t,z(t,x,v),v \big)=g_k(t,x,v)$, so that $\T_{\phi [f_k]}(f_{k+1})=0$. We recall that $\phi [f_k]=\nabla \Delta^{-1}\rho [f_k]$.
\end{itemize}
We start by proving boundedness for the sequence $(g_k)_{k \in \mathbb{N}}$.
\begin{Pro}\label{Proinduc}
Let $T \geq 0$, $\varepsilon \geq 0$ such that $\mathbb{E}_N[f_\infty] = \varepsilon$, and let $C[N,p,\varepsilon]>0$ be a constant to be chosen sufficiently large. Then, there exist $T_0[N,p,\varepsilon] \geq 0$ and $\varepsilon_0[N,p] >0$, such that the following holds. If $T \geq T_0$, or if $T=0$ and $\varepsilon \leq \varepsilon_0$, then, for any $k \geq 0$,
\begin{enumerate}[ label = (\Alph*)]
\item \label{OB:1} The function $g_k$ is well-defined and satisfies $\mathbb{E}_N [g_k](t) \leq 2 \varepsilon$ for all $t \geq T$,
\item \label{OB:2} The potential $\phi [f_k]$ and $\nabla_v \phi_\infty[f_\infty]$ satisfy $(\phi [f_k],\nabla_v \phi_\infty[f_\infty]) \in \mathcal{W}^N \big([T,+\infty[,C\varepsilon \big)$.
\end{enumerate}
\end{Pro}
\begin{proof}
We perform an induction. As $g_0(t , \cdot , \cdot) = f_\infty$ and $\E_N[f_\infty]=\varepsilon$, \ref{OB:1} holds for $k=0$. Next, if \ref{OB:1} holds for some $k \in \mathbb{N}$, then \ref{OB:2} is satisfied at order $k$ by Corollary \ref{Rqclassteleo}. Finally, consider $k \in \mathbb{N}$ such that \ref{OB:2} holds. Then, by applying Proposition \ref{ProAsympCau}, we obtain that $g_k$ is well-defined and, for all $t \geq T$,
 \[ \mathbb{E}_N [g_{k+1}](t) \leq  \E_N[f_\infty] e^{\mathbf{C}_0 \varepsilon \langle T \rangle^{-1/3}} , \qquad \qquad \mathbf{C}_0[N,p,\varepsilon] >0 . \]
 We then bound the right hand side by $2 \varepsilon$ by either choosing $T_0$ sufficiently large compared to $\mathbf{C}_0 \varepsilon$, or, if $T_0=0$, by choosing $\varepsilon$ sufficiently small compared to $\mathbf{C}_0[N,p,1]$.
\end{proof}

We are now able to show that the sequence of approximate solutions $(g_k)_{k \in \mathbb{N}}$ converges, as $k \to + \infty$. We recall from Definition \ref{Defnorm} the expression of the energy norm $\E_N^{(n_x,n_v)}[\cdot]$, and that $\E_N[\cdot ]= \E_N^{(7,7)} [ \cdot ]$.

\begin{Pro}\label{Prolocalwp}
Let $f_\infty \colon \R^3_x \times \R^3_v \to \R$ and $\varepsilon \geq 0$ be such that $\E_N[f_\infty]= \varepsilon$. Then, if $T \geq T_0$, or if $T=0$ and $\varepsilon \leq \varepsilon_0$, there exists a solution $g$ to \eqref{eq:VPasymp}, which satisfies
\begin{itemize}
\item  $\mathbb{E}_N [g](t) \leq 2 \varepsilon$ for all $t \geq T$, and $(\phi [f],\nabla_v \phi_\infty [f_\infty]) \in \mathcal{W}^N \big( [T,+\infty[,C\varepsilon \big)$ with $C[N,p]>0$,
\item  $\E_{N-1}\big[ g(t,\cdot , \cdot) - f_\infty \big] \lesssim \varepsilon \langle t \rangle^{-\frac{1}{3}} $ for all $t \geq T$, and $\E_{N}\big[ g(t,\cdot , \cdot) - f_\infty \big] \to 0$, as $t \to +\infty$.
\end{itemize}
\end{Pro}
\begin{proof}
Recall that $(g_k)_{k \in \mathbb{N}}$ is bounded in the norm $\sup_{t \geq T} \E_N[ \cdot ]$ by \ref{OB:1}. Moreover, using $\overline{\T}_{\phi [f_k]}(g_{k+1})=0$, together with \eqref{eq:phiassump1} and \eqref{eq:phiassump3}, we obtain that $\sup_{t \geq T} \| \partial_t g_k\|_{L^p_{x,v}}$ is bounded uniformly in $k \in \mathbb{N}$. By the Aubin–Lions–Simon compactness lemma and a diagonal extraction argument, there exists a subsequence $(g_{k_i})$ such that,
\[ \forall \, T_1 \in [T,+\infty[, \qquad \lim_{i \to + \infty} \sup_{T \leq t \leq T_1} \, \E_{N-1}^{(6,6)} \big[ g-g_{k_i} \big](t) =0. \]
Then, we apply Banach-Alaoglu theorem to get, from \ref{OB:1}, that $\E_N[g](t) \leq 2 \varepsilon$ for all $t \geq T$. Moreover, one can either use again the Banach-Alaoglu theorem and \ref{OB:2}, or use Proposition \ref{Proforcefield} and Corollary \ref{Corphiphiinfty}, to obtain $(\phi [f], \nabla_v \phi_\infty [f_\infty]) \in \mathcal{W}^N \big( [T,+\infty[,C\epsilon \big)$. Convergence in this norm allows us to pass to the limit in the equation $\overline{\T}_{\phi [f_{n-1}]} ( g_{n})=0$, so that $\overline{\T}_{\phi[f]}(g)=0$. Finally, the convergence statements follow from Corollary \ref{Corscatstrong}.
\end{proof}

We will conclude the proof of the converse statement in Theorem \ref{Th1}, that is show uniqueness for the asymptotic Cauchy problem \eqref{eq:VPasymp}, in Corollary \ref{Corunique} below. We finish by a propagation result for higher weighted Sobolev norms.

\begin{Cor}\label{Corhighnorm}
Assume that $\E_N^n[f_\infty] <+\infty$ for $n \in \mathbb{Z}^2$. Then, there exists $B[N,p,n,\varepsilon]>0$ such that,
\[\forall \, t \geq T, \qquad \qquad \E_N^n[g](t) \leq B \E_N^n[f_\infty], \qquad  \E_{N-1}^n\big[ g(t,\cdot , \cdot) - f_\infty \big] \leq B \E_N^n[f_\infty] \, \langle t \rangle^{-\frac{1}{3}} .\]
Moreover, we have $\E_{N}^n [ g(t,\cdot , \cdot) - f_\infty  ] \to 0$ as $t \to +\infty$.
\end{Cor}
\begin{proof}
This follows from Proposition \ref{ProAsympCau} and Corollary \ref{Corscatstrong}.
\end{proof}

\subsection{Small data solutions}\label{Subsecsmall}

We now establish the first part of Theorem \ref{Th1}, that is global existence and modified scattering for the small data solutions to the Vlasov-Poisson system. Let $\varepsilon \geq 0$ and $f$ be a solution to \eqref{VP} such that $\E_N[f](0)=\varepsilon$. We define further 
\[ \Phi_f (t,v) \coloneqq t^2 \nabla_x \phi [f] (t,tv), \qquad \qquad  g(t,x,v) \coloneqq f \big( t,x+tv+\lambda \Phi_f(t,v) \log \langle t \rangle , v \big), \]
so that $\overline{\T}_{\phi [f]}(g)=0$. We consider $T \in \R_+ \cup \{ + \infty \}$, the maximal time such that
\begin{equation}\label{BAsmall}
  \big( \phi [f] , \Phi_f \big) \in \mathcal{W}^N \big( [0,T[,C_{\mathrm{boot}}\varepsilon \big) , \tag{BA1}
  \end{equation}
  for some constant $C_{\mathrm{boot}} [N,p]>0$ which will be fixed sufficiently large later. By local well-posedness, $T>0$. Let us now improve, if $\varepsilon$ is small enough, the bootstrap assumption \eqref{BAsmall}. Proposition \ref{ProboundednessGeneral}, together with the energy estimate of Proposition \ref{Proenergy}, applied to $g_\alpha$ for any $|\alpha| \leq N$, yields
\[ \E_N^n[g](\tau) -\E_N^n[g](0) \leq p\int_{t=0}^{\tau} C_0[N,p,n,C_{\mathrm{boot}}\varepsilon]   \frac{ C_{\mathrm{boot}}\varepsilon }{\langle t\rangle ^{\frac{4}{3}}} \E_N^n[g](t) \dr t, \]
for any $n \in \mathbb{Z}^2$ such that $\E_N^n[g](0)<+\infty$. By the Grönwall inequality, we obtain, if $\varepsilon$ is small enough, that
\begin{equation}\label{eq:boundednessforn}
 \forall \, t \in  [0,T[, \qquad \E_N^n[g](\tau) \leq \E_N^n[g](0) e^{4pC_0  C_{\mathrm{boot}}\varepsilon}. 
 \end{equation}
Assume now that $\varepsilon$ is small enough so that $C_{\mathrm{boot}} \varepsilon \leq 1$. In particular, $C_0$ does not depend on $(C_{\mathrm{boot}},\varepsilon)$. Applying \eqref{eq:boundednessforn} to $n=(7,7)$, we obtain, if $\varepsilon$ is small enough, that $\E_N[g](t) \leq 2 \varepsilon$ for all $t \in [0,T[$. Then, by applying Corollary \ref{Rqclass} for $\Phi=\Phi_f$, and using again $C_{\mathrm{boot}} \varepsilon \leq 1$, we get 
  \[ \big( \phi [f] , \Phi_f \big) \in \mathcal{W}^N \big( [0,T[,\mathbf{C}\varepsilon \big) , \qquad \mathbf{C}[N,p]>0 . \]
  We then choose $C_{\mathrm{boot}} =2 \mathbf{C} $. This improves \eqref{BAsmall} and implies that $T=+\infty$. Moreover, we obtain that $\E_N[g](t) \leq 2 \varepsilon$ and \eqref{eq:boundednessforn} hold for all $t \geq 0$. We finally show \ref{itunan}--\ref{itdaou} in Theorem \ref{Th1}. Let $n \in \mathbb{Z}^2$ such that $\E_N^n [g](0)<+\infty$. 
\begin{itemize}
\item According to Corollary \ref{Corscatstrong}, there exists $f_\infty \colon \R^3_x \times \R^3_v \to \R$ such that $\E_N^n [g(t,\cdot , \cdot) -f_\infty] \to 0$, as $t \to +\infty$. Note in particular that $\E_N[f_\infty] \leq 2 \varepsilon$. However, this is not enough for our purposes since we would like to prove that $f$ converges along slightly different corrections to the linear characteristics.
\item In this step, we prove the existence of an effective force field. We recall that $\Phi_f (t,v) \coloneqq t^2 \nabla_x \phi [f] (t,tv)$, and that $\phi_\infty [f_\infty]$ satisfies $\Delta _v \phi_\infty[f_\infty]=\int_x f_\infty (x,\cdot ), \dr x$. Then, combining Corollary \ref{Corphiphiinfty}, the bound on $\E_N[g]$, and the previous step, we get the following estimates. For all $t \geq 0$,
\begin{align}\label{kevatalenn:unan1}
 \Big\| \langle v \rangle \Big(  \partial_v^\kappa \Phi_f (t,v)-\nabla_v  \phi_\infty \big[ \partial_v^\kappa f_\infty \big](v)\Big) \Big\|_{L^\infty (\R^3_v)} & \lesssim  \varepsilon \frac{\log^{2N+7} \langle t +1 \rangle}{\langle t \rangle} , \qquad \qquad |\kappa| \leq N-p_\infty .
 \end{align}
\item Finally, consider the function $G$ given by
\[ G(t,x,v) \coloneqq f \big( t,x+tv+\lambda \nabla_v \phi_\infty [f_\infty] (v) \log (t) , v \big) = g \big(t,x+\lambda \big[ \nabla_v \phi_\infty[f_\infty](v)-\Phi_f(t,v) \big] \log \langle t \rangle,v \big).\]
Then, by the mean value theorem in $x$ and \eqref{kevatalenn:unan1}, we have $G(t,\cdot , \cdot) \to f_\infty$ in $L^p_{x,v}$ as $ t \to + \infty$. Using $\mathcal{E}_N[\Phi_f](t) \lesssim \varepsilon$ and $\E_N[g](t) \leq 2 \varepsilon$ for all $t \geq 0$, and arguing exactly as in the proof of Corollary \ref{Rqclassteleo}, we obtain
\[ \big( \phi [f] , \nabla_v \phi_\infty [f_\infty ] \big) \in \mathcal{W}^N \big( \R_+,C \varepsilon \big), \qquad \qquad C[N,p]>0. \] 
Since $G \in  L^\infty (\R_+,L^p_{x,v})$, we apply Proposition \ref{ProAsympCau} to bound $ \E_N^n[G] $ in terms of $\E_N^n [f_\infty]$, which yields \ref{itunan} in Theorem \ref{Th1}. Corollary \ref{Corscatstrong} then gives the scattering result \ref{itdaou}. Finally, this also proves Remark \ref{RqTh1}.
\end{itemize}

\section{Global existence and modified scattering for perturbations of dispersive solutions}\label{Sec3}

We aim to construct a solution $f$ to the Vlasov-Poisson system, with initial data close to the ones of a global and dispersive solution $\mathbf{f}$ to \eqref{VP}. We will then study the function $g-\mathbf{g}$, with
\[ g(t,x,v) \coloneqq  f \big( t,z_f(t,x,v),v \big) , \qquad \qquad \mathbf{g}(t,x,v) \coloneqq \mathbf{f} \big( t,z_\mathbf{f}(t,x,v) , v \big) , \]
where the coordinates $z_f$ and $z_{\mathbf{f}}$ are given by
\[ z_f(t,x,v) \coloneqq x+tv+\lambda \Phi_f(t,v) \log\langle t \rangle, \qquad \qquad z_{\mathbf{f}} (t,x,v) \coloneqq x+tv+\lambda \Phi_{\mathbf{f}}(t,v) \log \langle t \rangle , \]
and the coefficients $\Phi_f$ and $\Phi_{\mathbf{f}}$ will be of one of the following two forms.
\begin{enumerate}[label= (\textbf{\roman*})]
\item \label{traouunan} For the forward problem, we will use $\Phi_{\mathfrak{f}} (t,v) \coloneqq t^2 \nabla_x \phi [\mathfrak{f} ] (t,tv)$.
\item \label{traoudaou} For the backward problem, we will set $\Phi_{\mathfrak{f}} (v) \coloneqq \nabla_v \phi_{\infty}[\mathfrak{f}_\infty](v)$, for $\mathfrak{f} (t,z_{\mathfrak{f}}(t,x,v),v) \to \mathfrak{f}_\infty (x,v)$ as $t \to + \infty$.
\end{enumerate}

For all this section, we fix an order of regularity $N \geq 2 p_\infty$, $n=(n_x,n_v) \in \mathbb{N}^2$ with $n_x, \, n_v \geq 7$, a constant $\mathbf{\Lambda} >0$, and a global solution $\mathbf{f} \colon \R_+ \times \R^3_x \times \R^3_v \to \R$ to \eqref{VP}. Assume that, for all $t \in \R_+$,
\begin{equation}\label{EQU:assumpPhibff}
 \E_N^n [\mathbf{g}](t) \leq \mathbf{\Lambda}, \qquad \qquad \qquad \big( \phi [\mathbf{f} ], \Phi_{\mathbf{f}} \big) \in \mathcal{W}^N \big( \R_+, \mathbf{\Lambda} \big)  , \tag{H$\mathbf{f}$}
\end{equation}
where $\Phi_{\mathbf{f}}$ is either given by \ref{traouunan} or by \ref{traoudaou}.

\subsection{Preparatories}\label{Secprepsubsec}

We complement here the results of Section \ref{Secprep}. We fix an interval of times $I \subset \R_+$, and assume that, in addition to \eqref{EQU:assumpPhibff}, the following estimates hold,
\begin{equation}\label{EQUA:Hfbff}
 \big(\phi [f], \Phi_f \big) \in \mathcal{W}^{N-1} \big( I, 2 \mathbf{\Lambda}  \big), \qquad \qquad \big( \phi [f-\mathbf{f}],\Phi_{f-\mathbf{f}} \big) \in \mathcal{W}^{N-1} \big( I,\omega \big), \tag{H$f$}
 \end{equation}
for a nonnegative function $\omega \colon t \mapsto \omega_t$. We start by computing $\overline{\T}_{\phi[f]}(g - \mathbf{g})$.

\begin{Lem}\label{LemTphi}
The function $g-\mathbf{g}$ satisfies
\begin{align*}
 \overline{\T}_{\phi [f]} \big(g-\mathbf{g} \big) = & -\frac{\lambda t}{ \langle t \rangle^2}\Big( \langle t \rangle^2\nabla_x \phi [f-\mathbf{f}] (t,z_f) - \Phi_{f-\mathbf{f}} \Big) \cdot \nabla_x \mathbf{g} - \lambda t \big( \nabla_x \phi [\mathbf{f}] (t,z_f) - \nabla_x \phi [\mathbf{f}] (t,z_{\mathbf{f}}) \big) \cdot \nabla_x \mathbf{g} \\
 &  +\lambda \Big( \nabla_x \phi [f-\mathbf{f}] (t,z_f)+\nabla_x \phi [\mathbf{f}] (t,z_f) -\nabla_x \phi [\mathbf{f}] (t,z_{\mathbf{f}}) \Big)\cdot \nabla_v \mathbf{g}  + \lambda \log \langle t \rangle \partial_t \Phi_{f-\mathbf{f}} \cdot \nabla_x \mathbf{g} \\
 & - \lambda \log \langle t \rangle \nabla_v \Phi_{f-\mathbf{f}} \cdot \nabla_x \phi[f](t,z_f) \cdot \nabla_x \mathbf{g}-\lambda \log \langle t \rangle \nabla_v \Phi_{\mathbf{f}}  \cdot  \nabla_x \phi[f - \mathbf{f}](t,z_f) \cdot \nabla_x \mathbf{g} \\
 & -\lambda \log \langle t \rangle \nabla_v \Phi_{\mathbf{f}} \cdot \Big( \nabla_x \phi [\mathbf{f}](t,z_f) -\nabla_x \phi[\mathbf{f}](t,z_{\mathbf{f}}) \Big) \cdot \nabla_x \mathbf{g} .
 \end{align*}
 \end{Lem}
\begin{proof}
We use that $\overline{\T}_{\phi [f]}(g)=\overline{\T}_{\phi [\mathbf{f}]}(\mathbf{g})=0$, and that $\phi [\cdot]$ as well as $\psi \mapsto \Phi_\psi$ are linear.
\end{proof}

We now show a result similar to Corollary \ref{Corenergy}.

\begin{Cor}\label{CorboundTh}
For any $|\alpha| \leq N-1$, we can bound $\langle x \rangle^{|\alpha_x|} \langle v \rangle^{-|\alpha_x|} \big| \partial_x^{\alpha_x} \partial_v^{\alpha_v} \overline{\T}_{\phi [f]} (g-\mathbf{g}) \big|$ by a linear combination of the terms
\begin{alignat*}{2}
& \frac{\langle v \rangle t}{\langle x \rangle \langle t \rangle^2} \Big| \langle t \rangle^2 t^{|\gamma|} \nabla_x \partial_x^\gamma \phi [f-\mathbf{f}] (t,z_f) - \partial_v^\gamma \Phi_{f-\mathbf{f}} \Big| \cdot \frac{\langle x \rangle^{|\beta_x|}}{\langle v \rangle^{|\beta_x|}}\big|  \partial_x^{\beta_x} \partial_v^{\beta_v} \mathbf{g} \big|, \qquad \qquad \qquad && |\gamma| + |\beta| = |\alpha|+1, \quad |\gamma| \leq |\alpha|, \\
& \log \langle t \rangle \Big| \langle v \rangle \partial_t \partial_v^\gamma \Phi_{f-\mathbf{f}} \Big| \cdot  \frac{\langle x \rangle^{|\beta_x|}}{\langle v \rangle^{|\beta_x|}}\big|  \partial_x^{\beta_x} \partial_v^{\beta_v} \mathbf{g} \big| ,         && |\gamma| + |\beta| = |\alpha|+1, \quad |\gamma| \leq |\alpha|,
\end{alignat*}
and
\begin{align*}
& \log^{N} \langle t +1\rangle \Big( 1+ |\langle v \rangle \nabla_v \partial_v^\xi \Phi_{\mathfrak{f}} | \Big)   \cdot \Big| \langle t+|z_f| \rangle^{|\gamma|} \partial_{x^k} \partial_x^\gamma \phi[f - \mathbf{f}](t,z_f) \Big| \cdot  \frac{\langle x \rangle^{|\beta_x|}}{\langle v \rangle^{|\beta_x|}}\big| \partial_x^{\beta_x} \partial_v^{\beta_v}  \mathbf{g} \big| ,  \\
& \log^{N} \langle t +1\rangle \Big( 1+ | \langle v \rangle \nabla_v \partial_v^\xi \Phi_{\mathbf{f}} | \Big)  \cdot \Big| \langle t+|z_f| \rangle^{|\gamma|+1}\Big( \partial_{x^k} \phi \big[ \partial_x^\gamma \mathbf{f} \big] (t,z_f) - \partial_{x^k} \phi \big[ \partial_x^\gamma \mathbf{f} \big] (t,z_{\mathbf{f}}) \Big) \Big| \cdot   \frac{\langle x \rangle^{|\beta_x|}}{\langle v \rangle^{|\beta_x|}} \big| \partial_x^{\beta_x} \partial_v^{\beta_v} \mathbf{g} \big|,    \\
& \log^{N} \langle t +1\rangle\Big( 1+ |\langle v \rangle \nabla_v \partial_v^\xi \Phi_{\mathfrak{f}} | \Big) \big|\langle v \rangle \nabla_v \partial_v^\kappa \Phi_{f-\mathbf{f}} \big| \cdot \Big|(t+|z_f|)^{|\gamma|} \nabla_x \phi \big[ \partial_x^\gamma f \big](t,z_f) \Big| \cdot  \frac{\langle x \rangle^{|\beta_x|}}{\langle v \rangle^{|\beta_x|}} \big|  \partial_x^{\beta_x} \partial_v^{\beta_v} \mathbf{g} \big| ,  
\end{align*}
where $|\beta| \geq 1$, $|\xi|+|\kappa|+|\gamma|+\big( |\beta| -1 \big) \leq |\alpha|$, and $\mathfrak{f} \in \{ f, \mathbf{f} \}$.
\end{Cor}
\begin{proof}
We begin by stating a result analogous to Corollary \ref{CorComm}. Note first that for a function $\psi : \R^3 \to \R$,
\[ \partial_{x^i} \big[ \psi (z_f) \big]= \partial_{x^i} \psi (z_f), \qquad \qquad \partial_{v^i} \big[ \psi (z_f) \big]=t \partial_{x^i} \psi (z_f)+\lambda \log \langle t \rangle \partial_{v^i} \Phi_f \cdot \nabla_x \psi (z_f) , \]
and
\[ \partial_{v^i} \big[ \psi (z_f)-\psi(z_\mathbf{f} ) \big] = t \big( \partial_{x^i} \psi (z_f)- \partial_{x^i} \psi (z_\mathbf{f}) \big) +\lambda \log \langle t \rangle \partial_{v^i} \Phi_{f-\mathbf{f}} \cdot \nabla_x \psi (z_f)+ \lambda \log \langle t \rangle \partial_{v^i} \Phi_{\mathbf{f}} \cdot \big(  \nabla_x \psi (z_f)-\nabla_x \psi (z_\mathbf{f}) \big) . \]
We recall from Definition \ref{DefpolynomPhi} the notation $P_{r,q}(\Phi_f)$. We use the obvious extension $P_{r,q}(\Phi_f,\Phi_{\mathbf{f}})$ for products involving both $\Phi_f$ and $\Phi_{\mathbf{f}}$. Hence, by Lemma \ref{LemTphi} and an induction, we can write $\partial_x^{\alpha_x} \partial_v^{\alpha_v} \overline{\T}_{\phi [f]} (g - \mathbf{g})$ as a linear combination of the terms
\begin{alignat*}{2}
& \frac{t}{\langle t \rangle^2} \Big( \langle t \rangle^2 t^{|\gamma|} \nabla_x \partial_x^\gamma \phi [f-\mathbf{f}] (t,z_f) - \partial_v^\gamma \Phi_{f-\mathbf{f}} \Big) \cdot \nabla_x \partial_x^{\alpha_x} \partial_v^{\beta_v} \mathbf{g}, \qquad \qquad \qquad && \gamma + \beta_v = \alpha_v, \\
& \log \langle t \rangle \partial_t \partial_v^\gamma \Phi_{f-\mathbf{f}} \cdot \nabla_x \partial_x^{\alpha_x} \partial_v^{\beta_v} \mathbf{g} ,         && \gamma + \beta_v = \alpha_v ,
\end{alignat*}
and, with the convention that $\partial_{x^i}^{\delta} \partial_{v^i}^{1-\delta}= \partial_{x^i}$ if $\delta=1$ and $\partial_{x^i}^{\delta} \partial_{v^i}^{1-\delta} = \partial_{v^i}$ for $\delta=0$, 
\begin{align}
& \log^{r} \langle t \rangle P_{r,q} (\Phi_f,\Phi_{\mathbf{f}}) \cdot t^{\delta +|\alpha_v|-|\beta_v|-q-r} \partial_{x^k} \partial_x^\gamma \phi[f - \mathbf{f}](t,z_f) \cdot \partial_{x^i}^{\delta} \partial_{v^i}^{1-\delta}  \partial_x^{\beta_x} \partial_v^{\beta_v}  \mathbf{g}, \qquad \delta \leq r \leq |\gamma|+1, \label{term:4} \\
& \log^r \langle t \rangle P_{r,q} (\Phi_{\mathbf{f}}) \cdot t^{\delta+|\alpha_v|-|\beta_v|-q-r} \big( \partial_{x^k} \phi \big[ \partial_x^\gamma \mathbf{f} \big] (t,z_f) - \partial_{x^k} \phi \big[ \partial_x^\gamma \mathbf{f} \big] (t,z_{\mathbf{f}}) \big) \cdot \partial_{x^i}^\delta \partial_{v^i}^{1-\delta} \partial_x^{\beta_x} \partial_v^{\beta_v} \mathbf{g}, \qquad r \leq |\gamma|, \label{term:6}   \\
& \log^{r+1} \langle t \rangle P_{r,q} (\Phi_f,\Phi_{\mathbf{f}}) \partial_{v^k} \partial_v^\kappa \Phi_{f-\mathbf{f}} \cdot t^{\delta+|\alpha_v|-|\beta_v|-|\kappa|-q-r-1} \nabla_x \phi \big[ \partial_x^\gamma f \big](t,z_f) \cdot \partial_{x^i}^\delta \partial_{v^i}^{1-\delta} \partial_x^{\beta_x} \partial_v^{\beta_v} \mathbf{g} ,  \qquad r \leq |\gamma|, \label{term:5}
\end{align}
where $q+|\kappa|+|\gamma|+|\beta| = |\alpha|$, $|\beta_x| \leq |\alpha_x|$, $|\beta_v| \leq |\alpha_v|$. Moreover, the exponent of $t$ is nonnegative. In other words, in \eqref{term:4}--\eqref{term:6} and \eqref{term:5}, we respectively have
\begin{align}
 \delta+|\alpha_v|-|\beta_v|-q-r & =  |\gamma|+|\beta_x|-|\alpha_x|+\delta -r \geq 0 , \label{eq:exponent0} \\
 \delta+|\alpha_v|-|\beta_v|-|\kappa|-q-r-1 &= |\gamma|+|\beta_x|-|\alpha_x|+\delta-r-1 \geq 0 . \label{eq:exponent1}
\end{align}
As for Corollary \ref{Corenergy}, the first two family of terms can be easily handled. For the last three ones, we proceed as for the terms of \eqref{type2} in the proof of Corollary \ref{Corenergy}. Recall first that $\Phi_f$ and $\Phi_{\mathbf{f}}$ satisfy \eqref{EQU:HPhi} for $\Lambda_t=2\mathbf{\Lambda}$. Since $N \geq 2p_\infty$, the Sobolev embedding $W^{p_\infty,p}(\R^3_v) \hookrightarrow L^\infty (\R^3_v)$, together with \eqref{EQU:assumpPhibff}--\eqref{EQUA:Hfbff}, implies that, if $r \geq 1$, there exists $|\xi| \leq q$ and $\mathfrak{f} \in \{ f , \mathbf{f} \}$ such that
\begin{equation}\label{eq:Phbff}
 \big| P_{r,q} (\Phi_f,\Phi_{\mathbf{f}}) \big|(t,v) \lesssim \mathbf{\Lambda}^{r-1} \langle v \rangle^{-1} \big| \langle v \rangle \nabla_v \partial_v^\xi \Phi_{\mathfrak{f}} \big|(t,v) . 
 \end{equation}
Then, we write $\partial_{x^i}^{\delta} \partial_{v^i}^{1-\delta}  \partial_x^{\beta_x} \partial_v^{\beta_v} = \partial_{x}^{\beta_x'} \partial_v^{\beta'_v}$ and we consider two cases.
\begin{itemize}
\item If $|\beta_x'|= |\alpha_x|+1$, which occurs if $|\beta_x|=|\alpha_x|$ and $\delta=1$. Then, in view of \eqref{eq:exponent0}--\eqref{eq:exponent1}, we focus on
\[  t^{|\gamma|+1-r} \langle x \rangle^{|\alpha_x|} \langle v \rangle^{-|\alpha_x|} = \frac{\langle v \rangle}{\langle x \rangle} t^{|\gamma|+1-r} \langle x \rangle^{|\beta'_x|} \langle v \rangle^{-|\beta'_x|}  . \]
Note then that for the terms \eqref{term:4}, we have $t^{|\gamma|+1-r} \leq \langle t \rangle^{|\gamma|}$ by \eqref{eq:exponent0} and $r \geq \delta$. We further absorb the weight $\langle v \rangle$ through \eqref{eq:Phbff}.

For the terms \eqref{term:5}, since we have an extra factor $t^{-1}$, we need to show $t^{|\gamma|-r} \leq \langle t \rangle^{|\gamma|}$, which follows from \eqref{eq:exponent1}. The weight $\langle v \rangle$ is absorbed by the factor $\partial_{v^k} \partial_v^\kappa \Phi_{f-\mathbf{f}}$.

Finally, for the terms \eqref{term:6}, we have $t^{|\gamma|+1-r} \leq t\langle t \rangle^{|\gamma|}$ since $r \leq |\gamma|$. Then, we use
\[ t\langle v \rangle \leq t+ |tv| \leq t+|z_f|+|x|+\| \Phi_{f} \|_{L^\infty_{t,v}} \log \langle t \rangle \lesssim \langle x \rangle \langle t+|z_f| \rangle . \]
\item Otherwise, $\delta +|\beta_x|=|\beta'_x| \leq |\alpha_x|$ and we recall from \eqref{eq:x} that $\langle x \rangle \lesssim \langle t+|z_f| \rangle \, \langle v \rangle$, so that
\[ t^{|\gamma|+|\beta_x|-|\alpha_x|+\delta -r}  \langle x \rangle^{|\alpha_x|} \langle v \rangle^{-|\alpha_x|} \lesssim t^{|\gamma|+|\beta_x'|-|\alpha_x| -r} \langle t+|z_f| \rangle^{|\alpha_x|-|\beta_x'|}  \langle x \rangle^{|\beta'_x|} \langle v \rangle^{-|\beta'_x|} . \]
We then get the desired bounds by using \eqref{eq:exponent0}--\eqref{eq:exponent1}, to deal with small times $0 \leq t \leq 1$, and $t \leq \langle t+|z_f| \rangle$.

\end{itemize} 
\end{proof}

The next corollary will be useful in the perspective of performing energy estimates. We recall from Definition \ref{Defgalpha} the notation $g_\beta$.

\begin{Cor}\label{CorforenergypertGeneral}
 There exists $\mathfrak{C}[N,n,p,\mathbf{\Lambda}]>0$ such that, for any $|\alpha_x|+|\alpha_v| \leq N-1$,
 \[\forall \, t \in I, \qquad \qquad \Big\| \langle x \rangle^{n_x+|\alpha_x|} \langle v \rangle^{n_v+N-|\alpha_x|} \partial_x^{\alpha_x} \partial_v^{\alpha_v} \overline{\T}_{\phi [f]} (g-\mathbf{g}) (t,\cdot , \cdot) \Big\|_{L^p ( \R^3_x \times \R^3_v) } \leq  \mathfrak{C}    \frac{\omega_t }{\langle t\rangle^{\frac{4}{3}}} . \]
\end{Cor}
\begin{proof}
In view of Corollary \ref{CorboundTh}, we are lead to bound the following five quantities,
\begin{alignat*}{2}
\mathfrak{K}_1 & \coloneqq \bigg\| \frac{ \langle v \rangle t}{\langle x \rangle \langle t \rangle^2} \Big| \langle t \rangle^2 t^{|\gamma|} \nabla_x \partial_x^\gamma \phi [f-\mathbf{f}] (t,z_f) - \partial_v^\gamma \Phi_{f-\mathbf{f}} \Big| \cdot    \mathbf{g}_\beta   \bigg\|_{L^p(\R^3_x \times \R^3_v)}   , \qquad \qquad &&|\gamma| + |\beta| = N, \quad |\gamma| \leq N-1, \\
\mathfrak{K}_3 & \coloneqq \Big\| \langle v \rangle \partial_t \partial_v^\gamma \Phi_{f-\mathbf{f}}(t,v) \cdot \mathbf{g}_\beta  \Big\|_{L^p(\R^3_x \times \R^3_v)}  , \qquad \qquad &&|\gamma| + |\beta| = N, \quad |\gamma| \leq N-1,
\end{alignat*}
and
\begin{align*}
\mathfrak{K}_2 & \coloneqq \Big\| \big( 1+ \big|\langle v \rangle\nabla_v \partial_v^\xi \Phi_{\mathfrak{f}} \big| \big)   \cdot \Big| \langle t+|z_f| \rangle^{|\gamma|} \partial_{x^k} \partial_x^\gamma \phi[f - \mathbf{f}](t,z_f) \Big| \cdot \mathbf{g}_\beta  \Big\|_{L^p(\R^3_x \times \R^3_v)}  , \\
\mathfrak{K}_4 & \coloneqq \Big\| \big( 1+ | \langle v \rangle \nabla_v \partial_v^\xi \Phi_{\mathbf{f}} | \big)  \cdot \Big| \langle t+|z_f| \rangle^{|\gamma|+1}\big( \partial_{x^k} \phi \big[ \partial_x^\gamma \mathbf{f} \big] (t,z_f) - \partial_{x^k} \phi \big[ \partial_x^\gamma \mathbf{f} \big] (t,z_{\mathbf{f}}) \big) \Big| \cdot  \mathbf{g}_\beta  \Big\|_{L^p(\R^3_x \times \R^3_v)}  , \\
\mathfrak{K}_5 & \coloneqq \Big\| \big( 1+ |\langle v \rangle \nabla_v \partial_v^\xi \Phi_{\mathfrak{f}} | \big) \big|\langle v \rangle \nabla_v \partial_v^\kappa \Phi_{f-\mathbf{f}} \big| \cdot \Big|\langle t+|z_f| \rangle^{|\gamma|} \nabla_x \phi \big[ \partial_x^\gamma f \big](t,z_f) \Big| \cdot \mathbf{g}_\beta  \Big\|_{L^p(\R^3_x \times \R^3_v)}  .
\end{align*}
where $|\beta| \geq 1$, $|\xi|+|\kappa|+|\gamma|+\big( |\beta| -1 \big) \leq N-1$ and $\mathfrak{f} \in \{ f, \mathbf{f} \}$.

By proceeding as in the proof of Proposition \ref{ProboundednessGeneral} for the analysis of $\mathfrak{I}_1$, $\mathfrak{I}_2$, and $\mathfrak{I}_3$, one can show
\[ \mathfrak{K}_1 \lesssim \omega_t \frac{\log \langle t+1 \rangle}{\langle t \rangle^{\frac{3}{2}}} \E_N^n[\mathbf{g}](t), \qquad \mathfrak{K}_2 \lesssim \langle \mathbf{\Lambda} \rangle\frac{\omega_t}{\langle t \rangle^2} \E_N^n[\mathbf{g}](t)  , \qquad \mathfrak{K}_3 \lesssim \frac{\omega_t}{\langle t \rangle^{\frac{9}{5}}} \E_N^n[\mathbf{g}](t) . \] 
In particular, to estimate $\phi [f-\mathbf{f}]$ and $\partial_t \Phi_{f-\mathbf{f}}$ through \eqref{EQUA:Hfbff} for $\mathfrak{K}_1$ and $\mathfrak{K}_3$, we note that either $|\gamma| \leq N-1-p_\infty$, or $|\gamma| \geq N-p_\infty$ and $|\beta| \leq p_\infty$. Then, we conclude by using $\mathbf{E}_N^n[\mathbf{g}](t) \leq \mathbf{\Lambda}$. 

For the analysis of $\mathfrak{K}_4$ and $\mathfrak{K}_5$, we will use the following ingredients.
\begin{enumerate}
\item  If $|\xi|, \, |\kappa|, \, |\gamma| \leq N-1-p_\infty$, then
\begin{equation*}
\big| \langle v \rangle \nabla_v \partial_v^\kappa \Phi_{f-\mathbf{f}} \big| \lesssim \omega_t, \qquad \big| \langle v \rangle \nabla_v \partial_v^\xi \Phi_{f} \big|+\big| \langle v \rangle \nabla_v \partial_v^\xi \Phi_{\mathbf{f}} \big| \lesssim \mathbf{\Lambda}, \qquad  \langle t+|z_f| \rangle^{2+|\gamma|} \big| \nabla_x \phi \big[ \partial_x^\gamma f \big](t,z_f) \big| \leq 2 \mathbf{\Lambda} . \color{white} \quad \square \color{black}
\end{equation*}
Recall the assumptions \eqref{EQU:assumpPhibff}--\eqref{EQUA:Hfbff}. The derivatives of $\Phi_{\mathbf{f}}$, $\Phi_f$ and $\Phi_{f-\mathbf{f}}$ are controlled pointwise by combining \eqref{EQU:HPhi} with the Sobolev embedding $W^{p_\infty,p}(\R^3_v) \hookrightarrow L^\infty (\R^3_v)$. For $\nabla_x \phi \big[ \partial_x^\gamma f \big]$, we use \eqref{eq:phiassump1}.
\item If $|\beta| \leq p_\infty$, we have $\int_x |\mathbf{g}_\beta|^p \dr x \lesssim |\E_N^n[\mathbf{g}]|^p \leq \mathbf{\Lambda}^p$ according to the Sobolev embedding for Banach-valued functions $W^{p_\infty , p}(\R^3_v,L^p(\R^3_x)) \hookrightarrow L^\infty (\R^3_v, L^p(\R^3_x))$ and $N \geq 2 p_\infty$.
\end{enumerate}
We start by studying $\mathfrak{K}_5$.
\begin{itemize}
\item Assume first that $|\xi|, \, |\kappa|, \, |\gamma| \leq N-1-p_\infty$. Then, we obtain $\mathfrak{K}_5 \lesssim \langle \mathbf{\Lambda} \rangle \omega_t \mathbf{\Lambda} \langle t\rangle^{-2} \E_N^n[\mathbf{g}](t) \leq \langle \mathbf{\Lambda} \rangle  \mathbf{\Lambda}^2 \omega_t \langle t\rangle^{-2}$.
\item Suppose now that $|\xi| \geq N-p_\infty$, so that $1 \leq |\beta| \leq p_\infty$ and $|\kappa|, \,  |\gamma| \leq p_\infty-1$. Then, $\Phi_{f-\mathbf{f}}$ and $\phi [f]$ can still be estimated pointwise. It yields
\[ \mathfrak{K}_5 \lesssim  \frac{  \mathbf{\Lambda}^2 \omega_t}{\langle t \rangle^2} + \frac{  \mathbf{\Lambda}  \omega_t}{\langle t \rangle^2}   \big\|\langle v \rangle \nabla_v \partial_v^\xi \Phi_{\mathfrak{f}} \big\|_{L^p(\R^3_v)} \bigg\| \int_{\R^3_x} \big| \mathbf{g}_\beta \big|^p \dr x \bigg\|_{L^\infty (\R^3_v)}^{\frac{1}{p}} \lesssim \langle \mathbf{\Lambda} \rangle \mathbf{\Lambda}^2 \frac{\omega_t }{ \langle t \rangle^2} . \]
\item We consider now the case $|\kappa| \geq N-p_\infty$, so that $1 \leq |\beta| \leq p_\infty$, $|\xi|, \,  |\gamma| \leq p_\infty-1$, and
\[ \mathfrak{K}_5 \lesssim \frac{\langle \mathbf{\Lambda} \rangle \mathbf{\Lambda}}{\langle t \rangle^2}   \big\|\langle v \rangle \nabla_v \partial_v^\kappa \Phi_{f-\mathbf{f}} \big\|_{L^p(\R^3_v)} \bigg\| \int_{\R^3_x} \big| \mathbf{g}_\beta \big|^p \dr x \bigg\|_{L^\infty (\R^3_v)}^{\frac{1}{p}} \lesssim \langle \mathbf{\Lambda} \rangle \mathbf{\Lambda}^2 \frac{\omega_t }{ \langle t \rangle^2}  . \]
\item Finally, if $|\gamma| \geq N-p_\infty$, we have $|\beta| \leq p_\infty$ and $|\xi|, \,  |\kappa| \leq p_\infty -1 $. Then, in that case, we have
\[ \mathfrak{K}_5 \lesssim \langle \mathbf{\Lambda} \rangle \omega_t  \Big\|  \langle t+|z_f| \rangle^{|\gamma|} \nabla_x \phi \big[ \partial_x^\gamma f \big](t,z_f) \big| \cdot \mathbf{g}_\beta  \Big\|_{L^p(\R^3_x \times \R^3_v)}  . \]
Recall from \eqref{EQUA:Hfbff} that we can estimate $\phi [f]$ in $L^p(\R^3_x)$ using \eqref{eq:phiassump2}. Then, proceeding exactly as in \ref{it3}, where we controlled $\mathfrak{I}_2$, we get
\[ \mathfrak{K}_5 \lesssim \langle \mathbf{\Lambda} \rangle \omega_t \frac{ \mathbf{\Lambda} }{\langle t \rangle^2} \E_N^n[\mathbf{g}](t) \leq \langle \mathbf{\Lambda} \rangle \mathbf{\Lambda}^2  \frac{\omega_t }{\langle t \rangle^2} . \]
\end{itemize}
Let us finally deal with $\mathfrak{K}_4$. Note that the fundamental theorem of calculus gives, since $z_f - z_\mathbf{f}=\lambda \Phi_{f-\mathbf{f}}(t,v)\log \langle t \rangle$,
\[ \nabla_x \phi \big[ \partial_x^\gamma \mathbf{f} \big] (t,z_f) - \nabla_x \phi \big[ \partial_x^\gamma \mathbf{f} \big] (t,z_{\mathbf{f}}) = \lambda \log \langle t \rangle \int_{\theta=0}^1 \Phi_{f-\mathbf{f}}(t,v) \cdot \nabla^2_x \phi \big[ \partial_x^\gamma \mathbf{f} \big] \Big(t,z_f-\theta \lambda  \Phi_{f-\mathbf{f}}(t,v)  \log \langle t \rangle \Big)  \dr \theta . \]
Let us use the notation $y_\theta (t,x,v) \coloneqq z_f-\theta \lambda  \Phi_{f-\mathbf{f}}(t,v) \log \langle t \rangle$. As $\| \Phi_{f-\mathbf{f}} (t,\cdot ) \|_{L^\infty_{v}} \leq \omega_t \leq 3 \mathbf{\Lambda} $, we have $t+|z_f| \leq 6\langle \mathbf{\Lambda} \rangle \langle t+|y_\theta| \rangle$ and
\begin{equation}\label{eq:meanval0}
\langle t+ |z_f| \rangle^{1+|\gamma|} \Big| \nabla_x \phi \big[ \partial_x^\gamma \mathbf{f} \big] (t,z_f) - \nabla_x \phi \big[ \partial_x^\gamma \mathbf{f} \big] (t,z_{\mathbf{f}}) \Big| \lesssim  \omega_t \log \langle t \rangle \int_{\theta=0}^1 \langle t+ |y_\theta| \rangle^{1+|\gamma|}\Big|  \nabla^2_x \phi \big[ \partial_x^\gamma \mathbf{f} \big] (t,y_\theta) \Big|   \dr \theta .
 \end{equation}
\begin{itemize}
\item Consider first the case $|\xi|, \, |\gamma| \leq N-1-p_\infty$, so that $|\langle v \rangle \nabla_v \partial_v^\xi \Phi_{\mathbf{f}}| \lesssim \mathbf{\Lambda}$. We estimate pointwise the force field using \eqref{eq:meanval0} and \eqref{eq:phiassump1}, since $(\phi [\mathbf{f}],\Phi_{\mathbf{f}}) \in \mathcal{W}^N (\R_+,\mathbf{\Lambda})$. It yields
\[ \mathfrak{K}_4 \lesssim \langle \mathbf{\Lambda} \rangle \frac{\omega_t\mathbf{\Lambda} \log \langle t \rangle}{\langle t \rangle^2} \big\|   \mathbf{g}_\beta  \big\|_{L^p(\R^3_x \times \R^3_v)} \lesssim \langle \mathbf{\Lambda}  \rangle \mathbf{\Lambda}^2 \frac{\omega_t \log \langle t \rangle}{\langle t \rangle^2}. \]
\item If $|\xi| \geq N-p_\infty$, we have $1 \leq |\beta| \leq p_\infty$ and $|\gamma| \leq p_\infty -1$. We can then use again \eqref{eq:meanval0}, together with \eqref{eq:phiassump1}, and $\int_x |\mathbf{g}_\beta|^p \dr x \lesssim \mathbf{\Lambda}^p$. We obtain
\[ \mathfrak{K}_4 \lesssim \mathbf{\Lambda}^2 \frac{\omega_t \log \langle t \rangle}{\langle t \rangle^2}+  \frac{\omega_t \mathbf{\Lambda} \log \langle t \rangle}{\langle t \rangle^2}\Big\|\langle v \rangle \nabla_v \partial_v^\xi \Phi_{\mathbf{f}}   \Big\|_{L^p(  \R^3_v)} \mathbf{\Lambda} \lesssim \langle \mathbf{\Lambda}  \rangle \mathbf{\Lambda}^2 \frac{\omega_t \log \langle t \rangle}{\langle t \rangle^2}. \]
\item We finally study the case $|\gamma| \geq N-p_\infty$, so that $|\beta| \leq p_\infty$ and $|\xi| \leq p_\infty -1$. Then, we use \eqref{eq:meanval0} and the Minkowski integral inequality to get
\[ \mathfrak{K}_4 \lesssim  \langle \mathbf{\Lambda} \rangle \omega_t \log \langle t \rangle \int_{\theta=0}^1 \Big\| \langle t+ |y_\theta| \rangle^{1+|\gamma|}  \nabla^2_x \big[ \partial_x^\gamma \mathbf{f} \big] (t,y_\theta)  \cdot \mathbf{g}_\beta \Big\|_{L^p(\R^3_x \times \R^3_v)}  \dr \theta , \]
where we recall $y_\theta (t,x,v) = x+tv+  \lambda \log \langle t \rangle \big[ (1-\theta)\Phi_f+\theta \Phi_{\mathbf{f}} \big](t,v)$. Next, we perform the affine change of variables $X(x)=y_\theta (t,x,v)$ to get
\[ \mathfrak{K}_4 \lesssim  \langle \mathbf{\Lambda} \rangle \omega_t \log \langle t \rangle \int_{\theta=0}^1\Big\| \langle t+ |X| \rangle^{1+|\gamma|}  \nabla^2_x \big[ \partial_x^\gamma \mathbf{f} \big] (t,X)  \Big\|_{L^p(\R^3_X)} \bigg\| \int_{\R^3_v} \big| \mathbf{g}_\beta \big|^p (t,Y_\theta (t,X,v) ,v) \dr v \bigg\|_{L^\infty (\R^3_X)}^{\frac{1}{p}}  \dr \theta  ,  \]
where $Y_\theta (t,x,v) = x-tv-\lambda \log \langle t \rangle \big[ (1-\theta)\Phi_f+\theta \Phi_{\mathbf{f}} \big](t,v)$. We may now proceed as in \ref{it3}, where $\mathfrak{I}_2$ was estimated, to obtain
\[ \mathfrak{K}_4 \lesssim \langle \mathbf{\Lambda}  \rangle \mathbf{\Lambda}^2 \omega_t  \log \langle t \rangle \langle t \rangle^{-2} . \]
\end{itemize}
\end{proof}

Next, we will require a result similar to Lemma \ref{Corderivgtof}.

\begin{Lem}\label{CorderivgtofDIFF}
We have, for all $t \in I$,
\begin{align*}
\E_{N-1} \big[ \mathbf{g}(t,x,v) -\mathbf{g}(t,x+\lambda \Phi_{f-\mathbf{f}} (t,v) \log \langle t \rangle ,v ) \big](t) \lesssim \log^{2N+6}  \langle t+1 \rangle \, \mathcal{E}_{N-1} [\Phi_{f-\mathbf{f}}](t) .
 \end{align*} 
\end{Lem}
\begin{proof}
Let $\widetilde{\mathbf{g}}(t,x,v) \coloneqq \mathbf{g}(t,x+\lambda \Phi_{f-\mathbf{f}} (t,v) \log \langle t \rangle ,v )$. We follow the proof of Lemma \ref{Corderivgtof} up to \eqref{eq:auxiliaire}, with $\Phi=-\Phi_{f-\mathbf{f}}$, $g=\mathbf{g}$, and $f_\circ = \widetilde{\mathbf{g}}$. It yields, for any $|\alpha_x|+|\alpha_v| \leq N-1$,
\[ \bigg\| \hspace{-0.04em} \frac{\langle x \rangle^{7+|\alpha_x|} \langle v \rangle^{7+N}}{\langle v \rangle^{|\alpha_x|}} \hspace{-0.04em} \Big( \hspace{-0.05em} \partial_x^{\alpha_x} \partial_v^{\alpha_v}\widetilde{\mathbf{g}} (t,x,v)- \big[\partial_x^{\alpha_x}\partial_v^{\alpha_v} \mathbf{g} \big] \hspace{-0.05em} \big( t,x+\lambda \Phi_{f-\mathbf{f}} (t,v) \log \langle t \rangle,v \big)  \hspace{-0.05em} \Big) \hspace{-0.08em} \bigg\|_{L^p_{x,v}} \! \lesssim \log^{b}  \langle t+1 \rangle \hspace{0.02em} \mathcal{E}_{N-1} [\Phi](t)  \E_{N-1}[\mathbf{g}](t) ,\]
where $b = 2N+5$. To conclude the proof, let us prove that
\[ \bigg\|  \frac{\langle x \rangle^{7+|\alpha_x|} \langle v \rangle^{7+N}}{\langle v \rangle^{|\alpha_x|}}  \Big(  \partial_x^{\alpha_x} \partial_v^{\alpha_v}\mathbf{g} (t,x,v)- \big[\partial_x^{\alpha_x}\partial_v^{\alpha_v} \mathbf{g} \big]  \big( t,x+\lambda \Phi_{f-\mathbf{f}} (t,v) \log \langle t \rangle,v \big)   \Big)  \bigg\|_{L^p_{x,v}} \lesssim \log^{b'} \! \langle t+1 \rangle \hspace{0.02em} \mathcal{E}_{N-1} [\Phi](t) \E_{N}[\mathbf{g}](t) ,\]
where $b'=N+8$. Indeed, we first apply the fundamental theorem of calculus. It remains to use the Minkowski integral inequality, $\langle v \rangle |\Phi_{f-\mathbf{f}}(t,v)| \leq \mathcal{E}_{N-1}[\Phi_{f-\mathbf{f}}](t)$, $\langle x \rangle \lesssim \big\langle x+\theta \lambda \Phi_{f-\mathbf{f}}(t,v) \log \langle t \rangle \big\rangle  \log \langle t+1 \rangle$ for all $0 \leq \theta \leq 1$ and to perform the change of variables $y(x)=x+\theta \lambda \Phi_{f-\mathbf{f}}(t,v) \log \langle t \rangle$.
\end{proof}

We conclude this subsection by proving the following result, which is a refinement of Proposition \ref{Prodecayintvbis}. 

\begin{Pro}\label{ProAppC}
Recall that the assumptions \eqref{EQU:assumpPhibff}--\eqref{EQUA:Hfbff} hold. Then, for any $|\alpha| \leq N-1$ and all $t \geq 0$,
\[ \Big\| \langle t+|x| \rangle^{3\frac{p-1}{p}+|\alpha|} \partial_x^\alpha \rho \big[ f-\mathbf{f} \big] (t,x) \Big\|_{L^p(\R^3_x)} \lesssim \E_{N-1} \big[ g- \mathbf{g} \big](t)+\frac{1}{\langle t \rangle^{\frac{1}{2}}} \mathcal{E}_{N-1} \big[ \Phi_{f-\mathbf{f}} \big](t)  . \]
\end{Pro}
\begin{proof}
Let $\mathcal{Z}_f(t,x,v) \coloneqq x-tv-\lambda \Phi_f (t,v) \log \langle t \rangle$, that we will often denote by $\mathcal{Z}_f$, and $\widetilde{\mathbf{g}}$ be defined as 
\[ \widetilde{\mathbf{g}} (t,x,v)= \mathbf{g} \big( t,x+\lambda \Phi_{f-\mathbf{f}}(t,v) \log \langle t \rangle , v \big)=\mathbf{f} \big(t,x+tv+\lambda \Phi_f (t,v) \log \langle t \rangle,v \big) . \] 
We start by applying the identities \eqref{eq:decayrhox0} and \eqref{eq:decayrhox}. Separating the cases $t+|x| \leq 1$ and $t+|x| \geq 1$, it yields, for $h=g-\widetilde{\mathbf{g}}$,
\begin{align*}
\langle t+|x| \rangle^{|\alpha|} \big| \partial_{x}^{\alpha}  \rho \big[f-\mathbf{f} \big] \big| (t,x)&  \lesssim \sum_{|\kappa_x|+ |\kappa_v|+|\kappa_\Omega|+m \leq |\alpha| }  \bigg| \int_{\R^3_v} \big[\partial_x^{\kappa_x} \partial_v^{\kappa_v} \Omega^{\kappa_\Omega} S^m h \big] \big(t, \mathcal{Z}_f,v \big) \dr v \bigg| \\ 
& \quad \, +  \sum_{|\xi|+|\beta| \leq |\alpha|}  \int_{\R^3_v} \frac{\langle v \rangle^{1+|\xi|}| \nabla_v \partial_v^{\xi} \Phi_f (t,v) |}{\langle t \rangle^{\frac{3}{4}} }  \langle \mathcal{Z}_f \rangle^{|\beta_x|} \langle v \rangle^{|\beta_v|} \big| \partial_{x,v}^\beta h \big(t, \mathcal{Z}_f,v \big)  \big|  \dr v .
 \end{align*}
 Then, as $N-1 \geq 2p_\infty -1$, we can apply Lemma \ref{Lem236458} to get, for any $|\xi|+|\beta| \leq N-1$,
\begin{align*} 
& \bigg\| \langle t+|x| \rangle^{3\frac{p-1}{p}} \int_{\R^3_v} \frac{\langle v \rangle^{1+|\xi|}| \nabla_v \partial_v^{\xi} \Phi_f (t,v) |}{\langle t \rangle^{\frac{3}{4}} } \langle \mathcal{Z}_f \rangle^{|\beta_x|} \langle v \rangle^{|\beta_v|} \big| \partial_{x,v}^\beta h \big(t, \mathcal{Z}_f,v \big)  \big|  \dr v \bigg\|_{L^p(\R^3_x)} \lesssim \frac{\mathcal{E}_{N-1}[\Phi_f ](t)}{\langle t \rangle^{\frac{3}{4}}}  \E_{N-1}[h](t).
\end{align*}
Note now that
\begin{align}
 \E_{N-1}[h](t) & \leq \E_{N-1}[g-\mathbf{g}](t)+\E_{N-1} \big[ \mathbf{g}(t,x,v)-\mathbf{g}\big( t,x+\lambda \Phi_{f-\mathbf{f}}(t,v) \log \langle t \rangle , v \big) \big] \nonumber \\
& \lesssim \E_{N-1}[g-\mathbf{g}](t)+\log^{2N+6} (t) \mathcal{E}_{N-1} \big[ \Phi_{f-\mathbf{f}}\big](t) ,  \label{14589}
\end{align}
where, in the last step, we used Lemma \ref{CorderivgtofDIFF}. Next, we fix multi-indices $|\kappa_x|+|\kappa_v|+|\kappa_\Omega|+m \leq |\alpha|$ and we use the shorthand $\mathcal{D} \coloneqq \partial_x^{\kappa_x} \partial_v^{\kappa_v} \Omega^{\kappa_\Omega} S^m$. Then, we remark
\[ \big[ \mathcal{D} h \big] \big( t, \mathcal{Z}_f , v \big) = \big[ \mathcal{D} g- \mathcal{D} \mathbf{g} \big] \big(t, \mathcal{Z}_f , v \big) +\big[ \mathcal{D} \mathbf{g} \big] \big( t,  \mathcal{Z}_f , v \big)- \big[ \mathcal{D} \mathbf{g} \big] \big( t, \mathcal{Z}_{\mathbf{f}},v \big)+ \big[ \mathcal{D} \mathbf{g} \big] \big( t, \mathcal{Z}_{\mathbf{f}}, v \big)-\big[ \mathcal{D} \widetilde{\mathbf{g}} \big] \big( t, \mathcal{Z}_f, v \big). \]
 Applying Lemma \ref{Lem236458}, together with \eqref{eq:23456789}, one gets
\begin{align*} 
 \bigg\| \langle t+|x| \rangle^{3\frac{p-1}{p}}\int_{\R^3_v} \big[ \mathcal{D} g-\mathcal{D}\mathbf{g} \big] \big( t,\mathcal{Z}_f,v \big)\dr v \bigg\|_{L^p(\R^3_x)} & \lesssim  \E_{N-1}[g-\mathbf{g}](t).
\end{align*}
For the second quantity to estimate, we introduce $\mathbf{h} \coloneqq \mathcal{D} \mathbf{g}$ to lighten the notations. Note first that
 \begin{align*}
 \int_{\R^3_v} \mathbf{h} \big(t , \mathcal{Z}_f ,v \big)- \mathbf{h} \big( t , \mathcal{Z}_{\mathbf{f}} , v\big) \dr v  = \log \langle t \rangle  \int_{\theta =0}^1 \int_{\R^3_v} \Phi_{f-\mathbf{f}} (t,v) \cdot \nabla_x \mathbf{h} \big( t,x-tv-\lambda \big[ \Phi_{f}-\theta \Phi_{f-\mathbf{f}} \big](t,v) \log \langle t \rangle , v \big) \dr v \dr \theta .
\end{align*}
If $t \geq 1$, we recall, for $\mathcal{Z}(t,x,v)=x-tv-\lambda \Phi (t,v) \log \langle t \rangle$, the identity
\begin{align}
 \partial_{x^i}\mathbf{h} (t,x,v) &  = - \frac{1}{t} \Big( \partial_{v^i} \big[ \mathbf{h} \big( t,\mathcal{Z},v \big) \big]+\lambda \log \langle t \rangle \partial_{v^i} \Phi \cdot \nabla_x \mathbf{h} \big(t, \mathcal{Z},v \big)-\partial_{v^i} \mathbf{h} \big( t, \mathcal{Z} , v \big) \Big). \label{foripp}
 \end{align}
 Applying it for $\Phi = \Phi_{f}-\theta \Phi_{f-\mathbf{f}}$, which satisfies $|\nabla_v \Phi (t,v) | \lesssim \mathbf{\Lambda}$, and performing integration by parts, yields
 \begin{align*}
\bigg| \int_{\R^3_v} \mathbf{h} \big(t , \mathcal{Z}_f ,v \big)- \mathbf{h} \big( t , \mathcal{Z}_{\mathbf{f}} , v\big) \dr v \bigg| &  \lesssim \frac{\mathcal{E}_{p_\infty} \! \big[ \Phi_{f-\mathbf{f}} \big](t)}{\langle t \rangle^{\frac{3}{4}}} \sum_{|\kappa| \leq 1}\int_{0}^1 \int_{\R^3_v} \! \big| \partial_{x,v}^{\kappa} \mathbf{h} \big( t,x-tv-\lambda \big[ \Phi_{f}-\theta \Phi_{f-\mathbf{f}} \big](t,v) \log \langle t \rangle , v \big) \big| \dr v \dr \theta ,
 \end{align*}
 since $|\Phi_{f-\mathbf{f}}|(t,v) +|\nabla_v \Phi_{f-\mathbf{f}}|(t,v) \lesssim \mathcal{E}_{p_\infty} [\Phi_{f-\mathbf{f}}](t)$ by the Sobolev embedding $W^{p_\infty,p}(\R^3_v) \hookrightarrow L^\infty (\R^3_v)$. Then, we bound $\partial_{x,v}^\kappa \mathcal{D} \mathbf{g}$ using \eqref{eq:23456789}. We finally obtain from Lemma \ref{Lem236458}, together with the Minkowski inequality for integrals, that
 \begin{align*}
  \bigg\| \langle t+|x| \rangle^{3\frac{p-1}{p}} \int_{\R^3_v} \big[\mathcal{D} \mathbf{g} \big] \big(t , \mathcal{Z}_f ,v \big)- \big[\mathcal{D} \mathbf{g} \big] \big( t , \mathcal{Z}_{\mathbf{f}} , v\big) \dr v \bigg\|_{L^p(\R^3_x)} \lesssim \frac{\mathcal{E}_{N-1} \big[ \Phi_{f-\mathbf{f}} \big](t)}{\langle t \rangle^{\frac{3}{4}}}  \E_N[\mathbf{g} ](t),
 \end{align*}
where we also used $p_\infty \leq N-1$. We conclude by using $\E_N[\mathbf{g} ](t) \leq \mathbf{\Lambda}$. We finally give the main steps to estimate the last quantity.
\begin{itemize}
\item Estimate $ \big[ \mathcal{D} \mathbf{g} \big] \big( t, \mathcal{Z}_{\mathbf{f}}, v \big)-\big[ \mathcal{D} \widetilde{\mathbf{g}} \big] \big( t, \mathcal{Z}_f, v \big)$ by carrying out computations similar to those used in the proof of Lemma \ref{Corderivgtof}.
\item Control the resulting terms using \eqref{foripp}, integration by parts, and \eqref{eq:23456789}. Finally, apply Lemma \ref{Lem236458}.
\end{itemize}

\end{proof}

\subsection{The backward problem}\label{Subsecbackward}

We finish here the construction of the wave operator $\mathscr{W}_+$, introduced in Definition \ref{DefWplus}. For this, to complement Proposition \ref{Prolocalwp}, we need to show uniqueness for the asymptotic Cauchy problem \eqref{eq:VPasymp} associated to the Vlasov-Poisson system. Then, we study regularity properties of $\mathscr{W}_+$. 

We fix, for all this section, $(f_\infty,\mathbf{f}_\infty) \in \mathcal{B}^2$. We assume that $g$ and $\mathbf{g}$ are two solutions to the asymptotic Cauchy problem \eqref{eq:VPasymp}, with respective scattering data $f_\infty$ and $\mathbf{f}_\infty$, respectively defined on the time intervals $[T_g,+\infty[$ and $[T_{\mathbf{g}},+\infty[$, where $T_g, \, T_{\mathbf{g}} \geq 0$, and such that
\[ \sup_{t \geq T_{\mathbf{g}}} \E_N^n[ \mathbf{g}](t) = \mathbf{\Lambda}, \qquad \qquad \qquad \sup_{t \geq T_g} \E_N^n[g](t) =  \mathbf{\Lambda}_0 ,  \qquad \qquad \qquad \mathbf{\Lambda}, \, \mathbf{\Lambda}_0 \geq 0 . \]
In this section, we then have
\begin{itemize}
\item $\Phi_f(t,v)=\nabla_v \phi_\infty[f_\infty](v)$ and $\Phi_{\mathbf{f}} (t,v) = \nabla_v \phi_\infty[\mathbf{f}_\infty](v)$.
\item $z_f(t,x,v)=x+tv+\lambda \Phi_f(v) \log \langle t \rangle$ and $z_{\mathbf{f}}(t,x,v)=x+tv+\lambda \Phi_{\mathbf{f}}(v) \log \langle t \rangle$.
\item $f$ and $\mathbf{f}$, satisfying $g(t,x,v)=f(t,z_f(t,x,v),v)$ and $\mathbf{g}(t,x,v) = \mathbf{f}(t,z_{\mathbf{f}}(t,x,v),v)$, are two solutions to \eqref{VP}. 
\end{itemize}
Note that in view of Proposition \ref{Prophiinfty},
\begin{equation}\label{eq:Phibackward}
 \mathcal{E}_N[\Phi_f]\lesssim \mathbf{\Lambda}, \qquad \qquad \mathcal{E}_N[\Phi_{\mathbf{f}}] \lesssim \mathbf{\Lambda}_0, \qquad \qquad  \mathcal{E}_{N-1} [\Phi_{f-\mathbf{f}}] \lesssim \E_{N-1} [f_\infty - \mathbf{f}_\infty]. 
 \end{equation}
We can then apply the results of Section \ref{Secprep} for $\Phi=\Phi_f$ and $\Phi=\Phi_{\mathbf{f}}$. We start by controlling $g-\mathbf{g}$.

\begin{Cor}\label{Corunique}
We have, for all $t \in [\max (T_g,T_{\mathbf{g}}),+\infty[$,
\[ \E_{N-1}^n [g-\mathbf{g}] (t)\leq \mathbf{C}_1   \E_{N-1}^n [f_\infty - \mathbf{f}_\infty],  \qquad \qquad \mathbf{C}_1 \big[N,n,p, \max (\mathbf{\Lambda}, \mathbf{\Lambda}_0) \big] >0 . \]
\end{Cor}
\begin{proof}
Assume, without loss of generality, that $\mathbf{\Lambda}_0 \leq \mathbf{\Lambda}$, and define $I \coloneqq [\min (T_g, T_{\mathbf{g}}),+\infty[$. In view of applying Proposition \ref{ProboundednessGeneral} and Corollary \ref{CorforenergypertGeneral}, let us prove that there exists a constant $C_2[N,p, \mathbf{\Lambda} ]$ such that
\[ \big(\phi[\mathbf{f}], \Phi_{\mathbf{f}} \big)  \in \mathcal{W}^N \big( I, C_2\mathbf{\Lambda} \big), \qquad \quad \big(\phi[f], \Phi_{f} \big)  \in \mathcal{W}^N \big( I, C_2\mathbf{\Lambda} \big), \quad \qquad \big(\phi[f-\mathbf{f}], \Phi_{f-\mathbf{f}} \big)  \in \mathcal{W}^{N-1} \big( I,\omega \big), \]
where $\omega_t = C_2 E_{N-1}[g-\mathbf{g}](t)+C_2\E_{N-1}[f_\infty - \mathbf{f}_\infty]$. For $\big(\phi[\mathbf{f}], \Phi_{\mathbf{f}} \big)$ and $\big(\phi[f], \Phi_{f} \big)$, this follows from Corollary \ref{Rqclassteleo}. For $(\phi[f-\mathbf{f}], \Phi_{f-\mathbf{f}})$, we proceed as follows. 
\begin{itemize}
\item We note that $\partial_t \Phi_{f-\mathbf{f}} =0$, and we obtain \eqref{EQU:HPhi} by \eqref{eq:Phibackward}. 
\item We obtain \eqref{eq:phiassump1}--\eqref{eq:phiassump2} from Proposition \ref{ProAppC} and the elliptic estimates of Corollary \ref{Corellip}.
\item We derive \eqref{eq:phiassump3}--\eqref{eq:phiassump4} by combining Corollary \ref{Corphiphiinfty} with \eqref{14589}--\eqref{eq:Phibackward}.
\end{itemize} 
 
  Recall from Definition \ref{Defgalpha} the notation $(g-\mathbf{g})_\alpha$, and apply the energy inequality of Proposition \ref{Proenergy} to $(g-\mathbf{g})_\alpha$, for any $|\alpha| \leq N-1$. Using Proposition \ref{ProboundednessGeneral} and Corollary \ref{CorforenergypertGeneral}, we obtain
\[ \E_{N-1}^n [g-\mathbf{g}](\tau)-\E_{N-1}^n \big[f_\infty - \mathbf{f}_\infty \big] \lesssim \int_{t=\tau}^{+\infty} \frac{1}{\langle t \rangle^{\frac{4}{3}}} \E_{N-1}^n [g-\mathbf{g}](t)+\frac{1}{\langle t \rangle^{\frac{4}{3}}}\E_{N-1} \big[f_\infty - \mathbf{f}_\infty \big] \dr t. \]
An application of the Grönwall inequality allows to conclude.
\end{proof}

By applying the previous Corollary \ref{Corunique} with $f_\infty=\mathbf{f}_\infty$, we obtain uniqueness for the asymptotic Cauchy problem associated with the Vlasov-Poisson system in $L^\infty ([T,+\infty[ , \mathcal{B})$, for any sufficiently large $T \geq 0$. This concludes the construction of the modified wave operator. Next, we focus on the regularity properties of $\mathscr{W}_+$.

\begin{Pro}\label{ProglobalLarge}
 Assume that $\mathbf{f}$ is a global solution to \eqref{VP} satisfying the hypotheses \eqref{EQU:assumpPhibff}. Let $\mathbf{\Lambda}_0 , \, \varepsilon \geq 0$, and let $f_\infty \in \mathcal{B}$ satisfy $\E_N[f_\infty] = \mathbf{\Lambda}_0$ and $\E_{N-1}[f_\infty - \mathbf{f}_\infty]= \varepsilon$. There exists $\varepsilon_0[N,p,\mathbf{\Lambda},\mathbf{\Lambda}_0]>0$ such that, if $\varepsilon \leq \varepsilon_0$, then the following hold:
 \begin{itemize}
 \item The unique solution $f$ to \eqref{VP} with scattering data $f_\infty$ is global in time,
 \item If, in addition, $\E_N^n[f_\infty]<+\infty$, then there exists $C[N,n,p,\mathbf{\Lambda},\mathbf{\Lambda}_0]>0$ such that $\sup_{t \geq 0}\E_N^n[g](t) \leq C \E_N^n[f_\infty]$.
 \end{itemize}
 \end{Pro}
 \begin{proof}
 Let $T \in \R_+ \cup \{+\infty \}$ be the minimal time such that, for all $t \geq T$,
 \begin{align}\label{BA:inf}
   \E_N[g](t) & \leq  C_{\mathrm{boot}}, \tag{BA2} 
 \end{align}
 for some constant $C_{\mathrm{boot}}[N,p,\mathbf{\Lambda},\mathbf{\Lambda}_0] \geq \mathbf{\Lambda}$ which will be fixed sufficiently large later. According to Proposition \ref{Prolocalwp}, and since uniqueness holds, we have $T \in \R_+$. More precisely, there exists $T_0[N,p,\mathbf{\Lambda}_0] \geq 0$ such that
 \begin{equation}\label{eq:fotBA2}
 \forall \, t \geq T_0, \qquad \qquad \E_N[g](t) \leq 2 \mathbf{\Lambda}_0.
 \end{equation}
  Let us show that, if $\varepsilon$ is small enough, then we can improve the bootstrap assumption \eqref{BA:inf}. This will imply that $T=0$. Note first that Corollary \ref{Corunique} and \eqref{BA:inf} provide 
 \begin{equation*}
\forall \, t \geq T, \qquad \qquad \E_{N-1}[g-\mathbf{g}](t) \leq \mathbf{C}_1 [N,p,C_{\mathrm{boot}}] \varepsilon . 
 \end{equation*}
Moreover, the proof of Corollary \ref{Corunique} also provides $\big(\phi[f-\mathbf{f}], \Phi_{f-\mathbf{f}} \big)  \in \mathcal{W}^{N-1} \big( [T,+\infty[, \mathbf{C}_1 \varepsilon \big)$, possibly after redefining $\mathbf{C}_1[N,p,C_{\mathrm{boot}}]$ as a larger constant. As a consequence, and since \eqref{EQU:assumpPhibff} holds, if $\varepsilon$ is small enough, we have
 \begin{equation*}
 \big(\phi[f], \Phi_{f} \big)  \in \mathcal{W}^{N-1} \big( [T,+\infty[, 2 \mathbf{\Lambda} \big) .
 \end{equation*}
In particular, using \eqref{eq:phiassump1}--\eqref{eq:phiassump2}, we observe that the assumptions of Proposition \ref{Proappendix1} are satisfied. We then obtain
\[ \forall \, T \leq t \leq T_0 , \qquad \qquad \E_N[f_\circ](t) \leq \mathbf{C}[N,p,\mathbf{\Lambda}] \cdot  \E_N[f_\circ](T_0) \log^{7+(15+2N)N} \langle T_0+1 \rangle, \]
where $f_\circ (t,x,v)= f(t,x+tv,v)$. Using both Lemma \ref{Corderivgtof} and Remark \ref{RqCorderivgtof}, together with $\mathcal{E}[\Phi_f] \lesssim \E_N[f_\infty]$, we get
\[ \forall \, T \leq t \leq T_0 , \qquad \qquad \E_N[g](t) \leq \widetilde{C}[N,p,\mathbf{\Lambda}, \mathbf{\Lambda}_0]\cdot \E_N[g](T_0) \log^{21+(19+2N)N} \langle T_0+1 \rangle. \]
In view of \eqref{eq:fotBA2}, it only remains to choose $C_{\mathrm{boot}}$ larger than both $3 \mathbf{\Lambda}_0$ and twice the right-hand side of the previous inequality. For the second part, the proof of Corollary \ref{Corunique} gives $(\phi[f], \Phi_{f} )  \in \mathcal{W}^{N} \big( R_+, C_2[N,p,C_{\mathrm{boot}}]C_{\mathrm{boot}} \big)$. Proposition \ref{ProboundednessGeneral}, together with the energy estimate of Proposition \ref{Proenergy}, yields boundedness of $\E_N^n[g]$.
\end{proof}
 
As a consequence, $\mathscr{W}_+$ is defined on an open subset $\mathcal{O}_+^{N,n} \subset \mathcal{B}_N^n$. Furthermore, since \eqref{BA:inf} holds globally in time, it follows from Corollary \ref{Corunique} that $\mathscr{W}_+$ is locally Lipschitz with respect to $ \E_{N-1}^n [\cdot]$.

\subsection{The forward problem}\label{SubsecForward}

In this section, we associate to a function $\mathfrak{f} \colon [0,T[ \times \R^3_x \times \R^3_v \to \R$ the coefficient $\Phi_\mathfrak{f}$, the coordinate $z_\mathfrak{f}$, and the function $\mathfrak{g} \colon [0,T[ \times \R^3_x \times \R^3_v \to \R$ given by
\[ \Phi_\mathfrak{f}(t,v) \coloneqq t^2 \nabla_x \phi[ \mathfrak{f}](t,tv), \qquad \quad z_\mathfrak{f} (t,x,v) \coloneqq  x+tv+\lambda \Phi_{\mathfrak{f}}(t,v) \log \langle t \rangle, \qquad \quad \mathfrak{g}(t,x,v)=\mathfrak{f} \big( t,z_{\mathfrak{f}}(t,x,v),v \big) . \]

\begin{Pro}\label{ProforwardLarge}
Let $\mathbf{f}$ be a global solution to \eqref{VP} satisfying the assumptions \eqref{EQU:assumpPhibff}, namely $\sup_{t \geq 0} \E_N^n [\mathbf{g}](t) \leq \mathbf{\Lambda}$ and $\big( \phi [\mathbf{f} ], \Phi_{\mathbf{f}} \big) \in \mathcal{W}^N \big( \R_+, \mathbf{\Lambda} \big)$, for some $\mathbf{\Lambda} \geq 0$. Let $f$ be another solution to \eqref{VP} such that
\[ \E_N [f](0) \leq 5  \mathbf{\Lambda}, \qquad \qquad \E_{N-1} [f-\mathbf{f} ](0)= \varepsilon, \qquad \varepsilon \geq 0  . \]
There exists $\varepsilon_0 [N,p, \mathbf{\Lambda} ]>0$ such that, if $\varepsilon \leq \varepsilon_0$, then $f$ is global in time. Moreover, if $\E_N^n[f](0) <+\infty$, then there exist constants $\mathbf{C}_{\mathrm{top}}^n [N,n,p,\mathbf{\Lambda}]>0$ and $\mathbf{C}_{\mathrm{Lip}}^n[N,n,p,\mathbf{\Lambda}] >0$, such that
\begin{equation*} \forall \, t \geq 0, \qquad \qquad \E_N^n[g](t) \leq \mathbf{C}^n_{\mathrm{top}}\E_N^n [f](0), \qquad \E_{N-1}^n[g-\mathbf{g}](t) \leq \mathbf{C}_{\mathrm{Lip}}^n\E_{N-1}^n [f - \mathbf{f}](0) . 
\end{equation*}
\end{Pro}
\begin{Rq}
The analysis carried out in Section \ref{Subsecmodscat} below will show that $\mathbf{f}(0,\cdot , \cdot) \in \mathscr{W}_+ \big( \mathcal{O}_+^{N,n} \big)$. Conversely, if there exists $\mathbf{f}_\infty \in \mathcal{B}_N^n$ such that $\mathbf{f} (0,\cdot , \cdot) = \mathscr{W}_+ ( \mathbf{f}_\infty )$, then $\mathbf{f}$ satisfies \eqref{EQU:assumpPhibff} for some $\mathbf{\Lambda} \geq 0$. Indeed, we have
\[\mathcal{E}_N \big[\nabla_v \phi_\infty [\mathbf{f}_\infty] \big] <+\infty , \qquad \qquad  \sup_{t \geq 0} \E_N \big[ \mathbf{G} \big](t) <+ \infty , \quad  \mathbf{G} (t,x,v) \coloneqq \mathbf{f} \big( t,x+tv+\lambda \nabla_v \phi_\infty [ \mathbf{f}_\infty ] \log \langle t \rangle , v \big). \]
Corollary \ref{Rqclass} then implies that $\big( \phi [\mathbf{f} ], \Phi_{\mathbf{f}} \big) \in \mathcal{W}^N \big( \R_+, \mathbf{\Lambda} \big)$, for some $\mathbf{\Lambda} \geq 0$. The uniform bound on $\E_N[\mathbf{f}]$ then follows from Proposition \ref{ProboundednessGeneral}, together with Proposition \ref{Proenergy} and Grönwall's inequality.
\end{Rq}
In addition, several important estimates for $f$ will be derived throughout the proof of Proposition \ref{ProforwardLarge}. To this end, we introduce a time $T_0[N,p,\mathbf{\Lambda}]\geq 0$, which will be chosen sufficiently large below.

\subsubsection{Estimates for a compact domain of times} We start by studying the solution on the time interval $[0,T_0]$. Recall the notation $\mathfrak{f}_\circ (t,x,v)=\mathfrak{f}(t,x+tv,v)$. By Cauchy stability, there exist $\varepsilon_1 [N,p,\mathbf{\Lambda},T_0]>0$ and $C_{T_0}[N,p,\mathbf{\Lambda},T_0]\geq 1$ such that, if $\varepsilon \leq \varepsilon_1$, then\footnote{One may show that $C_{T_0}$ grows only polynomially in $T_0$, rather than exponentially as predicted by a standard Grönwall argument.}
\begin{equation}\label{eq:trucmachin}
 \sup_{0 \leq t \leq T_0} \E_{N-1}[f_\circ - \mathbf{f}_{\circ}] (t) \leq C_{T_0} \varepsilon . 
 \end{equation}
We now need to estimate $g$ in order to prepare the analysis for $t \geq T_0$. Recall that $N-1 \geq 2p_\infty -1$, and apply Corollary \ref{Rqclass} with $\Phi=0$. It yields
\[ \big( \phi[f-\mathbf{f}], \Phi_{f-\mathbf{f}} \big) \in \mathcal{W}^{N-1} \big( [0,T_0], \mathbf{C}_{T_0} \varepsilon \big), \]
for some $\mathbf{C}_{T_0}[N,p,\mathbf{\Lambda},T_0]>0$. Let us finally show that, up to redefining $\mathbf{C}_{T_0}$ to be a larger constant,
\[ \forall \, t \in [0,T_0], \qquad \qquad  \E_{N-1}[g - \mathbf{g}] (t) \leq \mathbf{C}_{T_0} \varepsilon. \]
For this, we write
\[ \E_{N-1}[g-\mathbf{g}](t) \leq \E_{N-1}\big[(f_\circ-\mathbf{f}_\circ) (t,x+\lambda \Phi_{f}(t,v) \log \langle t \rangle , v) \big]+\E_{N-1} \big[\mathbf{g}(t,x+\lambda \Phi_{f-\mathbf{f}} (t,v)\log \langle t \rangle , v)-\mathbf{g}(t,x,v) \big]. \]
We handle the first term using Remark \ref{RqCorderivgtof}, applied with $\Phi=\Phi_f$, and $\mathcal{E}_{N-1}[\Phi_f](t) \leq \mathbf{C}_{T_0} \varepsilon + \mathbf{\Lambda}$. We deal with the second one by combining Lemma \ref{CorderivgtofDIFF} with the bound $\mathcal{E}_{N-1}[\Phi_{f-\mathbf{f}}](t) \leq \mathbf{C}_{T_0} \varepsilon$.

\subsubsection{Estimates for large times} Let $T \in [T_0,+\infty ]$ be the maximal time such that
\begin{align}
\big(\phi[f-\mathbf{f}], \Phi_{f-\mathbf{f}} \big) & \in \mathcal{W}^{N-1} \big( [0,T[,C_{\mathrm{boot}}\varepsilon \big) ,   \label{BA1} \tag{BA3}  
\end{align}
where $C_{\mathrm{boot}}[N,p,\mathbf{\Lambda},T_0] \geq \mathbf{C}_{T_0}$ is a constant which will be fixed below. 
\begin{Lem}\label{Lemtoporder}
If $\varepsilon$ is small enough, there exists $\mathbf{C}_{\mathrm{top}}[N,p,\mathbf{\Lambda}]>0$ such that 
\[ \sup_{0 \leq t <T} \E_N[g](t) \leq \mathbf{C}_{\mathrm{top}}, \qquad \qquad \big(\phi[f],\Phi_f \big) \in \mathcal{W}^N \big([0,T[,\mathbf{C}_{\mathrm{top}} \big) . \]
\end{Lem}
\begin{proof}
In view of \eqref{EQU:assumpPhibff} and the bootstrap assumption \eqref{BA1}, one has, for $\varepsilon$ small enough,
\begin{equation}\label{eq:justrq}
 \big(\phi[f], \Phi_f \big) \in \mathcal{W}^{N-1} \big([0,T[,2\mathbf{\Lambda} \big). 
 \end{equation}
As a consequence, we control $f_\circ$ through Proposition \ref{Proappendix1} proved below. Applied with $N_0=N$ and $a=0$, it yields
\[\forall \, t \in [0,T[, \qquad  \E_N [f_\circ ](t) \lesssim  \log^{7+(15+2N) N} \langle t+1 \rangle . \]
Applying Corollary \ref{Rqclass} with $\Phi=0$ yields 
\[  \big(\phi[f], \Phi_f \big) \in \mathcal{W}^N \big([0,T[, C \E_N[f_\circ] \big) , \qquad C[N,p] >0 . \]
One then obtains boundedness for $\E_N[g]$ by combining the energy estimate of Proposition \ref{Proenergy}, applied to $g_\alpha$ for any $|\alpha| \leq N$, Proposition \ref{ProboundednessGeneral}, and the Grönwall inequality. The proof is completed by applying Corollary \ref{Rqclass} with
$\Phi=\Phi_f$ and $a=7+(15+2N)N$.
\end{proof}

We are now able to bound $\E_N[g-\mathbf{g}]$. Note that if $N-1 \geq 2p_\infty$, we can simply use \eqref{eq:justrq} instead of Lemma \ref{Lemtoporder}.

\begin{Lem}\label{LemENh}
There exists $T_1[N,p,\mathbf{\Lambda}] \geq 0$ such that, if $T_0 \geq T_1$, then
\[ \forall \, t \in [T_0,T[, \qquad \qquad \E_{N-1}[g-\mathbf{g}](t) \leq \big(2 \mathbf{C}_{T_0}+T_0^{-1/6} C_{\mathrm{boot}} \big) \varepsilon  . \]
\end{Lem}
\begin{proof}
We apply the energy estimate of Proposition \ref{Proenergy} to $(g-\mathbf{g})_\alpha$, for any $|\alpha| \leq N-1$. In view of the previous Lemma \ref{Lemtoporder}, \eqref{BA1}, and \eqref{EQU:assumpPhibff}, one can estimate $\overline{\T}_{\phi [f]}(g_\alpha-\mathbf{g}_\alpha)$ by applying Proposition \ref{ProboundednessGeneral} and Corollary \ref{CorforenergypertGeneral}. Assuming without loss of generality that $\mathbf{C}_{\mathrm{top}} \geq \mathbf{\Lambda}$, we obtain
\[ \E_{N-1} [g-\mathbf{g}](\tau) \leq \E_{N-1} [g-\mathbf{g}](T_0) + p \! \int_{T_0}^\tau  \frac{1}{\langle t \rangle^{\frac{4}{3}}}  \Big( C_0[N,p , \mathbf{C}_{\mathrm{top}} ] \E_{N-1}[g-\mathbf{g}](t)+ \mathfrak{C}[N,p, \mathbf{C}_{\mathrm{top}}]   C_{\mathrm{boot}}\varepsilon \Big) \dr t , \]
for all $T_0 \leq \tau <T$. As $\mathbf{C}_{\mathrm{top}}$ does not depend on $T_0$, the Grönwall inequality yields the result.
\end{proof}

Let us now show that, for all $t \geq 0$,
\begin{align}
  \mathcal{E}_{N-1} \big[ \Phi_{f-\mathbf{f}} \big] (t) & \lesssim  \big\| \langle t+|x| \rangle^{2}  \nabla_x  \phi [f-\mathbf{f}](t,x) \big\|_{L^\infty_x}+ \sum_{|\gamma| \leq N-1}  \langle t\rangle^{2- \frac{3}{p}}\big\| \langle t+|x| \rangle^{1+|\gamma|} \nabla_x^2 \partial_x^\gamma \phi [f-\mathbf{f}](t,x) \big\|_{L^p_x} \nonumber \\
  & \lesssim \E_{N-1} \big[ g- \mathbf{g} \big](t)+ \langle t \rangle^{-\frac{1}{2}} \cdot \mathcal{E}_{N-1} \big[ \Phi_{f-\mathbf{f}} \big](t). \label{pourjusteici}
\end{align}
For the first inequality, we use \eqref{toremember} and $\langle v \rangle^2 |\Phi_{f-\mathbf{f}} |(t,v) \leq (t+|tv|)^2 | \nabla_x \phi [f] |(t,tv)$. For the second one, we use the elliptic estimates of Corollary \ref{Corellip}, together with Proposition \ref{ProAppC}. We then deduce that there exists $T_2[N,p,\mathbf{\Lambda}] \geq 0$ such that, if $T_0 \geq T_2$, then
\begin{equation}\label{pourjusteici2}
 \forall \, T_0 \leq t < T, \qquad \qquad  \mathcal{E}_{N-1} \big[ \Phi_{f-\mathbf{f}} \big] (t) \lesssim \E_{N-1} \big[ g- \mathbf{g} \big](t). 
 \end{equation}
Now, we claim that there exists $C[N,p,\mathbf{\Lambda}]>0$, such that
\[ \big(\phi[f-\mathbf{f}], \Phi_{f-\mathbf{f}} \big)  \in \mathcal{W}^{N-1} \big( [T_0,T[,C_{\mathrm{new}}\varepsilon \big), \qquad \qquad C_{\mathrm{new}} \coloneqq C \big(2 \mathbf{C}_{T_0}+T_0^{-1/6} C_{\mathrm{boot}} \big)  . \]
For this, we follow the analysis carried out in the proof of Corollary \ref{Rqclass}. First, note that
\begin{align}
 \E_{N-1} \big[ g(t,x , v)-\mathbf{g} \big(t,x+\lambda \Phi_{f-\mathbf{f}}(t,v)\log \langle t \rangle , v \big) \big] & \lesssim \E_{N-1}[g-\mathbf{g}](t)+\log^{2N+6} \langle t+1 \rangle \mathcal{E}_{N-1} \big[ \phi_{f-\mathbf{f}}\big](t)  \\
 & \lesssim \big(2 \mathbf{C}_{T_0}+T_0^{-1/6} C_{\mathrm{boot}} \big) \varepsilon \log^{2N+6} \langle t+1 \rangle . \label{equa:last}
 \end{align}
The first inequality is exactly \eqref{14589}, while the second follows from Lemma \ref{LemENh} and \eqref{pourjusteici2}. To establish \eqref{eq:phiassump1}--\eqref{eq:phiassump2} and \eqref{EQU:HPhi}, we combine \eqref{pourjusteici}--\eqref{equa:last} with Lemma \ref{LemENh}. To control $\partial_t \Phi_{f-\mathbf{f}}$, we apply Lemma \ref{Lemdpphismall} together with \eqref{equa:last}. Finally, \eqref{eq:phiassump3}--\eqref{eq:phiassump4} follow from \eqref{eq:521479} and \eqref{eq:phiassump1}--\eqref{eq:phiassump2}.

To conclude the proof of global existence, it remains to improve \eqref{BA1}. We deal with the time interval $[0,T_0[$ by choosing $C_{\mathrm{boot}} \geq 2\mathbf{C}_{T_0}$. For $[T_0,T[$, we have $C_{\mathrm{new}} \leq 2C_{\mathrm{boot}}/3$, provided that $T_0$ and $C_{\mathrm{boot}}$ are chosen such that
\[T_0 \geq \max (T_1,T_2), \qquad \qquad CT_0^{-1/6}  \leq \frac{1}{3}, \qquad \qquad C_{\mathrm{boot}} \geq 6C \mathbf{C}_{T_0} .\]
Finally, let us note that if $\E_N^n [f](0) <+\infty$, higher moments can be propagated by combining the energy inequality of Proposition \ref{Proenergy} with Proposition \ref{ProboundednessGeneral} and Corollary \ref{CorforenergypertGeneral}.

\subsubsection{Modified scattering}\label{Subsecmodscat}

We first prove that $g$ converges.
\begin{Pro}\label{Proscatt0}
There exists two functions $f_\infty, \, \mathbf{f}_\infty : \R^3_x \times \R^3_v \to \R$ such that,
\[ \forall \, t \geq 0, \qquad \qquad \E_{N-1}^n \big[ g(t,\cdot,\cdot) - f_\infty \big]+\E_{N-1}^n \big[ \mathbf{g}(t,\cdot,\cdot) - \mathbf{f}_\infty \big] \lesssim \langle t \rangle^{-\frac{1}{3}} . \]
Moreover, $\E_N^n[\mathbf{f}_\infty] \leq \mathbf{\Lambda}$, $\E_N^n [f_\infty] \leq \mathbf{C}_{\mathrm{top}}^n\E_N^n [f](0)$, $\E_{N-1}^n[f_\infty - \mathbf{f}_\infty] \leq \mathbf{C}_{\mathrm{Lip}}^n\E_N^n [f - \mathbf{f}](0)$, and
\[ \lim_{t \to +\infty}\E_N^n \big[ g(t,\cdot,\cdot) - f_\infty \big]+\E_{N}^n \big[ \mathbf{g}(t,\cdot,\cdot) - \mathbf{f}_\infty \big] =0. \]
\end{Pro}
\begin{proof}
The existence of $f_\infty$ and $\mathbf{f}_\infty$, as well as the corresponding convergence results, follows from Corollary \ref{Corscatstrong}. The bounds for $f_\infty$ and $\mathbf{f}_\infty$ follow from $\E_N^n[\mathbf{g}](t) \leq \mathbf{\Lambda}$ and the energy bounds in Proposition \ref{ProforwardLarge}. 
\end{proof}

To conclude this section, let us show that $G(t, \cdot , \cdot ) \to  f_\infty$ in $L^p_{x,v}$ as $t \to +\infty$, where
\[ G(t,x,v) \coloneqq f \big( t,x+tv+\lambda \nabla_v \phi_\infty [f_\infty](v) \log \langle t \rangle , v \big)= g \big(t,x+\lambda \big(\nabla_v \phi_\infty [f_\infty](v) - \Phi_f (t,v) \big) \log \langle t \rangle , v \big). \]
For this, we use the mean value theorem in $x$, $\sup_{t \geq 0}\E_N[g](t)<+\infty$, and $\|\nabla_v \phi_\infty [f_\infty] - \Phi_f (t,\cdot)\|_{L^\infty_v} \lesssim \langle t \rangle^{-1/2}$, which follows from Corollary \ref{Corphiphiinfty}. Finally, arguing exactly as for the small data solutions in Section \ref{Subsecsmall}, one can prove that the convergences in Proposition \ref{Proscatt0} remain valid with $g$ replaced by $G$.

\section{Continuity of the wave operator}\label{Sec5}

The goal of this section is to complete the proof of Theorem \ref{Th2}. Note first that by the time-reversibility of the Vlasov-Poisson system\footnote{The function $f(t,x,v)$ is a solution to \eqref{VP} if and only if the same holds for $f(-t,x,-v)$.}, it suffices to study $\mathscr{W}_+$. In view of the results in Section \ref{Subsecbackward}, as well as Propositions \ref{ProforwardLarge}, \ref{Proscatt0}, and \ref{Prospher}, it only remains to show that $\mathscr{W}_+ \colon \mathcal{O}_+^{N,n} \to \mathscr{W}_+ \big( \mathcal{O}_+^{N,n} \big)$ is bicontinous for the norm $\E_N^n [\cdot ]$.

\begin{Pro}\label{Pro1}
Let $N \geq 2p_\infty$ and $n =(n_x,n_v) \in \mathbb{Z}^2$ be such that $n_x, \, n_v \geq 7$. The inverse modified scattering operator $\mathscr{W}_+^{-1} \colon \mathscr{W}_+ \big( \mathcal{O}_+^{N,n} \big) \to \mathcal{O}_+^{N,n} $ is continuous with respect to the norm $\E_N^n [\cdot ]$.
\end{Pro}
\begin{proof}
For $(f_0^1,f_0^2) \in \mathscr{W}_+\big(\mathcal{O}_+^{N,n} \big)^2$, we denote, for any $i \in \{ 1,2 \}$, $f_\infty^i \coloneqq \mathscr{W}_+^{-1}(f_0^i)$, and let $f^i$ be the global solution to \eqref{VP} arising from the initial data $f_0^i$. We also set
\[ g^i(t,x,v) \coloneqq f^i \big(t,x+tv+\lambda \nabla_v \phi_\infty [f_\infty^i] \log \langle t \rangle ,v \big). \]
We fix $f_0^1 \in \mathscr{W}_+\big(\mathcal{O}_+^{N,n} \big)$, and we consider $\mathbf{\Lambda} \geq 0$ such that 
$\sup_{t \geq 0} \E_N^n[g^1](t) = \mathbf{\Lambda}$. In what follows, the constant $C \geq \mathbf{\Lambda}$ may vary from line to line, but depends only on $(N,n,p,\mathbf{\Lambda})$. Let $\varepsilon >0$ and $f_0^2 \in \mathscr{W}_+\big(\mathcal{O}_+^{N,n} \big)$ be such that $\E_N^n[f_0^1-f_0^2] \leq \delta$, where $0 <\delta \leq \varepsilon$ is a constant to be chosen later. If $\varepsilon$ is small enough, then
\begin{itemize}
\item $\E_{N-1}^n[f_\infty^1-f^2_\infty] \leq C \varepsilon$ and $\E_N^n[f_\infty^2]\leq C$ by Proposition \ref{Proscatt0},
\item $\E_N^n[g^2](t) \leq C$ for all $t \geq 0$ according to Proposition \ref{ProglobalLarge},
\item $\big(\phi [f^i],\nabla_v \phi_\infty[f_\infty^i] \big) \in \mathcal{W}^N(\R_+,C)$, for any $i \in \{1,2\}$, by Corollary \ref{Rqclassteleo}.
\end{itemize}
Note now that we have, for all $t \geq 0$,
\[ \E_N^n \big[ f_\infty^1-f_\infty^2 \big] \leq \E_N^n \big[f^1_\infty - g^1 (t,\cdot , \cdot) \big]+\E_N^n \big[ g^1-g^2 \big](t) + \E_N^n \big[ g^2 (t, \cdot , \cdot) - f^2_\infty \big]  . \]
Let us now prove that there exists $T[N,p,n,\mathbf{\Lambda}, \varepsilon]>0$ such that, for all $t \geq T$,
\[ \E_N^n \big[f^1_\infty - g^1 (t,\cdot , \cdot) \big] \leq \varepsilon, \qquad \quad \E_N^n \big[g^2 (t,\cdot , \cdot) -f^2_\infty \big] \leq 2\varepsilon . \]
For this, we need to refine Corollary \ref{Corscatstrong}. In fact, the result follows from a slight adaptation of the last step of its proof. More precisely, we need to show that the constant $C$ in \eqref{eq:forconti} can be chosen independently of $g^2$ whenever
\[ \E_N^n[g^1-g^2](0) = \E_N^n[f^1_0-f^2_0]\leq \varepsilon. \]
This follows from the estimate
\[ \big\|g^2_\alpha(0,\cdot,\cdot)-h\big\|_{L^p_{x,v}} \leq  \big\|g^1_\alpha(0,\cdot,\cdot)-h\big\|_{L^p_{x,v}} + \E_N^n[g^1-g^2](0). \]
We conclude the proof by applying Cauchy stability. There exists $\delta[N,p,n,\mathbf{\Lambda},T,\varepsilon]>0$ such that, if $\E_{N}^n [f^1_0-f_0^2] \leq \delta$, then $ \E_N^n [ g^1-g^2 ](T) \leq \varepsilon$.
\end{proof}

We now show that $\mathscr{W}_+$ is continuous. In other words, we need to establish Cauchy stability for \eqref{VP} with asymptotic data. Fix $N \geq 2p_\infty$ and $n=(n_x,n_v) \in \mathbb{Z}^2$ with $n_x , \, n_v \geq 7$. Let $(f^k_\infty)_{k \geq 1}$ be a sequence in $\mathcal{O}_+^{N,n}$, and assume that there exists $f_\infty^\infty \in \mathcal{O}_+^{N,n}$ such that $\E_N^n [f^k_\infty - f_\infty^\infty] \to 0$ as $k \to +\infty$. Let, for $k \in \mathbb{N}^* \cup \{ \infty \}$, $f_0^k \coloneqq \mathscr{W}_+(f_\infty^k)$, $f^k$ be the global solution to \eqref{VP} arising from the initial data $f_0^k$, and
\[ g^k(t,x,v) \coloneqq f^k \big(t,x+tv+\lambda \nabla_v \phi_\infty \big[ f_\infty^k \big](v) \log \langle t \rangle ,v \big). \]
In view of the assumptions and Proposition \ref{Prolocalwp}, there exists $\mathbf{\Lambda} \geq 0$ such that, for all $k \in \mathbb{N}^* \cup \{ \infty \}$,
 \begin{equation}\label{eq:decayphik} \forall \, t \geq 0, \quad \E_N^n [g^k](t) \leq \mathbf{\Lambda}, \qquad \qquad \qquad \big( \phi [f^k], \nabla_v \phi_\infty \big[ f_\infty^k \big] \big) \in \mathcal{W}^N \big(  \R_+, \mathbf{\Lambda}  \big)   .  \tag{decay,  $\phi[f^k]$}
 \end{equation}
The strategy is the following. By Corollary \ref{Corunique}, we already have $\sup_{t \geq 0} \E_{N-1}^n[g^k-g^\infty](t) \lesssim \E_{N-1}^n[f^k_\infty -f^\infty_\infty]$, so we only need to show convergence at top order. For any $|\alpha_x|+|\alpha_v| =N-1$, we write 
\[g^k_\alpha=M^k_\alpha+R^k_\alpha, \qquad \qquad g_\alpha^k (t,x,v) \coloneqq \langle x \rangle^{n_x+|\alpha_x|} \langle v \rangle^{n_v+N-|\alpha_x| }\partial_x^{\alpha_x} \partial_v^{\alpha_v} g^k (t,x,v) , \] 
where
\begin{alignat*}{2}
 &\overline{\T}_{\phi [f^k]} \big( M^k_\alpha \big)= \overline{\T}_{\phi[f^\infty]} \big(  g^\infty_\alpha \big), \qquad \qquad \qquad \qquad && \overline{\T}_{\phi [f^k]} \big( R^k_\alpha \big)=\overline{\T}_{\phi[f^k]} \big( g^k_\alpha \big) - \overline{\T}_{\phi[f^\infty]} \big( g^\infty_\alpha \big)  \\
 & M_\alpha^k(t,\cdot , \cdot) \xrightarrow[t \to+\infty]{} \big[ f_\infty^\infty \big]_\alpha , \qquad \qquad && R^k_\alpha (t, \cdot , \cdot ) \xrightarrow[t \to+\infty]{} \big[ f^k_\infty- f_\infty^\infty \big]_\alpha .
\end{alignat*}

\begin{Rq} 
Note that $M^\infty_\alpha=g^\infty_\alpha$ and $R^\infty_\alpha=0$. Moreover, since the source terms decays sufficiently fast in time according to Proposition \ref{ProboundednessGeneral}, a slight adaptation of Proposition \ref{ProAsympCau} shows that $M^k_\alpha$ and $R^k_\alpha$ are well-defined. 
\end{Rq}
For convenience, we define the norms
\begin{align*}
 \mathbf{E}_1 [h](t) & \coloneqq \| h(t,x,v)\|_{L^p(\R^3_x \times \R^3_v)}+ \big\| \nabla_v h(t,x,v) \big\|_{L^p(\R^3_x \times \R^3_v)}+\big\| \langle x \rangle \langle v \rangle^{-1} \nabla_x h(t,x,v) \big\|_{L^p(\R^3_x \times \R^3_v)} , \\
  \mathbf{E}_2 [h](t) & \coloneqq \mathbf{E}_1[h](t)+  \big\| \nabla_v^2 h(t,x,v) \big\|_{L^p_{x,v}}+\big\| \langle x \rangle \langle v \rangle^{-1} \nabla_x \nabla_v h(t,x,v) \big\|_{L^p_{x,v}} + \big\| \langle x \rangle^2 \langle v \rangle^{-2} \nabla^2_x h(t,x,v) \big\|_{L^p_{x,v}}.
 \end{align*}
Then, $\E_N^n [f_0^k-f_0^\infty] \to 0$, as $k \to +\infty$, would follow if we could show, for any $|\alpha|=N-1$,
\begin{align*}
 \lim_{k \to \infty} \sup_{t \geq 0} \mathbf{E}_1 \big[ M^k_\alpha -g^\infty_\alpha \big](t) =0, \qquad \qquad \lim_{k \to \infty} \sup_{t \geq 0} \mathbf{E}_1 \big[ R^k_\alpha \big](t) =0   .
\end{align*}
This is the purpose of the next two results, which will conclude the proof of Theorem \ref{Th2}. 
\begin{Lem}\label{LemconvM}
For any $|\alpha|=N-1$, we have $\sup_{t \geq 0} \mathbf{E}_1 \big[ M^k_\alpha -g^\infty_\alpha \big](t) \to 0$ as $k \to +\infty$.
\end{Lem}
\begin{proof}
Fix $|\alpha|=N-1$. Let $(\rho_\epsilon)_{0 < \epsilon \leq 1}$ be a mollifier, and $\chi \in C^\infty(\R,[0,1])$ be a cutoff function such that $\chi (s)=1$ for $s \leq 1$, and $\chi (s) =0$ for $s \geq 2$. We define $F_\infty \coloneqq \overline{\T}_{\phi[f^\infty]} (  g^\infty_\alpha )$, and
\[ F_j \coloneqq \chi \big(  |x,v|/j\big) \cdot F_\infty \ast_{x,v} \rho_{1/j}, \qquad \qquad m_j \coloneqq \chi \big( |x,v|/j \big) \cdot \big[ f_\infty^\infty \big]_\alpha   \ast_{x,v} \rho_{1/j}, \qquad \qquad m_\infty = \big[ f_\infty^\infty \big]_\alpha . \]
Then, we define, for $j, \, k \in \mathbb{N}^* \cup \{ \infty \}$, $M^k_{\alpha,j}$ as the unique solution to
\[ \overline{\T}_{\phi [f^k]} \big( M^k_{\alpha,j} \big) = F_j , \qquad \qquad M_{\alpha,j}^k(t,\cdot , \cdot) \xrightarrow[t \to+\infty]{} m_j, \]
so that $M_{\alpha , \infty}^k=M_{\alpha}^k$. Next, note that for any $j \in \mathbb{N}^*$,
\[ \mathbf{E}_1 \big[ m_j- m_\infty \big] \lesssim \mathbf{E}_1 \big[ m_\infty \big] \lesssim \E_N^n \big[ f_\infty^\infty \big], \qquad \qquad \mathbf{E}_2  \big[ m_j \big] \leq C_j  , \]
for some constant $C_j >0$. Let us also prove that, for any $j \in \mathbb{N}^*$,
\begin{equation}\label{decaysource}
 \mathbf{E}_1 \big[ F_j- F_\infty \big] (t)   \lesssim  \mathbf{E}_1 \big[  F_\infty \big] (t)  \lesssim \langle t \rangle^{-4/3}, \qquad \qquad \mathbf{E}_2 \big[ F_j \big] (t) \leq C_j \langle t \rangle^{-4/3} . 
 \end{equation}
For the first inequality, we use in particular Proposition \ref{ProboundednessGeneral} and \eqref{eq:Txv} to control $F_\infty$. For the second one, we use the compact support of $F_j$ to handle the $\langle x \rangle$-weights associated with the second-order derivatives. By the dominated convergence theorem, we then get
\begin{equation}\label{limites}
 \lim_{j \to +\infty}  \int_{t=0}^{+\infty} \mathbf{E}_1 \big[ F_j- F_\infty \big] (t) \dr t =0, \qquad \qquad \lim_{j \to +\infty} \mathbf{E}_1 \big[ m_j-m_\infty \big] =0 . 
 \end{equation}
In view of \eqref{eq:decayphik} and \eqref{decaysource}, the energy estimate of Proposition \ref{Proenergy}, together with Proposition \ref{ProboundednessGeneral}, yields
\[ \sup_{1 \leq k , \, j \leq \infty}  \mathbf{E}_1 \big[ M^k_{\alpha,j} \big](t) \leq C [N,n,p,\mathbf{\Lambda} ] , \qquad \qquad \sup_{1 \leq k \leq \infty} \sup_{j \in \mathbb{N}^*} \mathbf{E}_2 \big[ M^k_{\alpha,j} \big](t) \leq C_j [N,n,p,\mathbf{\Lambda},j ]. \]
Write next, for $t \geq 0$, that
\[ \mathbf{E}_1 \big[ M_\alpha^k - M_\alpha^\infty \big](t) \leq \mathbf{E}_1 \big[ M_\alpha^k - M_{\alpha,j}^k \big] (t) + \mathbf{E}_1 \big[ M_{\alpha,j}^k - M_{\alpha,j}^\infty \big](t) + \mathbf{E}_1 \big[ M_{\alpha,j}^\infty - M_{\alpha}^\infty \big](t) . \] 
Let $\varepsilon >0$. We shall prove that for all sufficiently large $j$, we have
\[ \forall \, k \in \mathbb{N}^*\cup \{ \infty \}, \qquad \sup_{t \geq 0} \mathbf{E}_1 \big[ M_\alpha^k - M_{\alpha,j}^k \big] (t) \leq \varepsilon. \]
For this, note that $\overline{\T}_{\phi [f^k]} ( M^k_{\alpha}-M^k_{\alpha,j})=F_\infty - F_j$. Then, the energy estimate of Proposition \ref{Proenergy} yields
\[ \mathbf{E}_1 \big[ M_\alpha^k - M_{\alpha,j}^k \big] (\tau)  \lesssim \mathbf{E}_1 \big[ m_j-m_\infty \big] +\int_{t=\tau}^{+\infty} \mathbf{E}_1 \big[F_j- F_\infty \big] (t)+ \sum_{|\kappa|=1}\big\| \big[ \overline{\T}_{\phi [f^k]}, \langle x \rangle^{|\kappa_x|} \langle v \rangle^{-|\kappa_x|} \partial_{x,v}^\kappa \big] \big(M_\alpha^k - M_{\alpha,j}^k \big) \big\|_{L^p_{x,v}}  \dr t . \]
The last term in the integral, on the right hand side, can be bounded by $\langle t \rangle^{-4/3} \mathbf{E}_1[M_\alpha^k - M_{\alpha,j}^k]$ using Proposition \ref{ProboundednessGeneral} and \eqref{eq:decayphik}. Then, it remains to use the Grönwall inequality, together with \eqref{limites}. 

Next, we observe that
\[ \mathbf{E}_1 \big[ M_{\alpha,j}^k - M_{\alpha,j}^\infty \big](\tau) \lesssim \sum_{|\kappa| \leq 1}\int_{t=\tau}^{+\infty} \Big\| \overline{\T}_{\phi[f^k]}\Big( \langle x \rangle^{|\kappa_x|} \langle v \rangle^{-|\kappa_x|} \partial_{x,v}^\kappa \big( M_{\alpha,j}^k - M_{\alpha,j}^\infty\big) \Big)(t,x,v) \Big\|_{L^p (\R^3_x \times \R^3_v)}  \dr t . \]
In view of the assumptions on $(f_\infty^k)_{k \geq 1}$, and Corollary \ref{Corunique}, we have
\[ \sup_{t \geq 0} \E_{N-1}^n \big[g^k-g^\infty \big] \lesssim \E_{N-1}^n \big[f_\infty^k - f_\infty^\infty \big], \qquad \qquad \big( \phi [f^k-f^\infty],\nabla_v \phi_\infty \big[ f_\infty^k-f^\infty_\infty \big] \big) \in \mathcal{W}^{N-1} \big( \R_+ , C \E_{N} \big[f_\infty^k - f_\infty^\infty \big] \big), \]
where $C [N,p,\mathbf{\Lambda}]>0$. We can then apply Proposition \ref{ProboundednessGeneral} and Corollary\footnote{Strictly speaking, since $M^k_{\alpha,j}$ and $M^\infty_{\alpha,j}$ are not solutions of \eqref{VP}, Corollary \ref{CorforenergypertGeneral} does not apply directly. However, as $\overline{\T}_{\phi [f^k]} (M^k_{\alpha,j})=\overline{\T}_{\phi [f^\infty]}(M^\infty_{\alpha,j})$, analogues results of Lemma \ref{LemTphi} and Corollaries \ref{CorboundTh}--\ref{CorforenergypertGeneral} in the present setting are obtained by the same arguments.} \ref{CorforenergypertGeneral}, together with \eqref{eq:decayphik} and $\mathbf{E}_2 [ M^\infty_{\alpha,j} ](t) \leq C_j$, to obtain
\[\forall \, \tau \geq 0, \qquad  \mathbf{E}_1 \big[ M_{\alpha,j}^k - M_{\alpha,j}^\infty \big](\tau) \lesssim C_j \E_{N}^n \big[f_\infty^k - f_\infty^\infty \big]. \]
Hence, by choosing first $j$ large enough, and then $k$ large enough, we get $\sup_{t \geq 0} \mathbf{E}_1 \big[ M_\alpha^k - M_\alpha^\infty \big](t) \leq 3 \varepsilon$. This concludes the proof since $M_\alpha^\infty=g_\alpha^\infty$.
\end{proof}

We now give the key ideas to show that the remainder $R^k_\alpha$ goes to $0$, as $k \to +\infty$. 

\begin{Lem}
For any $|\alpha|=N-1$, we have $\sup_{t \geq 0} \mathbf{E}_1 \big[ R^k_\alpha \big](t) \to 0$ as $k \to +\infty$.
\end{Lem}
\begin{proof}
We apply first Proposition \ref{Proenergy}. It allows to bound $\mathbf{E}_1 \big[ R^k_\alpha \big](\tau)$ by
\[  \mathbb{E}_N^n \big[f_\infty^k-f_\infty^\infty \big]+p \sum_{|\kappa| \leq 1} \int_{t = \tau}^{+\infty}  \bigg\| \bigg[ \overline{\T}_{\phi[f^k]} , \frac{\langle x \rangle^{|\kappa_x|}}{ \langle v \rangle^{|\kappa_x|}} \partial_{x,v}^\kappa \bigg]\big( R^k_\alpha \big) \bigg\|_{L^p_{x,v}} \! +\bigg\| \frac{\langle x \rangle^{|\kappa_x|}}{ \langle v \rangle^{|\kappa_x|}} \partial_{x,v}^\kappa \big[ \, \overline{\T}_{\phi [f^k]} (g^k_\alpha)- \overline{\T}_{\phi [f^\infty]} (g_\alpha^\infty) \big] \bigg\|_{L^p_{x,v}}  \dr t    . \]
By applying once again Proposition \ref{ProboundednessGeneral}, we control the first term in the integral by $\langle t \rangle^{-4/3} \mathbf{E}_1[R^k_\alpha]$. For the second one, we claim that for all $t \geq 0$,
\begin{align*}
& \bigg\| \frac{\langle x \rangle^{|\kappa_x|}}{ \langle v \rangle^{|\kappa_x|}} \partial_{x,v}^\kappa \big[ \, \overline{\T}_{\phi [f^k]} (g^k_\alpha)- \overline{\T}_{\phi [f^\infty]} (g_\alpha^\infty) \big] (t,x,v) \bigg\|_{L^p_{x,v}} \\
& \qquad \qquad \qquad \lesssim \frac{1}{\langle t \rangle^{\frac{4}{3}}} \Big( \mathbb{E}_{N-1}^n \big[ f_\infty^k-f^\infty_\infty \big]+ \mathbf{E}_1 \big[ M^k_\alpha - g^\infty_\alpha \big](t)+\mathbf{E}_1 \big[ R^k_\alpha \big](t)+\E_N \big[ g^\infty - \widetilde{g}^\infty \big](t) \Big) ,
\end{align*}
where $\widetilde{g}^\infty (t,x,v) \coloneqq g^\infty \big(t,x+\lambda \nabla_x \phi_\infty [f^k_\infty - f_\infty^\infty] \log \langle t \rangle,v \big)$. For this, the main steps are the following.
\begin{itemize}
\item We compute $ \langle x \rangle^{|\kappa_x|} \langle v \rangle^{-|\kappa_x|} \partial_{x,v}^\kappa \big[ \, \overline{\T}_{\phi [f^k]} (g^k_\alpha)- \overline{\T}_{\phi [f^\infty]} (g_\alpha^\infty) \big] $ by applying Corollary \ref{CorComm}. Then, we estimate it by performing an analysis similar to the one carried out in Lemma \ref{LemTphi} and Corollary \ref{CorboundTh}, when we controlled $\big[ \, \overline{\T}_{\phi [f^k]} (g^k-g^\infty) \big]_\alpha$. The main difference here is that $g^\infty_\alpha$ is not differentiated more than the other factors.
\item Then, in the error terms, apart from the top order derivatives of $\phi [f^k-f^\infty]$ and $g^k-g^\infty$, we control all the factors using, for $k \in \mathbb{N}^* \cup \{ \infty \}$ large enough,
\[ \sup_{t \geq 0} \mathbb{E}_N^n [g^k] \leq \mathbf{\Lambda} , \qquad \sup_{t \geq 0} \mathbb{E}_{N-1}^n [g^k-g^\infty] \lesssim \E_{N-1}^n \big[ f^k_\infty - f_\infty^\infty \big] ,  \qquad \mathcal{E}_N \big[ \nabla_v \phi_\infty [f^k_\infty -f_\infty^\infty ] \big] \lesssim \mathbb{E}_N \big[ f^k_\infty - f_\infty^\infty \big], \]
and that the force fields satisfy
 \[ \big(\phi [f^k],\nabla_v \phi_\infty[f^k_\infty] \big) \in \mathcal{W}^N \big(\R_+,\mathbf{\Lambda} \big), \qquad \big(\phi [f^k-f^\infty],\nabla_v \phi_\infty[f^k_\infty - f_\infty^\infty]\big) \in \mathcal{W}^{N-1} \big(\R_+,C\E_{N-1}[ f^k_\infty - f_\infty^\infty ]  \big), \]
 for some constant $C[N,p,\mathbf{\Lambda}]>0$. To deal with the top order derivatives, we observe that for any $|\beta|=N$, there exists $|\kappa|=1$ and $|\alpha|=N-1$ such that
 \[ \Big| \big[ g^k-g^\infty \big]_\beta \Big| \lesssim   \frac{\langle x \rangle^{|\kappa_x|}}{\langle v \rangle^{|\kappa_x|}} \Big| \partial_{x,v}^\kappa \big[ g^k-g^\infty \big]_\alpha \Big| + \Big| \big[ g^k-g^\infty \big]_\alpha \Big| \leq \frac{\langle x \rangle^{|\kappa_x|}}{\langle v \rangle^{|\kappa_x|}} \Big( \big| \partial_{x,v}^\kappa \big( M^k_\alpha- g^\infty_\alpha \big) \big|+\big| \partial_{x,v}^\kappa R^k_\alpha \big|\Big)+ \Big| \big[ g^k-g^\infty \big]_\alpha \Big| , \]
where the last term is lower order. In particular, it yields
\[ \E_N^n [g^k-g^\infty](t) \lesssim \E_{N-1}^n\big[ f^k_\infty - f_\infty^\infty \big]+\sum_{|\gamma|=N-1} \mathbf{E}_1 \big[ M^k_\gamma - M^\infty_\gamma \big](t)+\mathbf{E}_1 \big[ R^k_\gamma \big](t). \] 
For the force field, introducing $\widetilde{g}^\infty (t,x,v) \coloneqq g^\infty \big(t,x+\lambda \nabla_x \phi_\infty [f^k_\infty - f_\infty^\infty] \log \langle t \rangle,v \big)$, Corollary \ref{Rqclassteleo} yields \[ \big(\phi [f^k-f^\infty],\nabla_v \phi_\infty[f^k_\infty - f_\infty^\infty]\big) \in \mathcal{W}^N \big(\R_+,C (\E_N[f^k_\infty - f_\infty^\infty]+\E_N [g^k-\widetilde{g}^\infty]) \big), \qquad \qquad C[N,p,\mathbf{\Lambda} ] >0 .\] 
It remains to note $\E_N [g^k-\widetilde{g}^\infty] \leq \E_N [g^k-g^\infty]+\E_N [g^\infty-\widetilde{g}^\infty]$.
\end{itemize}
We conclude by applying the Grönwall inequality, together with the previous Lemma \ref{LemconvM} and
\[ \lim_{k \to + \infty} \int_{t=0}^{+\infty} \frac{1}{\langle t \rangle^{\frac{4}{3}}} \E_N [g^\infty-\widetilde{g}^\infty](t) \dr t = 0. \]
This last property follows from the dominated convergence theorem since
\begin{itemize}
\item $\E_N [g^\infty-\widetilde{g}^\infty](t) \lesssim \log^{2N+7}\langle t+1 \rangle$, by Lemma \ref{Corderivgtof} applied with $f_\circ =\widetilde{g}^\infty$, $g=g^\infty$ and $\Phi = -\nabla_v \phi_\infty [f^k_\infty-f^\infty_\infty]$,
\item $\E_N [g^\infty-\widetilde{g}^\infty](t) \to 0$ as $k \to + \infty$ for almost all $t \geq 0$. This follows from the two estimates established in the proof of Lemma \ref{CorderivgtofDIFF}, that $\mathcal{E}_N \big[ \nabla_v \phi_\infty [f^k_\infty-f^\infty_\infty] \big] \to 0$ as $k \to +\infty$, and an approximation argument to handle the top order derivatives.
\end{itemize}
\end{proof}

\appendix

\section{A top order energy estimate and applications}\label{AppA}

We begin with an energy estimate which, although not optimal with respect to the logarithmic growth, is sufficient for our purposes. In particular, it will allow us to control higher-order derivatives of spherically symmetric solutions to \eqref{VP} in the repulsive case.

\begin{Pro}\label{Proappendix1}
Let $N_0 \geq 2 p_\infty$, $\mathbf{\Lambda}, \, a \geq 0$, and $f \colon [0,T] \times \R^3_x \times \R^3_v \to \R$ be a solution to \eqref{VP} such that
\begin{itemize}
\item $\langle t+|x| \rangle^{2} \big| \nabla_x \phi [  f] \big|(t,x) \leq 2 \mathbf{\Lambda}$ for all $(t,x) \in [0,T] \times \R^3_x$,
\item $ \langle t\rangle^{2- \frac{3}{p}}\big\| \langle t+|x| \rangle^{1+|\gamma|} \nabla_x^2 \phi [ \partial_x^\gamma f ](t,x) \big\|_{L^p(\R^3_x)}\leq 2 \mathbf{\Lambda} \log^a \langle t+1 \rangle$ for all $t \in [0,T]$ and any $|\gamma| \leq N_0-1$,
\end{itemize} 
Assume $\E_{N_0}[f](0)<+\infty$ and let $f_\circ (t,x,v) \coloneqq f(t,x+tv,v)$. Then, there exists $\mathbf{C}[N_0,p,\mathbf{\Lambda}]>0$ such that,
\[ \forall \, t \in [0,T], \qquad \qquad  \E_N[f_\circ](t)   \leq \mathbf{C}\E_{N_0}[f](0) \log^{7+(a+15+2N_0) N_0} \langle t+1 \rangle . \]
Similarly, if $\E_{N_0}[f](T)<+\infty$, we have
\[ \forall \, t \in [0,T], \qquad \qquad  \E_N[f_\circ](t)   \leq \mathbf{C}\E_{N_0}[f](T) \log^{7+(a+15+2N_0) N_0} \bigg(1+\frac{\langle T \rangle}{\langle t \rangle} \bigg) . \]
\end{Pro}
\begin{proof}
Since the two cases are similar, we only treat the forward-in-time estimate. We will apply the results of Section \ref{Secprep} with $\Phi=0$, so that $g=f_\circ$. To avoid any confusion, let
\begin{equation}\label{eq:deftildeT}
 \widetilde{\T}_{\phi[f]} \coloneqq \partial_t + \lambda t \nabla_x \phi [f] (t,x+tv) \cdot  \nabla_x  -\lambda \nabla_x \phi [f] (t,x+tv) \cdot \nabla_v  , 
 \end{equation}
which coincides with $\overline{\T}_{\phi [f]}$ when $\Phi=0$. It will be convenient to introduce the function
\[ \mathbf{x}_{\mathrm{mod}}(t,x,v) \coloneqq x+\psi(t,x,v), \qquad \qquad \qquad \widetilde{\T}_{\phi[f]}(\psi)=-\widetilde{\T}_{\phi[f]}(x), \quad \psi (0,\cdot , \cdot)=0, \]
as well as the notation
\[f_\circ^\alpha (t,x,v) \coloneqq \langle \mathbf{x}_{\mathrm{mod}} \rangle^{7+|\alpha_x|} \langle v \rangle^{7+N_0-|\alpha_x|} \partial_x^{\alpha_x} \partial_v^{\alpha_v} f_\circ (t,x,v) , \qquad \qquad |\alpha|\leq N_0. \]
In particular, $\widetilde{\T}(\mathbf{x}_{\mathrm{mod}})=0$. To propagate and exploit moments in $v$ and $\mathbf{x}_{\mathrm{mod}}$, observe that
\begin{equation}\label{eq:moment}
 \big| \widetilde{\T}_{\phi[f]} \big( \langle v \rangle \big) \big| \lesssim \langle t \rangle^{-2} , \qquad \qquad  \big| \widetilde{\T}_{\phi [f]} ( x ) \big| \lesssim \langle t \rangle^{-1}.
 \end{equation}
 Duhamel's principle then allows to estimate $\psi$, which yields
 \begin{equation}\label{eq:xmod}
\langle x \rangle \lesssim \langle  \mathbf{x}_{\mathrm{mod}}  \rangle \log \langle t+1 \rangle, \qquad \qquad \langle \mathbf{x}_{\mathrm{mod}} \rangle \lesssim \langle x \rangle \log \langle t+1 \rangle . 
\end{equation}
We will also use that, for all $(t,x) \in [0,T] \times \R^3$,
\begin{equation}\label{eq:estitopLinfty}
\langle t\rangle^{2- \frac{3}{p}}\langle t+|x| \rangle^{\frac{3}{p}+|\kappa|} \big| \nabla_x \phi [ \partial_x^\kappa f] \big|(t,x) \lesssim \mathbf{\Lambda} \log^a \langle t+1 \rangle, \qquad \qquad 1 \leq |\kappa|\leq N_0-p_\infty,
\end{equation}
which follows from the assumptions and Proposition \ref{ProSob}. Let us show that, for any $|\alpha_x|+|\alpha_v|\leq N_0$,
\begin{align}
 \nonumber & \big\| f_{\circ}^\alpha (\tau,x,v) \big\|_{L^p_{x,v}} - \big\| f_\circ^\alpha (0,x,v) \big\|_{L^p_{x,v}}  \lesssim \int_{t=0}^t \big\|  \widetilde{\T}\big(  f_\circ^\alpha \big) (t,x,v) \big\|_{L^p_{x,v}} \dr t   \\
 &  \quad  \lesssim  \sum_{|\beta| \leq N_0}\int_{t=0}^\tau \frac{1}{\langle t \rangle^{\frac{3}{2}}} \big\| f_\circ^\beta (t,x,v) \big\|_{L^p_{x,v}}  \dr t  + \sum_{|\kappa| \leq N_0, \, |\kappa_v|<|\alpha_v|}\int_{t=0}^\tau \frac{\log^{a+14+|\alpha_x|+|\kappa_x|}\langle t+1 \rangle}{\langle t \rangle} \big\| f_\circ^\kappa (t,x,v) \big\|_{L^p_{x,v}} \dr t .  \label{eq:biduletorpove}
  \end{align}
The first inequality follows from Proposition \ref{Proenergy}. One cannot directly apply the commutation formula of Corollary \ref{Corenergy} since it is stated with the weight $\langle x \rangle$ rather than $\langle \mathbf{x}_{\mathrm{mod}} \rangle$. We deal with this issue using \eqref{eq:xmod}, which gives rise to the $\log^{14+|\alpha_x|+|\kappa_x|}\langle t+1 \rangle$ losses in \eqref{eq:biduletorpove}. Then, using \eqref{eq:moment}, $\widetilde{\T}_{\phi[f]}(\mathbf{x}_{\mathrm{mod}})=0$ and Corollary \ref{Corenergy}, applied with $\Phi=0$, we reduce the proof of \eqref{eq:biduletorpove} to showing that
\begin{align*}
\mathcal{I}_1 & \coloneqq \bigg\| \frac{\langle v \rangle t}{  \langle x \rangle}   t^{|\gamma|} \nabla_x \phi \big[ \partial_{x}^{\gamma} f \big]  (t,x+tv) \cdot f_{\circ}^\beta \bigg\|_{L^p(\R^3_x \times \R^3_v)} \lesssim \frac{\log^a \langle t+1 \rangle}{\langle t \rangle} \big\| f_\circ^\beta (t,x,v) \big\|_{L^p (\R^3_x \times \R^3_v)}  , \qquad \gamma + \beta_v=\alpha_v, \\
\mathcal{I}_2 & \coloneqq \Big\| \langle t+|x+tv| \rangle^{|\gamma|} \nabla_x \phi \big[ \partial_{x}^{\gamma} f \big] (t,x+tv) \cdot f_{\circ}^\beta \Big\|_{L^p(\R^3_x \times \R^3_v)} \lesssim \frac{\log^a \langle t+1 \rangle}{\langle t \rangle^2} \big\| f_\circ^\beta (t,x,v) \big\|_{L^p (\R^3_x \times \R^3_v)} , 
\end{align*}
where $ 1 \leq  |\gamma|  \leq N_0$ and $\big( |\gamma| - 1 \big)+|\beta| \leq N_0$. Observe for this that, as $t \langle v \rangle \leq t+t|v| \leq t+|x|+|x-tv|$, we have
\begin{equation}\label{eq:biduletruc}
\forall\, (t,x,v) \in \R_+ \times \R^3_x \times \R^3_v, \qquad \qquad \frac{\langle v \rangle t}{\langle x \rangle \, \langle t+| x+tv| \rangle^{1+\frac{3}{p}} \langle t \rangle^{1-\frac{3}{p}}} \lesssim \frac{1}{\langle t \rangle}.
\end{equation}
Assume first that $|\gamma| \leq N_0-p_\infty$. Then, we estimate pointwise the force field using the assumptions and \eqref{eq:estitopLinfty}. It yields, using also \eqref{eq:biduletruc} for $\mathcal{I}_1$, that
\[ \mathcal{I}_1 \lesssim \frac{\log^a \langle t+1 \rangle}{\langle t \rangle} \big\| f_\circ^{\beta} (t,x,v) \big\|_{L^p(\R^3_x \times \R^3_v)}, \qquad \qquad  \mathcal{I}_2 \lesssim  \frac{\log^a \langle t+1 \rangle}{\langle t \rangle^2} \big\| f_\circ^{\beta} (t,x,v) \big\|_{L^p(\R^3_x \times \R^3_v)} . \]
Otherwise, we have $|\gamma| \geq N_0+1-p_\infty$ and $|\beta| \leq p_\infty$. Using again \eqref{eq:biduletruc}, and the change of variables $y(x)=x+tv$, we get
\[ \mathcal{I}_1+\langle t \rangle \cdot \mathcal{I}_2 \lesssim \langle t \rangle \Big\| \langle t+|y| \rangle^{|\gamma|} \nabla_x \phi \big[ \partial_{x}^{\gamma} f \big] (t,y) \cdot f_{\circ}^\beta (t,y-tv,v) \Big\|_{L^p(\R^3_y \times \R^3_v)} \lesssim \frac{\log^a \langle t+1 \rangle}{ \langle t \rangle^{1-\frac{3}{p}}} \big\| f_{\circ}^\beta (t,y-tv,v) \big\|_{L^\infty \big(\R^3_y,L^p(\R^3_v) \big)}. \]
If $t \leq 1$, we use the Sobolev embedding for Banach-valued functions $W^{p_\infty,p}(\R^3_y, L^p(\R^3_v)) \hookrightarrow L^\infty (\R^3_y , L^p(\R^3_v))$. If $t \geq 1$, we perform the change of variables $y'(v)=y-tv$, and we use a Sobolev embedding in the velocity variable. It yields
\[  \big\| f_{\circ}^\beta (t,y-tv,v) \big\|_{L^\infty \big(\R^3_y,L^p(\R^3_v) \big)} \lesssim \frac{1}{\langle t \rangle^{\frac{3}{p}}} \sum_{|\kappa| \leq p_\infty} \big\| \partial_{x,v}^\kappa f_\circ^\beta (t,x,v) \big\|_{L^p_{x,v}}  \lesssim \frac{1}{\langle t \rangle^{\frac{3}{p}}} \sum_{ |\kappa| \leq |\beta|+p_\infty, \, \kappa_v \leq \beta_v+p_\infty} \big\| f_\circ^\kappa (t,x,v) \big\|_{L^p_{x,v}}  . \]
For $\mathcal{I}_1$, since $|\gamma| \geq N_0+1-p_\infty$, $|\gamma|+|\beta_v|=|\alpha_v|$, and $N_0 \geq 2p_\infty$, we have $|\beta_v|+p_\infty < |\alpha_v|$. This concludes the proof of \eqref{eq:biduletorpove}. Using $14+|\alpha_x|+|\kappa_x| \leq 14+2N_0$, and a triangular Grönwall inequality, we obtain
\[  \big\| \langle \mathbf{x}_{\mathrm{mod}} \rangle^{7+|\alpha_x|} \langle v \rangle^{N_0+7-|\alpha_x|} \partial_x^{\alpha_x} \partial_v^{\alpha_v} f_\circ (t,x,v) \big\|_{L^p(\R^3_x \times \R^3_v)}   \leq \mathbf{C}\E_{N_0}[f](0) \log^{(a+15+2N_0) |\alpha_v|} \langle t+1 \rangle . \]
The result then follows by bounding $\langle \mathbf{x}_{\mathrm{mod}} \rangle$ using \eqref{eq:xmod}.
\end{proof}

\subsection{Boundedness of spherically symmetric solutions for the repulsive Vlasov-Poisson system}

We assume here that $\lambda=-1$. The purpose of this section is to show the following result.

\begin{Pro}\label{Prospher}
Let $N \geq 1$, and $f_0 \in C_c^N(\R^3_x \times \R^3_v, \R_+)$ be a spherically symmetric initial distribution function. Then, the unique classical solution $f$ to \eqref{VP}, arising from the initial data $f_0$ satisfies the following properties.
\begin{enumerate}[label = (\textbf{\Roman*})]
\item \label{traou:UNAN} $f \in C^N(\R_+ \times \R^3_x \times \R^3_v)$ is global in time and there exists $K \geq 0$ such that
\[ \mathrm{supp} \, f_\circ (t,\cdot , \cdot ) \subset \big\{ (x,v) \in \R^3_x \times \R^3_v \; \big| \; |x| \leq K \log \langle t+1 \rangle, \; |v| \leq K \big\}, \qquad f_\circ (t,x,v) \coloneqq f(t,x+tv,v) . \]
\item \label{traou:DAOU} For all $(t,x) \in \R_+ \times \R^3_x$, we have $ |\nabla_x \phi [f](t,x) \big| \lesssim \langle t+|x| \rangle^{-2}$.
\item \label{traou:TRI} There exists $q \geq 0$ such that, for any $|\alpha|, \, |\gamma| \leq N$ and all $t \in \R_+$, there holds
\[  \big\| \partial_x^{\alpha_x} \partial_v^{\alpha_v} f_\circ (t,x,v) \big\|_{L^p(\R^3_x \times \R^3_v)} \lesssim \log^q \langle t+1\rangle , \qquad   \langle t \rangle^{|\gamma|+3-\frac{3}{p}}\big\| \nabla_x^2 \phi [\partial_x^\gamma f ](t,x) \big\|_{L^p(\R^3_x)} \lesssim \log^q \langle t+1 \rangle . \]
\item \label{traou:PEVAR} Let $\Phi_f (t,v) \coloneqq t^2 \nabla_x \phi [f] (t,tv)$ and $g(t,x,v) \coloneqq f(t,x+tv+\lambda \Phi_f (t,v) \log \langle t \rangle , v)$. There exists $\mathbf{\Lambda} \geq 0$ such that, for any $n \in \mathbb{Z}^2$,
\[ \sup_{t \geq 0} \E_N^n[g](t) \leq \mathbf{\Lambda}, \qquad \qquad \big( \phi [f], \Phi_f \big) \in \mathcal{W}^N \big(  \R_+, \mathbf{\Lambda} \big) . \]
\end{enumerate}
\end{Pro}

The global existence and $C^N$ regularity of $f$ follow from \cite{Pfa,LionsPerthame}. According to \cite[Lemmata~3.1 and~5.1]{PankavichSpheri}, 
\[ |\nabla_x \phi [f](t,x) \big| \lesssim \langle t \rangle^{-2}, \qquad \qquad |\rho [f]|(t,x) \lesssim \langle t \rangle^{-3}. \] 
It yields $|\T_{\phi[f]}(|v|)| \lesssim \langle t \rangle^{-2}$ and $|\T(|x-tv|)| \lesssim \langle t \rangle^{-1}$, which implies \ref{traou:UNAN}. Hence, there exists $K_0$ such that $\rho[f](t,x)=0$ for $|x| \geq K_0\langle t \rangle$. As a result, we get \ref{traou:DAOU} from Proposition \ref{Proellip}. To show \ref{traou:TRI}, we will require
\begin{equation}
\forall \, (t,x,v) \in \R_+\times \R^3_x \times \R^3_v, \qquad \qquad \big| \nabla_x f_\circ \big|(t,x,v) \lesssim 1,  \qquad \big| \nabla_v f_\circ \big|(t,x,v) \lesssim \log \langle t+1 \rangle , \label{eq:base2}
\end{equation}
which have been proved in from \cite[Lemma~8.1]{PankavichSpheri}, and in the proof of \cite[Lemma~9.3]{PankavichSpheri}. To derive $L^p$-estimates for the force field, we will use the next result.
\begin{Lem}\label{Lemdecaynablatwo}
Let $|\alpha| \leq N$. Then, for all $t \in \R_+$, 
\[ \langle t \rangle^{3\frac{p-1}{p}+|\alpha|} \big\| \nabla^2_x \phi \big[ \partial_x^\alpha f \big](t,x) \big\|_{L^p(\R^3_x)} \lesssim \big\| \langle x \rangle^4 \langle v \rangle^4 \big[ |\partial_x^\alpha f_\circ|+|\partial_v^\alpha f_\circ| \big](t,x,v) \big\|_{L^p(\R^3_x \times \R^3_v)}. \]
\end{Lem}
\begin{proof}
Note, using first an elliptic estimate from Proposition \ref{Proellip} and then the identity \eqref{eq:deffcirc}, that
 \begin{align*}
 \langle t \rangle^{|\alpha|} \big\| \nabla^2_x \phi \big[ \partial_x^\alpha f \big](t,x) \big\|_{L^p(\R^3_x)}  \lesssim  \langle t \rangle^{|\alpha|}\big\| \partial_x^\alpha \rho [f] (t,x) \big\|_{L^p( \R^3_x)} & \leq \bigg\| \int_{\R^3_v} \big[ |\partial_x^\alpha f_\circ|+|\partial_v^{\alpha} f_{\circ} | \big] (t,x-tv,v) \dr v \bigg\|_{L^p(\R^3_x)}.
  \end{align*}
It remains to apply the Hölder inequality, and to note
\[ \int_{\R^3_v} \frac{\dr v}{\langle x-tv \rangle^4} \leq \frac{2}{t^3}, \qquad \qquad \int_{\R^3_v} \frac{\dr v}{\langle v \rangle^4} \leq 2 . \]
\end{proof}
Then, in view of \eqref{eq:base2}, the support of $f_\circ (t,\cdot , \cdot)$, and the previous Lemma \ref{Lemdecaynablatwo}, 
\begin{equation}\label{eq:toiterate} \forall \, t \in \R_+, \qquad \sum_{|\gamma| \leq q} \langle t \rangle^{|\gamma|+3-\frac{3}{p}}\big\| \nabla_x^2 \phi [\partial_x^\gamma f ](t,x) \big\|_{L^p(\R^3_x)} \lesssim \log^{a_q} \langle t+1 \rangle , 
\end{equation}
holds for $q=1$, with $a_1=5+3/p$. 

Let us first show \ref{traou:TRI} for $p>3$, so that $p_\infty=1$. Then, by \ref{traou:DAOU} and \eqref{eq:toiterate}, the assumptions of Proposition \ref{Proappendix1} are satisfied for $N_0=2p_\infty$. We then obtain \ref{traou:TRI} by iterating Proposition \ref{Proappendix1} and Lemma \ref{Lemdecaynablatwo}.

We now consider the case $3/2< p \leq 3$, so that $p_\infty =2$. Using \ref{traou:TRI} for $p=7$, together with the Sobolev embedding $W^{1,7}(\R^3_x \times \R^3_v) \hookrightarrow L^\infty (\R^3_x \times \R^3_v)$, we control $\partial_{x,v}^\alpha f_\circ$ pointwise for any $|\alpha| \leq N-1$. Then, by applying Lemma \ref{Lemdecaynablatwo}, we get that the assumptions of Proposition \ref{Proappendix1} are satisfied for $N_0=N-1$, allowing to derive \ref{traou:TRI}.

We finally prove \ref{traou:PEVAR}. Applying Corollary \ref{Rqclass} with $\Phi=0$, we get $(\phi [f],\Phi_f ) \in \mathcal{W}^N(\R_+,C \E_N[f_\circ])$ for some constant $C \geq 0$. Note further that $\E_N[f_\circ](t) \lesssim \log^{q+N+7}\langle t+1 \rangle$ by support considerations and \ref{traou:TRI}. Then, the energy estimate of Proposition \ref{Proenergy}, together with Propostion \ref{ProboundednessGeneral}, yields to $\sup_{t \geq 0} \E_N^{(n_x,n_v)}[g](t) <+\infty$. We conclude by applying Corollary \ref{Rqclass} with $\Phi=\Phi_f$ and $a=q+N+7$.

\renewcommand{\refname}{References}
\bibliographystyle{abbrv}
\bibliography{biblio}

\end{document}